\documentclass{article}
\usepackage{subfiles} % ファイル分割

\usepackage[T1]{fontenc} % アクセント文字対応・PDF検索/コピペ品質向上
\usepackage{lmodern} % T1フォントのビットマップ化防止（Latin Modernベクトルフォント）

\usepackage[top=30truemm,bottom=30truemm,left=25truemm,right=25truemm]{geometry} %余白の調整

\usepackage{mathtools}
\mathtoolsset{showonlyrefs = true} % 参照した数式のみに番号付与
\usepackage{amssymb} % 記号を増やす。とりあえず読み込む。
\usepackage{physics2} % 記号の簡潔なコマンド。括弧やノルムなど。
\usephysicsmodule{ab,ab.legacy}
\usepackage{bbm} % 白抜き1
\usepackage{xcolor} % 色
\usepackage{comment} % コメントアウト機能

\usepackage{amsthm} % 定理環境
\usepackage{thmtools} % これがないとcrefで正しく引用されない（親ラベルの名前で引用される）
\usepackage[bookmarkstype=toc, colorlinks=true, pdfborder={0 0 0}, bookmarks=false, bookmarksnumbered=true, linkcolor=blue, citecolor=red,  anchorcolor=blue]{hyperref} % 引用でリンクを張る
\usepackage[capitalize,nameinlink,noabbrev,nosort]{cleveref} %引用時に自動で「定理」などを付けてくれる
\theoremstyle{plain}
\newtheorem{theorem}{Theorem}[section]
\newtheorem{lemma}[theorem]{Lemma}
\newtheorem{proposition}[theorem]{Proposition}
\newtheorem{corollary}[theorem]{Corollary}
\theoremstyle{definition} % remarkはDefinitionスタイルで
\newtheorem{remark}[theorem]{Remark}
\newtheorem{example}[theorem]{Example}

\numberwithin{equation}{section} % 式番号を「(3.5)」のように表示

\newcommand{\reals}{\mathbb{R}} % 実数
\newcommand{\naturals}{\mathbb{N}} % 自然数
\newcommand{\integers}{\mathbb{Z}} % 整数
\newcommand{\expectsymbol}{\mathbb{E}} % 期待値のE
\newcommand{\probabsymbol}{\mathbb{P}} % 確率のP
\newcommand{\indicatorsymbol}{\mathbbm{1}} % 指示関数の1
\NewDocumentCommand\expect{o m}{% % 期待値（カッコ自動調整・手動調整可能）
  {%
    \expectsymbol%
    \IfValueTF{#1}{\bab[#1]{#2}}{\bab{#2}}%
  }%
}
\ExplSyntaxOn % noneかどうかの判定のために実行
\NewDocumentCommand\condexpect{o m m}{% % 条件付き期待値（カッコ自動調整・手動調整可能）
  \expectsymbol%
  \IfValueTF{#1}{%
    \str_if_eq:nnTF{#1}{none}{%
      \bab[none]{#2 \mathrel{\vert} #3}%
    }{%
      \bab[#1]{#2 \mathrel{\csname #1\endcsname\vert} #3}%
    }%
  }{%
    \bab{#2 \mathrel{}\middle|\mathrel{} #3}%
  }%
}
\ExplSyntaxOff
\NewDocumentCommand\probab{o m}{% % 確率（カッコ自動調整・手動調整可能）
  {%
    \probabsymbol%
    \IfValueTF{#1}{\pab[#1]{#2}}{\pab{#2}}%
  }%
}
\ExplSyntaxOn % noneかどうかの判定のために実行
\NewDocumentCommand\condprobab{o m m}{% % 条件付き確率（カッコ自動調整・手動調整可能）
  \probabsymbol%
  \IfValueTF{#1}{%
    \str_if_eq:nnTF{#1}{none}{%
      \pab[none]{#2 \mathrel{\vert} #3}%
    }{%
      \pab[#1]{#2 \mathrel{\csname #1\endcsname\vert} #3}%
    }%
  }{%
    \pab{#2 \mathrel{}\middle|\mathrel{} #3}%
  }%
}
\ExplSyntaxOff
\NewDocumentCommand\indicator{o m}{% % 指示関数（カッコ自動調整・手動調整可能）
  {%
    \indicatorsymbol%
    \IfValueTF{#1}{\Bab[#1]{#2}}{\Bab{#2}}%
  }%
}
\NewDocumentCommand\pairing{o m m}{% % ペアリング（カッコ自動調整・手動調整可能）
  {%
    \IfValueTF{#1}{\aab[#1]{#2, #3}}{\aab{#2, #3}}%
  }%
}
\NewDocumentCommand\iprod{o m m}{% % 内積（inner product）（カッコ自動調整・手動調整可能）
  {%
    \IfValueTF{#1}{\aab[#1]{#2, #3}}{\aab{#2, #3}}%
  }%
}

\newcommand{\cechcpx}[2]{\mathcal{C}(#1, #2)} % \v{C}ech複体
\newcommand{\rgg}[2]{G(#1, #2)} % 幾何グラフ

\newcommand{\map}[3]{#1 \colon #2 \to #3} % 写像
\newcommand{\signedmeas}{\mathcal{M}} % 符号付き測度の空間
\newcommand{\lebmeas}{\mathsf{Leb}} % ルベーグ測度
\newcommand{\diam}{\mathsf{diam}} % 直径
\newcommand{\lorder}{O} % ランダウ記号（大きいO）
\newcommand{\centered}[1]{\overline{#1}} % 中心化（X - E[X]）
\newcommand{\card}[1]{\abs{#1}} % 濃度（要素数）
\newcommand{\supnorm}[1]{\norm{#1}_{\infty}} % supノルム
\newcommand{\suprdnorm}[1]{\norm{#1}_{\reals^m, \infty}} % R^mノルムのsupノルム

\newcommand{\tvnorm}[1]{\norm{#1}_{\mathsf{TV}}} % 全変動ノルム
\newcommand{\sigmatau}{\mathcal{G}_{\tau}} % 有界可測関数とのペアリングを可測にする最小のσ集合族
\newcommand{\sigmaweak}{\mathcal{G}_{\mathsf{weak}}} % 有界可測関数とのペアリングを可測にするσ集合族
\DeclareMathOperator{\interior}{int} % 内部
\DeclareMathOperator{\closure}{cl} % 閉包
\newcommand{\interiortau}[1]{\interior_{\tau}(#1)} % τ-topoogyでの内部
\newcommand{\closuretau}[1]{\closure_{\tau}(#1)} % τ-topoogyでの閉包
\newcommand{\interiorweak}[1]{\interior_{\mathsf{weak}}(#1)} % weak-topoogyでの内部
\newcommand{\closureweak}[1]{\closure_{\mathsf{weak}}(#1)} % weak-topoogyでの閉包
\newcommand{\interiorrd}[1]{{#1}^{\circ}} % R^dでの内部
\newcommand{\closurerd}[1]{\overline{#1}} % R^dでの閉包
\newcommand{\ball}[2]{B(#1, #2)} % 球
\newcommand{\lex}{l} % 辞書式順序
\newcommand{\lexineq}{\preceq} % 辞書式順序
\newcommand{\roundup}[1]{\lceil #1 \rceil} % 切り上げ
\newcommand{\Roundup}[1]{\left\lceil #1 \right\rceil} % 切り上げ（括弧自動調整）
\newcommand{\rounddown}[1]{\lfloor #1 \rfloor} % 切り捨て
\newcommand{\fqexp}{\gamma} % テスト関数列{f_q}_{q >= 1}の勾配増大度の指数

\newcommand{\bfa}{\mathbf{a}}
\newcommand{\bfe}{\mathbf{e}}
\newcommand{\bfx}{\mathbf{x}}
\newcommand{\bfy}{\mathbf{y}}

\newcommand{\bfu}{\mathbf{u}}
\newcommand{\bfv}{\mathbf{v}}
\newcommand{\bft}{\mathbf{t}}
\newcommand{\bfzero}{\mathbf{0}}
\newcommand{\calB}{\mathcal{B}}
\newcommand{\calI}{\mathcal{I}}
\newcommand{\calF}{\mathcal{F}}
\newcommand{\calE}{\mathcal{E}}
\newcommand{\calG}{\mathcal{G}}

\newcommand{\calX}{\mathcal{X}}
\newcommand{\calY}{\mathcal{Y}}
\newcommand{\calZ}{\mathcal{Z}}
\newcommand{\calT}{\mathcal{T}}
\newcommand{\calR}{\mathcal{R}}
\newcommand{\calU}{\mathcal{U}}
\newcommand{\calK}{\mathcal{K}}

\newcommand{\convint}[1]{\mathsf{conv}^\circ(#1)} % 凸包の内部

\newcommand{\ppp}{\mathcal{P}} % Poisson点過程
\newcommand{\bpp}{\mathcal{B}} % 二項点過程
\newcommand{\pppentire}{\Pi} % Poisson点過程（Rd全体）
\newcommand{\nrnd}{n r_n^d}
\newcommand{\nkrndkmo}[1][k]{n^{#1} r_n^{d({#1}-1)}}
\newcommand{\ldpspeed}[1][n]{\rho_{#1}}
\newcommand{\mdpfactor}[1][n]{b_{#1}}
\newcommand{\lilsubseq}[2]{n(#1, #2)} % LILで使う部分列（n_{s, l} = \rho_n = s^lとなるn）
\newcommand{\cube}[2]{[#1, #2]^d} % 直方体
\newcommand{\unitcube}{\cube{0}{1}} % 単位立方体
\newcommand{\abscont}{\ll} % 絶対連続
\newcommand{\diracdelta}[1]{\delta_{#1}} % δ測度
\NewDocumentCommand\isolated{o m m m}{% % 他の点と繋がっていない条件
  {%
    \IfValueTF{#1}{c_{#1}}{c} (#2, #3; #4)
  }%
}
\newcommand{\bddcont}{C_{\mathsf{b}}} % 有界連続関数の空間
\newcommand{\bddmble}{M_{\mathsf{b}}} % 有界可測関数の空間
\newcommand{\distancemeas}{d_{\signedmeas(E)}} % 符号付き測度の空間に入れる距離
\newcommand{\distancevec}{d_{\reals^m}} % R^mに入れる（集合との）距離
\newcommand{\accptmeas}{\calK} % LILで使う集積点全体（符号付き測度）
\newcommand{\accptvec}{K} % LILで使う集積点全体（R^m）
\newcommand{\tailfunc}{\psi} % tailに出てくる関数 1 + x - √(1+2x)
\newcommand{\zeroed}[1]{\widetilde{#1}} % f(0) = 0とする操作

\newcommand{\orlicznorm}[1]{\norm{#1}_{\psi_1}} % Orliczノルム

\let\emptyset\varnothing % 空集合
\let\epsilon\varepsilon % ε
\let\liminf\varliminf % liminf
\let\limsup\varlimsup % limsup
\title{Moderate deviations and laws of the iterated logarithm for geometric functionals in the sparse regime}
\author{
  Yudai Sakihara\thanks{Joint Graduate School of Mathematics for Innovation, Kyushu University, 744, Motooka, Nishi-ku, Fukuoka, 819-0395, Japan.
   \textit{E-mail address}: \texttt{sakihara.yudai.145@s.kyushu-u.ac.jp}}
   \and
  Kenkichi Tsunoda\thanks{Faculty of Mathematics, Kyushu University, 744, Motooka, Nishi-ku, Fukuoka, 819-0395, Japan.\\
   \textit{E-mail address}: \texttt{tsunoda@math.kyushu-u.ac.jp}}
}

\date{\empty}

\begin{document}

\maketitle

\begin{abstract}
We prove moderate deviation principles and laws of the iterated logarithm in the sparse regime
for geometric and topological functionals associated with $k$-point connected components.
These limit theorems are established at both the measure-valued and vector-valued levels, for Poisson and binomial point processes.
As applications, we derive corresponding results for component counts in random geometric graphs and \v{C}ech complexes and for counts of Morse critical points.
\end{abstract}

% Keywords command
\providecommand{\keywords}[1]
{
  \small	
  \textbf{Keywords and Phrases:} #1
}
%MSC Classification command
\providecommand{\subjclass}[1]
{
  \small	
  \textbf{Mathematics Subject Classification 2020:} #1
}



% メモ（確認用、後で削除）：
% \begin{itemize}
%   \item 60F10: Large deviations
%   \item 60D05: Geometric probability and stochastic geometry
%   \item 60G55: Point processes (e.g., Poisson, Cox, Hawkes processes)
%   \item 55N31: Persistent homology and applications, topological data analysis
% \end{itemize}

\section{Introduction}

In recent years, topological data analysis has attracted widespread attention as a method for extracting topological and geometric features underlying complex data structures such as point clouds.
In particular, persistent homology tracks the birth and death of homological features across a filtration.
Since the pioneering works that established its computational foundations \cite{elz2002tda,zg2005computingph}, this tool has led to a wide range of applications.
Alongside these developments, the field of random topology has grown rapidly, providing probabilistic models and tools for the study of random topological structures.

Random geometric complexes and their generalizations form an important class of objects in random topology.
Historically, random geometric graphs, where edges are constructed based on the distances between points of a spatial point process,
have been widely studied in the context of stochastic geometry \cite{penrose2003rgg}.
From this perspective, random geometric complexes are higher-dimensional generalizations of these random geometric graphs,
and simultaneously provide probabilistic models for complexes constructed from point data in topological data analysis.
As seen in \cite{ya2015rgcstationary,ysa2017rgcthermo},
there is a substantial body of literature on laws of large numbers (LLNs) and
central limit theorems (CLTs) for topological and geometric functionals of random geometric complexes, including Betti numbers.
As further developments, the LLNs for Betti numbers have been extended to manifolds \cite{gtt2019sllnmanifold},
a CLT for persistent Betti numbers has been established \cite{kp2025cltbetti},
and functional limit theorems have been derived for stochastic processes of subgraph counts, Betti numbers, and the Euler characteristic,
indexed by the radius parameter \cite{to2020plimbetti,to2021flimeuler,owada2017fcltsubgraph}.
LLNs and other convergence results for persistence diagrams have also been established \cite{hst2018pd,owada2022pd}.

While the LLN and CLT describe typical behaviors that occur with high probability,
the large deviation principle (LDP) describes the probabilities of rare deviations from the average behavior.
Here we review the definition of LDP \cite{dz2009ldp}.
A sequence of probability measures $\{\mu_n\}_{n \geq 1}$ on a topological space $\calX$ equipped with a $\sigma$-field $\calB$ is said to satisfy the LDP with speed $a_n\to\infty$ and a lower semicontinuous function $\map{I}{\calX}{[0, \infty]}$ (called rate function)
if, for any $A \in \calB$,
\begin{align}
  - \inf_{x \in \interior (A)} I(x)
  \leq \liminf_{n \to \infty} \frac{1}{a_n} \log \mu_n(A)
  \leq \limsup_{n \to \infty} \frac{1}{a_n} \log \mu_n(A)
  \leq - \inf_{x \in \closure (A)} I(x)
\end{align}
holds, where $\interior (A)$ and $\closure (A)$ denote the interior and the closure of $A$, respectively.
Furthermore, the moderate deviation principle (MDP) interpolates between the central-limit and large-deviation scales.
More precisely, an MDP is an LDP for suitably centered and rescaled random objects, with the fluctuation scale lying between these scales.
For the precise statement of the MDP, see \cref{sec:result.mdp}.

In contrast to the extensive literature on LLNs and CLTs, relatively few results are available for LDPs in random topology.
For example, in the dense regime, even for a basic functional such as the number of edges in a random geometric graph,
determining the speed and rate function of the large deviations is a highly nontrivial problem \cite{cm2020denseedgeldp}.
In the critical regime, LDPs have been studied for functionals satisfying near additivity \cite{st2001ldp},
as well as for functionals satisfying weak dependence conditions characterized by stabilization radii \cite{sy2005ldp}.
The LDP for persistence diagrams has also been investigated in random cubical complexes, which are lattice-based analogs of random geometric complexes \cite{hiraoka2024large}.
Against this background, Hirsch and Owada \cite{ho2023ldp} recently focused on the sparse regime,
and established LDPs for geometric functionals such as connected component counts.
To briefly recall their setting, let $\ppp_n$ be a Poisson point process of intensity $n$ on $[0, 1]^d$,
let $r_n$ be a sequence of connectivity radii, and for a function $\map{H}{(\reals^d)^k}{[0, \infty)}$, define
\begin{align}
  G_n(\calY, \ppp_n; t) \coloneqq H(r_n^{-1} \calY) \indicator[none]{\norm[none]{y - y'}_{\reals^d} \geq r_n t, \, \text{ for all $y \in \calY$ and $y' \in \ppp_n \setminus \calY$}}.
\end{align}
They proved the LDP for real-valued statistics of the form
\begin{align}
  \sum_{\calY \subset \ppp_n, \card{\calY}=k} G_n(\calY, \ppp_n; t). \label{eq:intro.statistics.scalar}
\end{align}
Furthermore, their results include the corresponding Sanov-type (measure-valued) LDP,
that is, the LDP for the empirical measure
\begin{align}
  \sum_{\calY \subset \ppp_n, \card{\calY}=k} \diracdelta{G_n(\calY, \ppp_n; t)}, \label{eq:intro.statistics.measure}
\end{align}
where $\diracdelta{x}$ is the Dirac delta measure at $x$,
as well as analogous results for binomial point processes.
Various functionals in stochastic geometry can be treated as such statistics \cite{byy2019fastdecay,bsy2022poissonpalm,otto2025poissonapprox,lr2016poissonustat}.
As applications, they also obtained LDPs for persistent Betti numbers and for counts of Morse critical points, the latter being closely related to homological changes in random geometric complexes
\cite{ba2014distfunc,bm2015topolpgymanifold}, in specific dimensions.
More recently, Hirsch and Willhalm \cite{hw2024lowerldp}
investigated lower large deviations for a broader class of geometric functionals,
 extending the fixed-size component statistics considered in \cite{ho2023ldp}
 to sums over connected components of size at least $k$.

Motivated by the result of Hirsch and Owada \cite{ho2023ldp},
we study MDPs and LILs for the same class of geometric statistics.
As our main results, we establish the MDPs and LILs corresponding to the LDPs of Hirsch and Owada \cite{ho2023ldp},
both at the real-valued level \eqref{eq:intro.statistics.scalar} and at the measure-valued level \eqref{eq:intro.statistics.measure},
and for both Poisson and binomial point processes.
The resulting MDPs have explicit quadratic rate functions.
The measure-level LIL is stated as a Strassen-type LIL (e.g., \cite{kuelbs1974strassenlil}),
where the set of all accumulation points of the suitably scaled and centered empirical measure \eqref{eq:intro.statistics.measure} coincides with a deterministic compact set determined by the rate function (\cref{thm:lil.measure,cor:lil.scalar}).
As applications, we derive corresponding MDPs and LILs for component counts in random geometric graphs and complexes and for counts of Morse critical points.

A main difficulty is that the isolation condition depends on points outside the selected $k$-tuple, so that the statistics are not standard Poisson $U$-statistics.
For the proof of the MDP, we use the Dawson--G\"{a}rtner theorem (the projective-limit method) and the G\"{a}rtner--Ellis theorem.
Furthermore, to estimate the tail probabilities of the centered differences between point process integrals,
which are necessary for establishing exponential equivalence (\cref{prop:exp.eq.eta.xi})
and the cutoff approximations in the contraction principle (\cref{lem:contraction_cutoff}),
we decompose these differences in a way that preserves centering and
represent them as spatial martingale difference sequences as used in \cite{py2001clt,krebs2021lilsip}.
In the proof of the LIL, a major challenge lies in controlling the errors between subsequences (e.g., \cref{lemma:lil.upper.z1})
when showing that the accumulation points are contained in the level sets of the rate function (\eqref{eq:lil.upper.meas} and \eqref{eq:lil.upper.vec}).
To bound the tail probabilities of the suprema of the differences between point process integrals that arise in this context,
we combine Bousquet's concentration inequality with chaining arguments.
In particular, controlling the errors arising from the variation in $n r_n^d$ (\cref{lemma:lil.upper.z4,lemma:lil.upper.z4.ei}) is a non-trivial technical step
that is not required for the critical regime.

We mention several previous studies related to MDPs and LILs.
In the critical regime, Baryshnikov et al.\ \cite{besy2008mdp} proved a measure-level MDP and LIL for geometric functionals satisfying exponential stabilization,
such as random sequential packing, spatial birth-growth models, and $k$-nearest neighbor random graphs.
Eichelsbacher et al.\ \cite{ers2015mdp} relaxed the moment conditions on the functionals and the tail conditions on the stabilization radius compared to Baryshnikov et al.\ \cite{besy2008mdp},
and proved a measure-level MDP.
Their study imposes restrictions on the scale of the MDP.
Krebs \cite{krebs2021lilsip} proved LILs for geometric and topological functionals satisfying exponential stabilization and polynomially bounded moment conditions in the critical regime.
Krebs establishes real-valued LILs
and imposes a local growth condition on the functionals (see Condition (1.3) in their paper),
whereas the functionals \eqref{eq:intro.statistics.scalar} treated in our work do not necessarily satisfy such a condition.
Furthermore, in the context of Poisson $U$-statistics, Schulte and Th\"{a}le \cite{st2024mdp} proved MDPs for Poisson $U$-statistics,
which include subgraph counts (not restricted to connected components) as a special case.
Although the results of this study cover the sparse, critical, and dense regimes,
their framework deals with functionals (Poisson $U$-statistics) determined solely by combinations of $k$ points,
unlike the functional \eqref{eq:intro.statistics.scalar}, which imposes the condition of being isolated from other points,
and their results are formulated at the real-valued level.
Adamczak and Kutek \cite{ak2026explil} proved exponential tail probability inequalities and LILs for Poisson $U$-statistics.
Their LILs deal with a setting 
where the connectivity radius $r_n$ is fixed and only the intensity $n$ of the Poisson point process increases.
Moreover, their LILs are formulated at the real-valued level.
To the best of our knowledge, these are the first measure-valued and vector-valued MDPs and LILs for sparse isolated-configuration statistics of this type.

The rest of this paper is structured as follows.
In \cref{sec:model.results}, we rigorously formulate the main theorems on the MDPs and LILs for the point processes and functionals treated in this study,
and present specific applications, such as component counts and the number of Morse critical points.
In \cref{sec:pf.mdp}, we prove the MDPs, while the LILs are proved in \cref{sec:lil.upper}.
The technical lemmas, concentration inequalities, and other tools used in the proof of the LILs in \cref{sec:lil.upper} are collected in \cref{sec:technical.tools}.
Finally, the proofs of the MDPs and LILs for Morse critical points are provided in \cref{sec:morse.proof}.
% the Appendix.

\section{Model and main results}
\label{sec:model.results}

\subsection{Model}
\label{sec:model}
Let $\ppp_n $ be a Poisson point process on $\unitcube$ with intensity $n$ ($d \geq 2$),
and $\bpp_n $ a binomial point process consisting of $n$ independent and identically distributed uniform random vectors in $\unitcube$.
Let $\{r_n\}_{n \geq 1} $ be a sequence of positive numbers decaying to $0$ as $n \to \infty$,
and let $k \geq 2 $ be an integer, which will be fixed throughout the rest of this paper.
We investigate the \textit{sparse regime},
which is characterized by the following conditions:
\begin{align}
  \ldpspeed \coloneqq \nkrndkmo \to \infty, \quad \nrnd \to 0 \quad \text{as } n\to \infty. \label{eq:rn.sparse}
\end{align}

We consider a nonnegative measurable function $\map{H \coloneqq (h^{(1)}, \dots, h^{(m)})}{(\reals^d)^{k}}{[0, \infty)^m} $
satisfying the following conditions:
\begin{description}
  \item[{\rm (H1)}] $H$ is symmetric with respect to permutations of variables in $\reals^{d}$.
  \item[{\rm (H2)}] $H$ is translation invariant:
    \begin{align}
      H(x_1, \dots, x_k)=H(x_1 + y, \dots, x_k + y), \quad x_{i}, y \in \reals^d .
    \end{align}
  \item[{\rm (H3)}] $H$ is locally determined:
    \begin{align}
      H(x_1, \dots, x_k)=\bfzero_{m} \quad \text { whenever } \diam(x_1, \dots, x_k) \coloneqq \max_{1 \leq i, j \leq k} \norm{x_i - x_j}_{\reals^d} > L,
    \end{align}
    where $L>0$ is a constant, $\bfzero_{m}=(0, \dots, 0) \in \reals^m$,
    and $\norm{\cdot}_{\reals^d} $ denotes the Euclidean norm on $\reals^d$.
  \item[{\rm (H4)}] For every $\bfa= (a_1, \dots, a_m) \in \reals^m$,
    \begin{align}
      \int_{(\reals^{d})^{k-1}} e^{\iprod{\bfa}{H(\bfzero_d, \bfy)}}\indicator[none]{H(\bfzero_d, \bfy) \neq \bfzero_m} \,d \bfy < \infty,
    \end{align}
  where $\bfy = (y_1, \dots,  y_{k-1}) \in (\reals^d )^{k-1}$ and $\iprod{\cdot}{\cdot}$ denotes the Euclidean inner product.
  \item[{\rm (H5)}] For every $\bfa \in \reals^m \setminus \Bab{\bfzero_m}$,
  \begin{align}
    \int_{(\reals^d)^{k-1}} \abs{\iprod{\bfa}{H(\bfzero_d, \bfy)}} \,d\bfy > 0.
  \end{align}
\end{description}
Since condition (H3) remains valid even if $L$ is chosen arbitrarily large,
without loss of generality,
we may assume that $L$ is larger than the parameters $t, t_1, \dots, t_m > 0$ introduced later.
Throughout this paper, we assume (H1)--(H5).

For a $k$-point subset $\calY \subset \reals^d$, 
a finite subset $\calX \subset \reals^d$ and $t \in (0, \infty)$,
define
\begin{align}
  \isolated{\calY}{\calX}{t} = \isolated{\calY}{\calX \setminus \calY}{t}
  \coloneqq \indicator[none]{\norm[none]{y - y'}_{\reals^d} \geq t, \, \text{for all } y \in \calY \text{ and } y' \in \calX \setminus \calY}.
\end{align}
Moreover, for $\bft = (t_1, \dots, t_m) \in (0, \infty)^m$, we define
\begin{align}
  \isolated{\calY}{\calX}{\bft} \coloneqq (\isolated{\calY}{\calX}{t_i})_{i=1, \dots, m}. \label{eq:isolated.def}
\end{align}
We also define
\begin{align}
  G(\calY, \calX; \bft) \coloneqq  \isolated{\calY}{\calX}{\bft} \odot H(\calY) = (\isolated{\calY}{\calX}{t_i} h^{(i)}(\calY))_{i=1, \dots, m},
\end{align}
where $\odot$ denotes the Hadamard product.
We define a scaled version of $H$ by $H_n(\calX) \coloneqq H(r_n^{-1}\calX) $,
as well as
\begin{align}
  \isolated[n]{\calY}{\calX}{\bft} \coloneqq \isolated{r_n^{-1}\calY}{r_n^{-1}\calX}{\bft}
  = \isolated{\calY}{\calX}{r_n \bft},
  \quad
  G_n(\calY, \calX; \bft) \coloneqq \isolated[n]{\calY}{\calX}{\bft} \odot H_n(\calY).
\end{align}
For a finite $\calX \subset \reals^d $, we define
\begin{align}
  \kappa_n(\calX) \coloneqq \sum_{\calY \subset \calX, \card{\calY}=k} \indicator[none]{G_n(\calY,\calX;\bft) \neq \bfzero_{m}} \diracdelta{G_n(\calY,\calX;\bft)} \label{eq:def.kappa}
\end{align}
and 
\begin{align}
  T_n(\calX) \coloneqq \sum_{\calY \subset \calX, \card{\calY}=k} G_n(\calY,\calX;\bft). \label{eq:def.t}
\end{align}

Let $\signedmeas(E)$ be the space of finite signed measures on $E \coloneqq [0, \infty)^m$.
For a random variable $X$ (either $\reals^m$- or $\signedmeas(E)$-valued), we denote its centered version by $\centered{X} \coloneqq X - \expect[none]{X}$. 
Here, for an $\signedmeas(E)$-valued random variable $\nu$, its expectation $\expect[none]{\nu}$ is the measure defined by $\expect[none]{\nu}(A) \coloneqq \expect[none]{\nu(A)}$ for every Borel set $A \subset E$.
Let $\bddmble(E)$ be the set of all bounded measurable functions on $E$.
We equip $\signedmeas(E)$ with the $\tau$-topology, which is the weakest topology making the map $\nu \mapsto \pairing{f}{\nu} \coloneqq \int_{E} f d\nu$ continuous for every $f \in \bddmble(E)$.
Let $\sigmatau$ be the $\sigma$-field on $\signedmeas(E)$ generated by the maps $\nu \mapsto \pairing{f}{\nu}$ for all $f \in \bddmble(E)$.
Note that $\sigmatau$ does not coincide with the Borel $\sigma$-field induced by the $\tau$-topology.

Fix an arbitrary positive sequence $\Bab{\mdpfactor}_{n \geq 1}$ satisfying
\begin{align}
  \lim_{n \to \infty} \frac{\ldpspeed}{\mdpfactor} = \infty
  \quad \text{and} \quad \lim_{n \to \infty} \frac{\ldpspeed}{\mdpfactor^2} = 0. \label{eq:mdpfactor.speed}
\end{align}
Our aim is to establish the moderate deviation principle (MDP) and the law of the iterated logarithm (LIL) for the $\signedmeas(E)$-valued sequences $\{ \mdpfactor^{-1}\centered{\kappa_n(\ppp_n)} \}_{n \geq 1} $ and $\{\mdpfactor^{-1}\centered{\kappa_n(\bpp_n)} \}_{n \geq 1} $,
as well as the $\reals^m$-valued sequences $\{\mdpfactor^{-1}\centered{T_n(\ppp_n)} \}_{n \geq 1} $ and $\{\mdpfactor^{-1}\centered{T_n(\bpp_n)} \}_{n \geq 1} $.
We remark that, in contrast to \cite{ho2023ldp}, we do not remove $\bfzero_m$ from the domain $E$ of these measures. Instead, we explicitly exclude $\diracdelta{\bfzero_m}$ from the definition of these measures in \eqref{eq:def.kappa}.

Let $\tau$ be the measure on $E$ defined by
\begin{align}
  \tau(A) \coloneqq \frac{1}{k!} \lebmeas \ab\big(\{\bfy \in (\reals^{d})^{k-1} : H(\bfzero_d, \bfy) \in A \setminus \{\bfzero_m\}\}), \quad A \subset E,
\end{align}
where $A$ is a Borel set and $\lebmeas$ denotes the Lebesgue measure.
We define the rate function $\map{\Lambda}{\signedmeas(E)}{[0,\infty]}$ by
\begin{align}
  \Lambda(\nu) \coloneqq
  \begin{cases}
    \displaystyle
    \frac{1}{2} \int_{E} \ab\Big(\frac{d \nu}{d \tau})^2 \,d \tau & \text{if } \nu \abscont \tau, \\
    \infty & \text{otherwise},
  \end{cases}
\end{align}
where $\nu \in \signedmeas(E)$ and $\nu \abscont \tau$ indicates that $\nu$ is absolutely continuous with respect to $\tau$.

\subsection{Main result: moderate deviation principles}
\label{sec:result.mdp}
First, under the definitions and notation introduced above, we state our main result regarding the measure-level MDP.

\begin{theorem}[Measure-level MDP]
\label{thm:mdp.measure}
Let $\calX_n$ be either the Poisson point process $\ppp_n$ or the binomial point process $\bpp_n$.
Assume that $H$ satisfies \textup{(H1)--(H5)}.
Then, the sequence $\{\mdpfactor^{-1} \centered{\kappa_n(\calX_n)} \}_{n \geq 1}$ satisfies an MDP in $\signedmeas(E)$ endowed with the $\tau$-topology and the $\sigma$-field $\sigmatau$, with speed $\mdpfactor^2 / \ldpspeed$ and the good rate function $\Lambda$.
That is, for every $A \in \sigmatau$,
\begin{align}
  - \inf \Bab{\Lambda(\nu) : \nu \in \interiortau{A}}
  &\leq \liminf_{n \to \infty} \frac{\ldpspeed}{\mdpfactor^2} \log \probab[Big]{\frac{1}{\mdpfactor} \centered{\kappa_n(\calX_n)} \in A} \\
  &\leq \limsup_{n \to \infty} \frac{\ldpspeed}{\mdpfactor^2} \log \probab[Big]{\frac{1}{\mdpfactor} \centered{\kappa_n(\calX_n)} \in A}
  \leq - \inf \Bab{\Lambda(\nu) : \nu \in \closuretau{A}},
\end{align}
where $\interiortau{A}$ and $\closuretau{A}$ are the interior and the closure of $A$ in the $\tau$-topology, respectively.
\end{theorem}

\begin{remark}
Since the $\tau$-topology is finer than the standard weak topology on $\signedmeas(E)$,
and the $\sigma$-field $\sigmaweak$ generated by the maps $\{\nu \mapsto \pairing{f}{\nu} : f \in \bddcont(E) \} $ coincides with the Borel $\sigma$-field of the weak topology,
\cref{thm:mdp.measure} implies that the LDP also holds with respect to the weak topology and its Borel $\sigma$-field. 
That is, for every Borel set $A \subset \signedmeas(E)$ in the weak topology, we have
\begin{align}
  -\inf \Bab{\Lambda(\nu) : \nu \in \interiorweak{A}}
  &\leq \liminf_{n \to \infty} \frac{\ldpspeed}{\mdpfactor^2} \log \probab[Big]{\frac{1}{\mdpfactor} \centered{\kappa_n(\calX_n)} \in A} \\
  &\leq \limsup_{n \to \infty} \frac{\ldpspeed}{\mdpfactor^2} \log \probab[Big]{\frac{1}{\mdpfactor}\centered{\kappa_n(\calX_n)} \in A}
  \leq - \inf \Bab{\Lambda(\nu) : \nu \in \closureweak{A}},
\end{align}
where $\interiorweak{A}$ and $\closureweak{A}$ are the interior and the closure of $A$ in the weak topology, respectively.
\end{remark}

\begin{remark}
While the LDPs established in \cite{ho2023ldp} are formulated with respect to the weak topology,
they can be extended to the $\tau$-topology and its associated $\sigma$-field $\sigmatau$ by employing our method in a manner similar to \cite{acosta1994psldp}.
\end{remark}

As an immediate consequence of our measure-level result (\cref{thm:mdp.measure}), we can derive the vector-level MDP for $T_n$. 
To formulate this result, let $\Gamma$ be the $m \times m$ matrix defined by
\begin{align}
  \Gamma \coloneqq \ab\bigg( \frac{1}{k!} \int_{(\reals^d)^{k-1}} h^{(i)}(\bfzero_d, \bfy) h^{(j)}(\bfzero_d, \bfy) \,d\bfy)_{1 \leq i,j \leq m}.
\end{align}
We define the rate function $\map{I}{\reals^m}{[0,\infty]}$ by
\begin{align}
  I(\bfu) \coloneqq \frac{1}{2} \pairing{\bfu}{\Gamma^{-1} \bfu}, \quad \bfu = (u_1, \dots, u_m) \in \reals^m. \label{eq:ratefunc.vector}
\end{align}
Note that $\Gamma$ is invertible under condition (H5).

As a corollary of \cref{thm:mdp.measure}, we obtain the following vector-level MDP.

\begin{corollary}[Vector-level MDP]
\label{cor:mdp.vector}
Let $\calX_n$ be either the Poisson point process $\ppp_n$ or the binomial point process $\bpp_n$.
Assume that $H$ satisfies \textup{(H1)--(H5)}.
Then, the sequence $\{\mdpfactor^{-1} \centered{T_n(\calX_n)} \}_{n \geq 1} $ satisfies an MDP in $\reals^m$ with speed $\mdpfactor^2 / \ldpspeed $ and the good rate function $I$.
That is, for every Borel set $A \subset \reals^m$,
\begin{align}
  - \inf_{\bfu \in \interiorrd{A}} I(\bfu)
  \leq \liminf_{n \to \infty} \frac{\ldpspeed}{\mdpfactor^2} \log \probab[Big]{\frac{1}{\mdpfactor} \centered{T_n(\calX_n)} \in A} 
  \leq \limsup_{n \to \infty} \frac{\ldpspeed}{\mdpfactor^2} \log \probab[Big]{\frac{1}{\mdpfactor} \centered{T_n(\calX_n)} \in A}
  \leq - \inf_{\bfu \in \closurerd{A}} I(\bfu),
\end{align}
where $\interiorrd{A}$ and $\closurerd{A}$ are the interior and the closure of $A$ in $\reals^m$, respectively.
\end{corollary}

%%%%%%%%%%%%%%%%%%%%% LIL
\subsection{Main result: laws of the iterated logarithm}
Next, we introduce the LIL.
For the LIL, we additionally assume the following condition on $H$:
\begin{description}
  \item[{\rm (H6)}] There exists a constant $C > 0$ such that
  for all $a < 1$ sufficiently close to $1$ and every $i=1, \dots, m$,
  \begin{align}
  \begin{split}
    &\lebmeas \ab\bigg(\Bab[bigg]{\bfy \in (\reals^{d})^{k-1} : \begin{aligned}
    &(h^{(i)}(\bfzero_d, \bfy) = 0 \text{ and } h^{(i)}(\bfzero_d, a\bfy) \neq 0) \\
    &\text{or } (h^{(i)}(\bfzero_d, \bfy) \neq 0 \text{ and } h^{(i)}(\bfzero_d, a\bfy) = 0)
    \end{aligned}} )
    \leq C (1 - a),\label{eq:h.condition.h6.vanisleb}
  \end{split}
  \end{align}
  and for all $p \geq 2$,
  \begin{align}
    \int_{(\reals^d)^{k-1}} \abs[Big]{h^{(i)}(\bfzero_d, \bfy) - h^{(i)}(\bfzero_d, a \bfy)}^p \,d\bfy
    \leq C^p (1 - a). \label{eq:h.condition.h6}
  \end{align}
\end{description}
Unlike (H1)--(H5), condition (H6) will be assumed only when needed.

In addition, we consider the coupling between $\ppp_n$ and $\bpp_n$ introduced in \cite{krebs2021lilsip}.
Let $\pppentire$ be a homogeneous Poisson point process on $\reals^d$ with unit intensity, let $W_n \coloneqq [0, n^{1/d}]^d$,
and let $X_{n, 1}', X_{n, 2}', \dots$ be independent and identically distributed random vectors uniformly distributed on $\unitcube$.
We define the Poisson point process on $\unitcube$ by
\begin{align}
  \ppp_n \coloneqq n^{-1/d} (\pppentire \cap W_n) = \Bab{X_1, \dots, X_{N_n}}, \label{eq:lil.ppp.coupling}
\end{align}
where $N_n \coloneqq \ppp_n(\unitcube) = \pppentire(W_n)$,
and we define the binomial point process on $\unitcube$ by
\begin{align}
  \bpp_n \coloneqq \Bab{X_i : i \in \Bab{1,\dots,n \land N_n}} \cup \Bab{X_{n, i}' : i \in \Bab{1, \dots, (n - N_n) \lor 0}}. \label{eq:lil.bpp.coupling}
\end{align}
Furthermore, for simplicity, we assume that $r_n$ decays as
\begin{align}
  r_n = n^{-\alpha/d}, \quad 1 < \alpha < \frac{k}{k-1}, \label{eq:rn.decay.speed}
\end{align}
where the condition $1 < \alpha < k/(k-1)$ is imposed to ensure that \eqref{eq:rn.sparse} holds.

Following \cite{besy2008mdp}, we introduce a metric $\distancemeas$ on $\signedmeas(E)$.
Let $\{f_q\}_{q \geq 1}$ be a sequence of bounded continuous functions on $E$.
We assume that $\{f_q\}_{q \geq 1}$ induces the metric $\distancemeas$ on $\signedmeas(E)$ defined below;
for example, one may take $\{f_q\}_{q \geq 1}$ to be a dense subset of the space of all continuous functions with compact support on $E$.
Furthermore, we assume that each $f_q$ is continuously differentiable,
and that there exist constants $C > 0$ and $\fqexp \geq 0$ such that for all $q \geq 1$,
\begin{align}
  \suprdnorm{\nabla f_q} \coloneqq \supnorm{ \norm{\nabla f_q}_{\reals^m}} \leq C \supnorm{f_q} q^{\fqexp}, \label{eq:test.func.grad.bound.cond}
\end{align}
where $\nabla f_q \coloneqq (\partial f_q / \partial x_i)_{1\leq i \leq m}$ is the gradient of $f_q$,
and $\supnorm{\cdot}$ denotes the supremum norm on $E$.
With these preparations, we define the metric $\distancemeas$ by
\begin{align*}
  \distancemeas(\mu, \nu) \coloneqq \sum_{q=1}^{\infty} \frac{1}{2^q \supnorm{f_q}} \abs{\pairing{f_q}{\mu} - \pairing{f_q}{\nu}}, \quad \mu, \nu \in \signedmeas(E).
\end{align*}
For the results on the LILs, we equip $\signedmeas(E)$ with the topology induced by the metric $d_{\signedmeas(E)}$.

For the realizations constructed above,
we shall prove the following measure-level LIL.

\begin{theorem}[Measure-level LIL]
\label{thm:lil.measure}
Let $\calX_n$ be either the Poisson point process $\ppp_n$ or the binomial point process $\bpp_n$.
Assume that $H$ satisfies \textup{(H1)--(H6)} and $r_n$ satisfies \eqref{eq:rn.decay.speed}.
Then, the set of all accumulation points of the sequence $\{(\ldpspeed \log \log \ldpspeed)^{-1/2} \centered{\kappa_n(\calX_n)} \}_{n \geq 1}$ with respect to the metric $\distancemeas$ almost surely coincides with $\Lambda^{-1}([0, 1]) = \{\nu \in \signedmeas(E): \Lambda(\nu) \leq 1 \}$.
\end{theorem}

As a natural consequence of \cref{thm:lil.measure}, we obtain the following vector-level LIL.

\begin{corollary}[Vector-level LIL]
\label{cor:lil.scalar}
Let $\calX_n$ be either the Poisson point process $\ppp_n$ or the binomial point process $\bpp_n$.
Assume that $H$ satisfies \textup{(H1)--(H6)} and $r_n$ satisfies \eqref{eq:rn.decay.speed}.
Then, the set of all accumulation points of the sequence $\{(\ldpspeed \log \log \ldpspeed)^{-1/2} \centered{T_n(\calX_n)} \}_{n \geq 1}$ almost surely coincides with $I^{-1}([0, 1]) = \{\bfu \in \reals^m: I(\bfu) \leq 1\}$.
In particular, when $m = 1$, almost surely,
\begin{align}
  \limsup_{n \to \infty} \frac{1}{\sqrt{2 \ldpspeed \log \log \ldpspeed}}\centered{T_n(\calX_n)} = \sqrt{\frac{1}{k!} \int_{(\reals^d)^{k-1}} H(\bfzero_d, \bfy)^2 \,d\bfy}
\end{align}
and
\begin{align}
  \liminf_{n \to \infty} \frac{1}{\sqrt{2 \ldpspeed \log \log \ldpspeed}}\centered{T_n(\calX_n)} = -\sqrt{\frac{1}{k!} \int_{(\reals^d)^{k-1}} H(\bfzero_d, \bfy)^2 \,d\bfy}.
\end{align}
\end{corollary}

\begin{remark}
Condition (H6) is needed only to show that all accumulation points of $\{(\ldpspeed \log \log \ldpspeed)^{-1/2} \centered{\kappa_n(\calX_n)} \}_{n \geq 1}$ and $\{(\ldpspeed \log \log \ldpspeed)^{-1/2} \centered{T_n(\calX_n)} \}_{n \geq 1}$ are contained in $\Lambda^{-1}([0, 1])$ and $I^{-1}([0, 1])$ (\eqref{eq:lil.upper.meas} and \eqref{eq:lil.upper.vec}, respectively).
In particular, it is used in \cref{lemma:lil.upper.z4,lemma:lil.upper.z4.ei}.
However, it is not needed to show that all points in $\Lambda^{-1}([0, 1])$ and $I^{-1}([0, 1])$ are accumulation points of these sequences (\eqref{eq:lil.lower.meas} and \eqref{eq:lil.lower.vec}).
\end{remark}

\subsection{Applications}
In this subsection, we present several applications of our main results.

\begin{example}[\v{C}ech complex component counts]
Let $\cechcpx{\calX}{r} $ be the \v{C}ech complex for a point set $\calX = \{x_1, \dots, x_M\} \subset \reals^d $ and a connectivity radius $r > 0$.
That is, we define it as follows:
\begin{align}
  \cechcpx{\calX}{r} \coloneqq \Bab[bigg]{\{x_{i_0}, \dots, x_{i_j}\} : j \geq 0, \, x_{i_0}, \dots, x_{i_j} \in \calX, \, \bigcap_{l=0}^j \ball{x_{i_l}}{r/2} \neq \emptyset}. \label{eq:cech.complex}
\end{align}
We define the function $H = (h^{(1)}, \dots, h^{(m)}) $ by
\begin{align}
  h^{(i)}(x_1, \dots, x_k) \coloneqq \indicator[none]{\cechcpx{\{x_1, \dots, x_k\}}{t_i} \cong \Gamma_i},
  \quad (x_1, \dots, x_{k}) \in (\reals^d)^{k},
\end{align}
where $t_1, \dots, t_m \in (0, \infty) $,
$\Gamma_i $ is a connected \v{C}ech complex with $k$ vertices,
and $\cong $ denotes isomorphism between simplicial complexes.
We also assume that $\Gamma_{i} \not\cong \Gamma_{j} $ for $i \neq j $.
Here, we investigate the MDP and LIL for the following functional:
\begin{align}
  S_{n}(\calX) = (S_{n}^{(i)}(\calX))_{i=1,\dots, m}
  \coloneqq \sum_{\calY \subset \calX, \card{\calY} = k} \isolated[n]{\calY}{\calX}{\bft} \odot H_{n}(\calY),
\end{align}
where $\bft \coloneqq (t_1, \dots, t_m) $.
Namely, $S_{n}^{(i)}(\calX)$ represents the total number of connected components isomorphic to $\Gamma_i $ in $\cechcpx{\calX}{r_n t_i} $.
Since $H$ satisfies conditions (H1)--(H6),
by applying \cref{cor:mdp.vector} and \cref{cor:lil.scalar},
we obtain the MDP and LIL for this functional.

\begin{corollary}[MDP for component counts in random geometric complexes]
Let $\calX_n $ be either $\ppp_n $ or $\bpp_n $.
Assume that $\ldpspeed = \nkrndkmo \to \infty $ and $\nrnd \to 0 $ as $n \to \infty $.
Then, for any Borel set $A \subset \reals^m$,
\begin{align}
  - \inf_{\bfu \in \interiorrd{A}} I(\bfu)
  \leq \liminf_{n \to \infty} \frac{\ldpspeed}{\mdpfactor^2} \log \probab[Big]{\frac{1}{\mdpfactor} \centered{S_n(\calX_n)} \in A}
  \leq \limsup_{n \to \infty} \frac{\ldpspeed}{\mdpfactor^2} \log \probab[Big]{\frac{1}{\mdpfactor} \centered{S_n(\calX_n)} \in A}
  \leq - \inf_{\bfu \in \closurerd{A}} I(\bfu),
\end{align}
where $\map{I}{\reals^m}{[0, \infty]} $ is the rate function given by \eqref{eq:ratefunc.vector}.
\end{corollary}

\begin{corollary}[LIL for component counts in random geometric complexes]
Let $\calX_n $ be either $\ppp_n $ or $\bpp_n $.
Assume that $r_n$ satisfies \eqref{eq:rn.decay.speed}.
Then, the set of all accumulation points of $\{(\ldpspeed \log \log \ldpspeed)^{-1/2} \centered{S_n(\calX_n)} \}_{n \geq 1}$ almost surely coincides with $I^{-1}([0, 1])$.
In particular, for $i = 1, \dots, m$, almost surely
\begin{align}
  \limsup_{n \to \infty} \frac{1}{\sqrt{2 \ldpspeed \log \log \ldpspeed}}\centered{S_{n}^{(i)}(\calX_n)} = \sqrt{\frac{1}{k!} \lebmeas\ab\big(\{\bfy \in (\reals^{d})^{k-1} : \cechcpx{\{\bfzero_d, \bfy\}}{t_i} \cong \Gamma_i \})}
\end{align}
and
\begin{align}
  \liminf_{n \to \infty} \frac{1}{\sqrt{2 \ldpspeed \log \log \ldpspeed}}\centered{S_{n}^{(i)}(\calX_n)} = - \sqrt{\frac{1}{k!} \lebmeas\ab\big(\{\bfy \in (\reals^{d})^{k-1} : \cechcpx{\{\bfzero_d, \bfy\}}{t_i} \cong \Gamma_i \})}.
\end{align}
\end{corollary}

\end{example}

\begin{example}[Component counts in a random geometric graph]
Next, we also investigate the component counts for random geometric graphs.
Let $\rgg{\calX}{r} $ be the geometric graph for a finite point set $\calX = \{x_1, \dots, x_M\} \subset \reals^d $ and a connectivity radius $r > 0$.
That is, $\rgg{\calX}{r} $ is a graph with the vertex set $ \calX$ and the edge set $E(\calX, r) $ given by
\begin{align}
  E(\calX, r) \coloneqq \Bab{\{x_i, x_j\}: 1\leq i < j \leq M, \, \norm{x_i - x_j}_{\reals^d} \leq r}.
\end{align}
We define the function $H = (h^{(1)}, \dots, h^{(m)}) $ by
\begin{align}
  h^{(i)}(x_1, \dots, x_k) \coloneqq \indicator[none]{\rgg{\{x_1, \dots, x_k\}}{t_i} \cong \Gamma_i},
  \quad (x_1, \dots, x_{k}) \in (\reals^d)^{k},
\end{align}
where $t_1, \dots, t_m \in (0, \infty) $,
$\Gamma_i $ is a connected graph with $k$ vertices,
and $\cong $ denotes isomorphism between graphs.
We also assume that $\Gamma_{i} \not\cong \Gamma_{j} $ for $i \neq j $.
Let $S_{n}^{(i)}(\calX)$ be the total number of connected components isomorphic to $\Gamma_i $ in $\rgg{\calX}{r_n t_i} $,
and we set
\begin{align}
  S_{n}(\calX) = (S_{n}^{(i)}(\calX))_{i=1,\dots, m}
  \coloneqq \sum_{\calY \subset \calX, \card{\calY} = k} \isolated[n]{\calY}{\calX}{\bft} \odot H_{n}(\calY),
\end{align}
where $\bft \coloneqq (t_1, \dots, t_m) $.
As in the case of component counts for random geometric complexes,
we obtain the MDP and LIL for this functional.

\begin{corollary}[MDP for component counts in random geometric graphs]
Let $\calX_n $ be either $\ppp_n $ or $\bpp_n $.
Assume that $\ldpspeed = \nkrndkmo \to \infty $ and $\nrnd \to 0 $ as $n \to \infty $.
Then, for any Borel set $A \subset \reals^m$,
\begin{align}
  - \inf_{\bfu \in \interiorrd{A}} I(\bfu)
  \leq \liminf_{n \to \infty} \frac{\ldpspeed}{\mdpfactor^2} \log \probab[Big]{\frac{1}{\mdpfactor} \centered{S_n(\calX_n)} \in A}
  \leq \limsup_{n \to \infty} \frac{\ldpspeed}{\mdpfactor^2} \log \probab[Big]{\frac{1}{\mdpfactor} \centered{S_n(\calX_n)} \in A}
  \leq - \inf_{\bfu \in \closurerd{A}} I(\bfu),
\end{align}
where $\map{I}{\reals^m}{[0, \infty]} $ is the rate function given by \eqref{eq:ratefunc.vector}.
\end{corollary}

\begin{corollary}[LIL for component counts in random geometric graphs]
Let $\calX_n $ be either $\ppp_n $ or $\bpp_n $.
Assume that $r_n$ satisfies \eqref{eq:rn.decay.speed}.
Then, the set of all accumulation points of $\{(\ldpspeed \log \log \ldpspeed)^{-1/2} \centered{S_n(\calX_n)} \}_{n \geq 1}$ almost surely coincides with $I^{-1}([0, 1])$.
In particular, for $i = 1, \dots, m$, almost surely
\begin{align}
  \limsup_{n \to \infty} \frac{1}{\sqrt{2 \ldpspeed \log \log \ldpspeed}}\centered{S_{n}^{(i)}(\calX_n)} = \sqrt{\frac{1}{k!} \lebmeas\ab\big(\{\bfy \in (\reals^{d})^{k-1} : \rgg{\{\bfzero_d, \bfy\}}{t_i} \cong \Gamma_i \})}
\end{align}
and
\begin{align}
  \liminf_{n \to \infty} \frac{1}{\sqrt{2 \ldpspeed \log \log \ldpspeed}}\centered{S_{n}^{(i)}(\calX_n)} = - \sqrt{\frac{1}{k!} \lebmeas\ab\big(\{\bfy \in (\reals^{d})^{k-1} : \rgg{\{\bfzero_d, \bfy\}}{t_i} \cong \Gamma_i \})}.
\end{align}
\end{corollary}
\end{example}

\begin{example}[Morse critical points of min-type distance functions]
In this example, we investigate the MDP and LIL for the number of Morse critical points of min-type distance functions.
As in \cite{ho2023ldp}, our discussion in this example is restricted to dimension two ($d=2$), with $k=3$ in the notation of \cref{sec:model}.
This restriction is due to the fact that the arguments for the MDP and LIL require a specific geometric property unique to two dimensions,
namely that the number of vertices in a planar Voronoi tessellation is bounded above by a linear function of the number of points.

For a homogeneous Poisson point process $\ppp_n$ of intensity $n$ on $[0, 1]^2$,
we define the min-type distance function $\map{d_{\ppp_n}}{\reals^2}{[0, \infty)}$ as follows:
\begin{align}
  d_{\ppp_n}(x) \coloneqq \min_{y \in \ppp_n} \norm{x - y}_{\reals^2}, \quad x \in \reals^2. 
\end{align}
A point $c \in \reals^2$ is defined to be a Morse critical point of index 2 for $d_{\ppp_n}$ if it satisfies the following condition:
there exists a three-point subset $\calY \subset \ppp_n$ satisfying (i)--(iii) below.
\begin{description}
  \item[(i)] The three points in $\calY$ are in general position.
  \item[(ii)] It holds that $d_{\ppp_n}(c) = \norm{c - y}_{\reals^2}$ for all $y \in \calY$,
  and furthermore, $d_{\ppp_n}(c) < \min_{z \in \ppp_n \setminus \calY} \norm{c - z}_{\reals^2}$.
  \item[(iii)] It holds that $c \in \convint{\calY}$,
  where $\convint{\calY}$ denotes the interior of the convex hull spanned by the points in $\calY$.
\end{description}

The primary motivation for studying the Morse critical points of min-type distance functions is that they serve as highly tractable alternatives to Betti numbers
for capturing the homological changes of random geometric complexes.
According to the Nerve theorem (e.g., Theorem 10.7 in \cite{bjorner1995topological}),
for $r>0$, the sublevel set $d_{\ppp_n}^{-1}([0, r])$ of the distance function $d_{\ppp_n}$ is homotopy equivalent to the \v{C}ech complex $\cechcpx{\ppp_n}{2r} $ in \eqref{eq:cech.complex}.
Therefore, it is well known that the number of Morse critical points of index 2 behaves similarly to the first Betti number of the corresponding \v{C}ech complex \cite{ba2014distfunc,bm2015topolpgymanifold}.
In general, directly evaluating Betti numbers requires taking into account the global structure of the entire simplicial complex,
which often makes the analysis difficult.
On the other hand, since Morse critical points are determined solely by local geometric conditions involving a few points (such as the interior of the circumcircle being empty),
their probabilistic evaluation is relatively straightforward,
and thus they serve as a practical proxy for investigating complex homological properties.

To formally denote Morse critical points, we introduce the following notation.
For a three-point subset $\calY \subset \ppp_n$ in general position,
let $\gamma(\calY)$ be the center of the unique circumcircle passing through $\calY$,
$\calR(\calY)$ the circumradius of $\calY$,
and $\calU(\calY) \coloneqq \interiorrd{\ball{\gamma(\calY)}{\calR(\calY)}} $ the interior of the circumcircle of $\calY $.
Note that when $\gamma(\calY) \in \convint{\calY} $ and $\calU(\calY) \cap \ppp_n = \emptyset $,
the point $\gamma(\calY) $ is a Morse critical point of index 2
and $\calR(\calY) = d_{\ppp_n}(\gamma(\calY)) $.
For given $t_1, \dots, t_m \in (0, \infty)$ (with $t_i \ne t_j$ for $i \ne j$),
we define the functional $N_n$ as follows:
\begin{align}
  N_n = (N_n^{(i)})_{i=1,\dots,m}
  \coloneqq \ab\bigg( \sum_{\calY \subset \ppp_n, \card{\calY}=3} \indicator[none]{\gamma(\calY) \in \convint{\calY}, \calR(\calY) \le r_n t_i, \calU(\calY) \cap \ppp_n = \emptyset} )_{i=1,\dots,m}.
\end{align}
Here, $N_n^{(i)} $ represents the number of Morse critical points of index 2 whose critical values are at most $r_n t_i $.

We can also establish the MDP and LIL for $N_n$.
Note that, as in \cite{ho2023ldp}, we do not provide the binomial point process version of this result,
as it would require further geometric considerations.

\begin{corollary}[MDP for Morse critical points]
\label{cor:mdp.morse.vector}
Assume that $\ldpspeed = n^{3} r_n^{4} \to \infty $ and $n r_n^2 \to 0 $ as $n \to \infty $.
Then, for any Borel set $A \subset \reals^m$,
\begin{align}
  - \inf_{\bfu \in \interiorrd{A}} I(\bfu)
  \leq \liminf_{n \to \infty} \frac{\ldpspeed}{\mdpfactor^2} \log \probab[Big]{\frac{1}{\mdpfactor} \centered{N_n} \in A}
  \leq \limsup_{n \to \infty} \frac{\ldpspeed}{\mdpfactor^2} \log \probab[Big]{\frac{1}{\mdpfactor} \centered{N_n} \in A}
  \leq - \inf_{\bfu \in \closurerd{A}} I(\bfu),
\end{align}
where $\map{I}{\reals^m}{[0, \infty]} $ is the rate function given by \eqref{eq:ratefunc.vector}.
\end{corollary}

\begin{corollary}[LIL for Morse critical points]
\label{cor:lil.morse.vector}
Assume that $\ldpspeed = n^{3} r_n^{4} \to \infty $ and $n r_n^2 \to 0 $ as $n \to \infty $.
Assume that $r_n$ satisfies \eqref{eq:rn.decay.speed}.
Then, the set of all accumulation points of $\{(\ldpspeed \log \log \ldpspeed)^{-1/2} \centered{N_n} \}_{n \geq 1}$ almost surely coincides with $I^{-1}([0, 1])$.
In particular, for $i = 1, \dots, m$, almost surely
\begin{align}
  \limsup_{n \to \infty} \frac{1}{\sqrt{2 \ldpspeed \log \log \ldpspeed}}\centered{N_{n}^{(i)}} = \sqrt{\frac{1}{6} \lebmeas\ab\big(\{\bfy \in (\reals^{2})^{2} : \gamma(\bfzero_2, \bfy) \in \convint{\bfzero_2, \bfy}, \calR(\bfzero_2, \bfy) \le t_i \})}
\end{align}
and
\begin{align}
  \liminf_{n \to \infty} \frac{1}{\sqrt{2 \ldpspeed \log \log \ldpspeed}}\centered{N_{n}^{(i)}} = - \sqrt{\frac{1}{6} \lebmeas\ab\big(\{\bfy \in (\reals^{2})^{2} : \gamma(\bfzero_2, \bfy) \in \convint{\bfzero_2, \bfy}, \calR(\bfzero_2, \bfy) \le t_i \})}.
\end{align}
\end{corollary}

These results essentially follow from the arguments used for \cref{cor:mdp.vector} and \cref{cor:lil.scalar}.
Indeed, if we define the function $H = (h^{(1)}, \dots, h^{(m)}) $ by
\begin{align}
  h^{(i)}(x_1, x_2, x_3) \coloneqq \indicator[none]{\gamma(x_1, x_2, x_3) \in \convint{x_1, x_2, x_3}, \calR(x_1, x_2, x_3) \le t_i},
  \quad (x_1, x_2, x_3) \in (\reals^2)^{3}, \label{eq:h.morse.critical}
\end{align}
(one can easily verify that $H$ satisfies conditions (H1)--(H6)),
and if we define the indicator function $c(\calY, \calX) $ by
\begin{align}
  c(\calY, \calX) = c(\calY, \calX \setminus \calY)
  \coloneqq \indicator[none]{\calU(\calY) \cap \calX = \emptyset}, \label{eq:indicator.morse}
\end{align}
then $N_n$ can be expressed as $N_n= \sum_{\calY \subset \ppp_n, \card{\calY}=3} c(\calY, \ppp_n) H_n(\calY) $.
However, since the indicator function \eqref{eq:indicator.morse} is different from \eqref{eq:isolated.def},
we cannot directly apply \cref{cor:mdp.vector} and \cref{cor:lil.scalar} to $\{N_n\}_{n \geq 1}$.
Nevertheless, by replacing \eqref{eq:isolated.def} with \eqref{eq:indicator.morse}
and making minor modifications to the proofs of \cref{cor:mdp.vector} and \cref{cor:lil.scalar},
we can obtain \cref{cor:mdp.morse.vector} and \cref{cor:lil.morse.vector}.
The proofs of \cref{cor:mdp.morse.vector} and \cref{cor:lil.morse.vector} are given in \cref{sec:morse.proof}.

\end{example}

\begin{remark}
The LDP for the first persistent Betti number for the alpha complex in dimension two ($d=2$) has been established by Hirsch and Owada (Theorem 4.1 in \cite{ho2023ldp}).
However, their arguments cannot be applied to the MDP and LIL targeted in this paper,
which necessitates the development of further techniques.
\end{remark}

\subsection{General notation}
Finally, we collect some general notation used throughout this paper.
For vectors $\bfu = (u_1, \dots, u_m), \bfv = (v_1, \dots, v_m)$ in  $\reals^m$, we denote the Euclidean inner product by $\iprod{\bfu}{\bfv} \coloneqq \sum_{i=1}^m u_i v_i$,
and the Euclidean norm on $\reals^d$ by $\norm{\bfu}_{\reals^m} \coloneqq \iprod{\bfu}{\bfu}^{1/2}$.
Let $\bfe_i \coloneqq (0,\dots,1,\dots,0) \in \reals^m$ be the $i$-th standard unit vector in $\reals^m$.
Let $Q(a, z) \coloneqq a(\unitcube + z) $ for $a > 0$ and $z \in \integers^d$.
For a Borel set $A \subset \reals^d $
and $a > 0$,
we define $A^{(a)} \coloneqq \{x \in \reals^d : \inf_{y \in A} \norm{x - y}_{\reals^d} \leq a \} $.
For $B \in \naturals $ and $w \in \{0, 1, \dots, B-1\}^{d} $,
let $B\integers^d + w \coloneq \{Bz + w : z \in \integers^d \} $.
We let $\ball{x}{a} \coloneqq \{y \in \reals^d: \norm{x - y}_{\reals^d} \leq a \} $ denote the $d$-dimensional ball,
and for $\bfx = (x_1, \dots, x_M) \in (\reals^d)^M $ and $a > 0$,
we write $\ball{\bfx}{a} = \ball{\{x_1, \dots, x_M\}}{a} \coloneqq \bigcup_{j=1}^M \ball{x_j}{a} $.
We let $\lexineq$ denote the lexicographic order, 
and for a finite set $\calX \subset \reals^d $,
we use $\lex_j(\calX)$ to represent the $j$-th element of $\calX$ with respect to this order;
that is, if $\calX = \{x_1, \dots, x_M\} $ is ordered such that $x_1 \lexineq \dots \lexineq x_M $,
then $\lex_j(\calX) = x_j $.
For $x \in \reals$,
let $\rounddown{x} $ denote the largest integer less than or equal to $x$,
and let $\roundup{x}$ denote the smallest integer greater than or equal to $x$.
Let
\begin{align}
  \tailfunc(x) \coloneqq 1 + x - \sqrt{1 + 2x}, \quad x > 0. \label{eq:tailfunc}
\end{align}
Note that
\begin{align}
  \tailfunc(x) >
  \begin{cases}
  x^2/4 & \text{for } 0 < x \leq 1,\\ 
  x/4  & \text{for } x \geq 1
  \end{cases}, \label{eq:tailfunc.lower.bound}
\end{align}
and for all $s, a, c > 0$, it holds that
\begin{align}
  \inf_{\lambda \in (0, c^{-1})} \ab(-\lambda s + \frac{a\lambda^2}{2(1 - c\lambda)})
  = - \frac{a}{c^2}\tailfunc\ab(\frac{cs}{a}).
\end{align}
For a map $\map{f}{E}{\reals^m} $,
we define $\zeroed{f}(\bfu) \coloneqq \indicator[none]{\bfu \neq \bfzero_m} f(\bfu)$ for $\bfu \in E $.
For a measurable function $f$ on $E$ and a signed measure $\nu \in \signedmeas(E)$,
we write $\pairing{f}{\nu} \coloneqq \int_E f \,d\nu$.
For a function $f \colon E \to \reals$, $\supnorm{f} \coloneqq \sup_{\bfu \in E} \abs{f(\bfu)}$ denotes its supremum norm,
and for a smooth function $f$,
$\nabla f \coloneqq (\partial f / \partial x_i)_{1\leq i \leq m} $ and
$\suprdnorm{\nabla f} \coloneqq \supnorm{ \norm{\nabla f}_{\reals^m}}$.
Let $\pppentire$ be a homogeneous Poisson point process on $\reals^d$ with unit intensity, and $W_n \coloneqq [0, n^{1/d}]^d$,
Throughout this paper, $C > 0$ denotes a generic positive constant that may change from line to line.

\section{Proof of the moderate deviation principles}\label{sec:pf.mdp}

In this section, we establish \cref{thm:mdp.measure} and \cref{cor:mdp.vector}.
Following the approach in \cite{ho2023ldp}, we introduce several auxiliary processes.
Let $Q_{n, z}^{*} \coloneqq Q(\ldpspeed^{-1/d}, z)$ for $z \in \integers^d$
and
$I_n^{*} \coloneqq \{z \in \integers^d : Q_{n,z}^{*} \subset \unitcube \}$.
For a finite set $\calX \subset \reals^d$,
we define
\begin{align}
  \eta_{n,z}(\calX) &\coloneqq \sum_{\calY \subset \calX \cap Q_{n,z}^{*}, \card{\calY}=k} \isolated[n]{\calY}{\calX \cap Q_{n,z}^{*}}{t} \indicator[none]{H_n(\calY) \neq \bfzero_{m}} \diracdelta{H_n(\calY)}, \\
  \eta_n(\calX) &\coloneqq \sum_{z \in I_n^{*}} \eta_{n,z}(\calX), \quad
  \xi_n(\calX)\coloneqq \sum_{\calY \subset \calX, \card{\calY}=k} \isolated[n]{\calY}{\calX}{t} \indicator[none]{H_n(\calY) \neq \bfzero_{m}} \diracdelta{H_n(\calY)}.
\end{align}

The proofs of \cref{thm:mdp.measure} and \cref{cor:mdp.vector}
follow readily from the lemmas established in the following subsections
together with standard techniques from large deviation theory.
We now prove \cref{thm:mdp.measure} and \cref{cor:mdp.vector},
assuming for the moment the lemmas to be established later.

\begin{proof}[Proof of \cref{thm:mdp.measure} and \cref{cor:mdp.vector}]
First, by \cref{prop:mdp.eta}, the sequence $\{\mdpfactor^{-1} \centered{\eta_n(\ppp_n)}\}_{n \geq 1} $ satisfies the same LDP as in \cref{thm:mdp.measure}, which is obtained by applying the G\"{a}rtner--Ellis theorem and the Dawson--G\"{a}rtner theorem.
Next, \cref{prop:exp.eq.eta.xi} ensures that $\{\mdpfactor^{-1} \centered{\eta_n(\ppp_n)}\}_{n \geq 1}$ and $\{\mdpfactor^{-1} \centered{\xi_n(\ppp_n)}\}_{n \geq 1}$ are exponentially equivalent,
and $\{\mdpfactor^{-1} \centered{\xi_n(\ppp_n)}\}_{n \geq 1}$ and $\{\mdpfactor^{-1} \centered{\kappa_n(\ppp_n)}\}_{n \geq 1}$ are also exponentially equivalent by \cref{prop:exp.eq.xi.kappa}.
Note that this exponential equivalence is first established on $\reals^m$, and then extended to $\signedmeas(E)$ by applying the Dawson--G\"{a}rtner theorem.
Consequently, the LDP holds for $\{\mdpfactor^{-1} \centered{\kappa_n(\ppp_n)}\}_{n \geq 1} $, which completes the proof for the Poisson point process case in \cref{thm:mdp.measure}.
Furthermore, applying the contraction principle to $\{\mdpfactor^{-1} \centered{\kappa_n(\ppp_n)}\}_{n \geq 1} $ via \cref{lem:contraction_cutoff}, we obtain the LDP for $\{\mdpfactor^{-1} \centered{T_n(\ppp_n)}\}_{n \geq 1} $.
Note that since the contraction map employed in \cref{lem:contraction_cutoff} is discontinuous, the cutoff procedure requires justification.
Finally, \cref{prop:exp.eq.ppp.bpp} implies the exponential equivalence of $\{\mdpfactor^{-1} \centered{\kappa_n(\ppp_n)}\}_{n \geq 1} $ and $\{\mdpfactor^{-1} \centered{\kappa_n(\bpp_n)}\}_{n \geq 1}$, as well as that of $\{\mdpfactor^{-1} \centered{T_n(\ppp_n)}\}_{n \geq 1} $ and $\{\mdpfactor^{-1} \centered{T_n(\bpp_n)}\}_{n \geq 1}$.
This completes the proof of \cref{thm:mdp.measure} and \cref{cor:mdp.vector} for each point process.
\end{proof}

\subsection{The Poisson case}
In this subsection, we focus on the Poisson point process $\ppp_n$ only. 
For simplicity of notation, we omit $\ppp_n$ from the notation; for instance, $\eta_n(\ppp_n) $ is abbreviated to $\eta_n$.

We first establish a technical proposition concerning the asymptotic behavior of the moment generating function
for a sequence of random variables with uniformly bounded exponential moments.

\begin{proposition}
\label{prop:exp.moment.finite}
Let $\Bab{Z_n}_{n \geq 1}$ be a sequence of $\reals$-valued random variables satisfying $\sup_{n \geq 1}\expect[none]{\exp(\lambda \abs{Z_n})} < \infty $ for some $\lambda > 0$.
Let $\{a_n\}_{n \geq 1} $ be a sequence of positive numbers such that $a_n \to 0 $ as $n \to \infty$, and let
\begin{align}
  \epsilon_n \coloneqq \expect[none]{\exp(a_n  \centered{Z_n})} - \ab\Big(1 + \frac{a_n^2}{2}\expect[big]{(\centered{Z_n})^2}).
\end{align}
Then, $\epsilon_n = \lorder(a_n^3) $ as $n \to \infty$.
\end{proposition}

\begin{proof}
For sufficiently large $n$, we have $a_n < \lambda/3 $ and
\begin{align}
  \abs{\epsilon_n} \leq a_n^3 \expect{\abs[none]{\centered{Z_n}}^3 \exp(a_n \abs[none]{\centered{Z_n}})}
  \leq a_n^3 \expect{\abs[none]{\centered{Z_n}}^3 \exp(\lambda/3\abs[none]{\centered{Z_n}})}
  \leq a_n^3 (3/\lambda)^3 \expect{\exp(\lambda \abs[none]{\centered{Z_n}})},
\end{align}
where we used the elementary inequality $\abs{e^{x} - (1 + x + x^2/2) } \leq \abs{x}^3 e^{\abs{x}} \leq e^{3\abs{x}} $ for all $x \in \reals$.
An application of Jensen's inequality yields $\expect[none]{\exp(\lambda \abs[none]{\centered{Z_n}})} \leq \expect[none]{\exp(\lambda \abs{Z_n})} \exp(\lambda \expect[none]{\abs{Z_n}}) \leq \expect[none]{\exp(\lambda \abs{Z_n})}^2 $.
Thus, it follows that
\begin{align}
  \limsup_{n \to \infty} \abs{\epsilon_n} / a_n^3
  \leq \limsup_{n \to \infty} (3/\lambda)^3 \expect{\exp(\lambda \abs{Z_n})}^2 < \infty,
\end{align}
which implies $\epsilon_n = \lorder(a_n^3) $.
\end{proof}

Next, we compute the covariance of geometric statistics under the Poisson point process.
The following proposition provides an explicit representation
of the covariance with an error bound of the required order.

\begin{proposition}
\label{prop:cov}
Let $\map{F_1, F_2}{(\reals^d)^k}{\reals}$ be bounded functionals satisfying conditions \textup{(H1)--(H3)} with $H$ replaced by $F_1$ or $F_2$.
Let $A, A' \subset \unitcube$ be Borel sets, and let $t > 0$, $n  \geq 1$.
Define
\begin{align}
  Z_{i} \coloneqq \sum_{\calY \subset \ppp_n \cap A, \card{\calY}=k} \isolated[n]{\calY}{\ppp_n \cap A'}{t} F_i(r_n^{-1} \calY)
\end{align}
for $i=1,2$. Then,
\begin{align}
  \expect[none]{\centered{Z_1}\, \centered{Z_2}}
 = &\frac{\ldpspeed}{k!} \int_A dx \int_{(\reals^d)^{k-1}} d\bfy\, F_1(\bfzero_d, \bfy) F_2(\bfzero_d, \bfy) \prod_{i=1}^{k-1}\indicator[none]{x + r_n y_i \in A} \\
  &\quad \times \exp\ab\big(-n \lebmeas(A' \cap \ball{\Bab{x} \cup (x + r_n \bfy)}{r_n t}))
  + \mathsf{error},
\end{align}
where
\begin{align}
  \abs{\mathsf{error}}
  \leq C \supnorm{F_1} \supnorm{F_2} \lebmeas(A) \ldpspeed \nrnd (1 + \nrnd)^{k-1}
\end{align}
and $C > 0$ is a constant depending only on $d, k, t $, and $L$.
\end{proposition}

\begin{proof}
By the Mecke formula (see, e.g., Chapter 4 in \cite{lp2018lecturesppp}), we can write
\begin{align}
  \expect[none]{ Z_1 Z_2 }
  &= \expect[Bigg]{ \sum_{\calY, \calY' \subset \ppp_n \cap A}
  \indicator[none]{\card{\calY}=\card{\calY'}=k}
  \isolated[n]{\calY}{\ppp_n \cap A'}{t}
  \isolated[n]{\calY'}{\ppp_n \cap A'}{t} F_1(r_n^{-1}\calY) F_2(r_n^{-1}\calY') } \\
  &= \sum_{l=0}^{k} \frac{n^{2k-l}}{l!((k-l)!)^2} \int_{A^{l}} d \bfx_0 \int_{A^{k-l}} d \bfx_1 \int_{A^{k-l}} d \bfx_2\, F_1(r_n^{-1}(\bfx_0 \cup \bfx_1)) F_2(r_n^{-1}(\bfx_0 \cup \bfx_2)) \\
  &\quad \times \expect[none]{\isolated[n]{\bfx_0 \cup \bfx_1}{(\ppp_n \cup \bfx_2) \cap A'}{t} \isolated[n]{\bfx_0 \cup \bfx_2}{(\ppp_n \cup \bfx_1) \cap A'}{t}}, \label{eq:cov.cross}
\end{align}
and
\begin{align}
  &\expect[none]{Z_1} \expect[none]{Z_2}\\
  &\quad =\frac{n^{2k}}{(k!)^2} \int_{A^{k}} d \bfx_1 \int_{A^{k}} d \bfx_2\, 
  F_1(r_n^{-1} \bfx_1) F_2(r_n^{-1} \bfx_2)
  \expect[none]{\isolated[n]{\bfx_1}{\ppp_n \cap A'}{t}} \expect[none]{\isolated[n]{\bfx_2}{\ppp_n \cap A'}{t}}. \label{eq:cov.expect.prod}
\end{align}

For each $l = 1, \dots, k-1$, the integral in the sum of \eqref{eq:cov.cross} is bounded by
\begin{align}
  & \supnorm{F_1} \supnorm{F_2}
  \int_{A^l} d \bfx_0 \int_{A^{k-l}} d\bfx_1 \int_{A^{k-l}} d\bfx_2\,
  \indicator[none]{\diam(\bfx_0 \cup \bfx_1) \leq L r_n} \indicator[none]{\diam(\bfx_0 \cup \bfx_2) \leq L r_n} \\
  &\quad = \supnorm{F_1} \supnorm{F_2} r_n^{d(2k-l-1)} \int_{A} dx \int_{(\reals^d)^{l-1}} d\bfy_0 \int_{(\reals^d)^{k-l}} d \bfy_1 \int_{(\reals^d)^{k-l}} d \bfy_2 \\
  &\qquad \times \indicator[none]{\diam(\bfzero_d, \bfy_0 \cup \bfy_1) \leq L} \indicator[none]{\diam(\bfzero_d, \bfy_0 \cup \bfy_2) \leq L}\\
  &\qquad \times \prod_{i=1}^{l-1}\indicator[none]{x + r_n y_{0,i} \in A}
  \prod_{i=1}^{k-l}\indicator[none]{x + r_n y_{1,i} \in A}
  \prod_{i=1}^{k-l}\indicator[none]{x + r_n y_{2,i} \in A}\\
  &\quad\leq \supnorm{F_1} \supnorm{F_2} r_n^{d(2k-l-1)} \lebmeas(A) \lebmeas(\ball{\bfzero_d}{L})^{2k-1}.
\end{align}
For $l=k$, the corresponding integral in \eqref{eq:cov.cross} equals
\begin{align}
  &r_n^{d(k-1)} \int_{A} dx \int_{(\reals^d)^{k-1}} d\bfy\,
  F_1(\bfzero_d, \bfy) F_2(\bfzero_d, \bfy) \prod_{i=1}^{k-1}\indicator[none]{x + r_n y_i \in A} \\
  &\quad \times \exp\ab\big(-n \lebmeas(A' \cap \ball{\Bab{x} \cup (x + r_n \bfy)}{r_n t})).
\end{align}

The term $\expect[none]{\isolated[n]{\bfx_1}{(\ppp_n \cup \bfx_2) \cap A'}{t} \isolated[n]{\bfx_2}{(\ppp_n \cup \bfx_1) \cap A'}{t}}$ in the integral for $l=0$ in \eqref{eq:cov.cross} reduces to
\begin{align}
  \expect[none]{\isolated[n]{\bfx_1}{\ppp_n \cap A'}{t}} \expect[none]{\isolated[n]{\bfx_2}{\ppp_n \cap A'}{t}}
\end{align}
whenever $\ball{\bfx_1}{r_n t} \cap \ball{\bfx_2}{r_n t} = \emptyset $.
Thus, the difference between the $l=0$ term in \eqref{eq:cov.cross} and \eqref{eq:cov.expect.prod} is bounded by
\begin{align}
  &\frac{n^{2k}}{(k!)^2} \supnorm{F_1} \supnorm{F_2}
  \int_{A^k} d\bfx_1 \int_{A^k} d\bfx_2\,
  \indicator[none]{\diam(\bfx_1) \leq L r_n}\indicator[none]{\diam(\bfx_2) \leq L r_n}\\
  &\qquad \times \indicator[none]{\ball{\bfx_1}{r_n t} \cap \ball{\bfx_2}{r_n t} \neq \emptyset} \\
  &\quad\leq \frac{n^{2k}}{(k!)^2} \supnorm{F_1} \supnorm{F_2}
  r_n^{d(2k-1)} \int_{A} dx \int_{(\reals^d)^{k-1}} d\bfy_1 \int_{(\reals^d)^k} d \bfy_2\, \\
  &\quad \qquad \times \prod_{i=1}^{k-1}\indicator[none]{\norm[none]{y_{1, i}}_{\reals^d} \leq L}
  \prod_{i=1}^{k}\indicator[none]{\norm[norm]{y_{2, i}}_{\reals^d} \leq L + 2t} \\
  &\quad \leq \frac{\ldpspeed (\nrnd)^k}{(k!)^2} \supnorm{F_1} \supnorm{F_2}
  \lebmeas(A) \lebmeas(\ball{\bfzero_d}{L})^{k-1} \lebmeas(\ball{\bfzero_d}{L+2t})^{k}.
\end{align}

Combining these estimates, we obtain
\begin{align}
  \expect[none]{\centered{Z_1}\, \centered{Z_2}}
  &= \frac{\ldpspeed}{k!} \int_A dx \int_{(\reals^d)^{k-1}} d\bfy\, F_1(\bfzero_d, \bfy) F_2(\bfzero_d, \bfy) \prod_{i=1}^{k-1}\indicator[none]{x + r_n y_i \in A} \\
  &\quad \times \exp\ab\big(-n \lebmeas(A' \cap \ball{\Bab{x} \cup (x + r_n \bfy)}{r_n t})) + \mathsf{error},
\end{align}
where the error term satisfies
\begin{align}
  \abs{\mathsf{error}} &\leq \lebmeas(\ball{\bfzero_d}{L+2t})^{2k-1} \supnorm{F_1} \supnorm{F_2} \lebmeas(A) \ldpspeed \sum_{l=0}^{k-1} \frac{(\nrnd)^{k-l}}{l!((k-l)!)^2} \\
  &\leq \lebmeas(\ball{\bfzero_d}{L+2t})^{2k-1} \supnorm{F_1} \supnorm{F_2} \lebmeas(A) \ldpspeed \frac{1}{(k-1)!} \sum_{l=0}^{k-1} \frac{(k-1)! (\nrnd)^{k-l}}{l!(k-1-l)!}\\
  &= C \supnorm{F_1} \supnorm{F_2} \lebmeas(A) \ldpspeed (\nrnd) (1 + \nrnd)^{k-1}.
\end{align}
This concludes the proof of \cref{prop:cov}.
\end{proof}

Combining \cref{prop:exp.moment.finite,prop:cov},
we compute the following exponential moment related to $\centered{\eta_n}$,
which plays a fundamental role when applying the G\"{a}rtner--Ellis theorem.

\begin{proposition}
\label{prop:ge.eta}
Let $M \geq 1$ and $f_1, \dots, f_M \in \bddmble(E)$.
Then, for any $\lambda = \ab(\lambda_1, \dots, \lambda_M) \in \reals^{M}$,
\begin{align}
  \lim_{n \to \infty} \frac{\ldpspeed}{\mdpfactor^2} \log \expect[Big]{\exp\ab\Big(\frac{\mdpfactor}{\ldpspeed} \sum_{j=1}^M \lambda_j \pairing{f_j}{\centered{\eta_n}} )}
  = \frac{1}{2}\sum_{i,j=1}^M \lambda_i \lambda_j \int_{E} f_i f_j \,d\tau. \label{eq:ge.eta}
\end{align}
\end{proposition}

\begin{proof}
Since $\{\eta_{n,z}\}_{z \in I_n^{*}} $ is a family of i.i.d. random measures, it follows that
\begin{align}
  \frac{\ldpspeed}{\mdpfactor^2} \log \expect[Big]{ \exp\ab\Big(\frac{\mdpfactor}{\ldpspeed} \sum_{j=1}^M \lambda_j \pairing{f_j}{\centered{\eta_n}} ) }
  &= \frac{\ldpspeed \card{I_n^{*}}}{\mdpfactor^2} \log \expect[Big]{ \exp\ab\Big(\frac{\mdpfactor}{\ldpspeed} \sum_{j=1}^M \lambda_j \pairing{f_j}{\centered{\eta_{n,\bfzero_d}}} ) } \\
  &= \frac{\ldpspeed \card{I_n^{*}}}{\mdpfactor^2} \log \ab\bigg(1 + \frac{1}{2} \frac{\mdpfactor^2}{\ldpspeed^2} \expect[Big]{\ab\Big(\sum_{j=1}^M \lambda_j \pairing{f_j}{\centered{\eta_{n,\bfzero_d}}})^2} + \epsilon_n),
\end{align}
where
\begin{align}
  \epsilon_n \coloneqq \expect[Big]{\exp\ab\Big(\frac{\mdpfactor}{\ldpspeed} \sum_{j=1}^M \lambda_j \pairing{f_j}{\centered{\eta_{n,\bfzero_d}}} )}
  - \ab\bigg(1 + \frac{1}{2} \frac{\mdpfactor^2}{\ldpspeed^2} \expect[Big]{\ab\Big(\sum_{j=1}^M \lambda_j \pairing{f_j}{\centered{\eta_{n,\bfzero_d}}})^2}).
\end{align}

Note that
\begin{align}
  \abs{\sum_{j=1}^M \lambda_j \pairing{f_j}{\eta_{n,\bfzero_d}}}
  \leq \ab\Big(\max_{j=1,\dots,M} \abs{\lambda_j} \supnorm{f_j}) \sum_{\calY \subset \ppp_n \cap Q_{n,\bfzero_d}^{*}, \card{\calY}=k} \isolated[n]{\calY}{\ppp_n \cap Q_{n,\bfzero_d}^{*}}{t} \indicator[none]{H_n(\calY) \neq \bfzero_{m}}, \label{eq:eta.zero.bound}
\end{align}
Since the exponential moments of the right-hand side of \eqref{eq:eta.zero.bound} remain uniformly bounded as $n \to \infty$ (see (5.16) in \cite{ho2023ldp}),
\cref{prop:exp.moment.finite} implies that $\epsilon_n = \lorder( (\mdpfactor / \ldpspeed)^3 )$ as $n \to \infty$.
Applying \cref{prop:cov} with $A=A'=Q_{n,\bfzero_d}^{*}$,
as $n \to \infty$, we obtain
\begin{align}
  &\expect[none]{ \pairing{f_i}{\centered{\eta_{n,\bfzero_d}}} \pairing{f_j}{\centered{\eta_{n,\bfzero_d}}} } \\
  &\quad = \frac{\ldpspeed}{k!} \int_{Q_{n, \bfzero_d}^{*}} dx \int_{(\reals^d)^{k-1}} d \bfy\,
  % \indicator[none]{H(\bfzero_d, \bfy) \neq \bfzero_m}
  \zeroed{f_i}(H(\bfzero_d, \bfy)) \zeroed{f_j}(H(\bfzero_d, \bfy)) \\
  &\qquad \times \exp\ab\big(-n \lebmeas(Q_{n,\bfzero_d}^{*} \cap \ball{\Bab{x} \cup (x + r_n \bfy)}{r_n t})) \prod_{i=1}^{k-1}\indicator[none]{x + r_n y_i \in Q_{n,\bfzero_d}^{*}} + \lorder(\nrnd) \\
  &\quad \to \frac{1}{k!} \int_{(\reals^d)^{k-1}} 
  % \indicator[none]{H(\bfzero_d, \bfy) \neq \bfzero_m}
  \zeroed{f_i}(H(\bfzero_d, \bfy)) \zeroed{f_j}(H(\bfzero_d, \bfy)) \,d \bfy
  = \int_{E} f_i(\bfu) f_j(\bfu) \,\tau(d\bfu).
\end{align}
Therefore, since $\card{I_n^{*}}/\ldpspeed \to 1$, the left-hand side of \eqref{eq:ge.eta} is equal to
\begin{align}
  &\lim_{n\to\infty} \frac{\ldpspeed \card{I_n^{*}}}{\mdpfactor^2}
  \log \ab\Big(1 +  \frac{\mdpfactor^2}{2\ldpspeed^2} \sum_{i,j=1}^M \lambda_i \lambda_j \expect[none]{ \pairing{f_i}{\centered{\eta_{n,\bfzero_d}}} \pairing{f_j}{\centered{\eta_{n,\bfzero_d}}} } + \lorder\ab\Big(\frac{\mdpfactor^3}{\ldpspeed^3})) 
  = \frac{1}{2}\sum_{i,j=1}^M \lambda_i \lambda_j \int_{E} f_i f_j \,d\tau,
\end{align}
which establishes \eqref{eq:ge.eta}.
\end{proof}

By virtue of \cref{prop:ge.eta},
we obtain the MDP for the sequence $\{\mdpfactor^{-1} \centered{\eta_n}\}_{n \geq 1}$.

\begin{lemma}
\label{prop:mdp.eta}
The sequence $\{\mdpfactor^{-1} \centered{\eta_n}\}_{n \geq 1} $ satisfies the LDP in $\signedmeas(E)$,
endowed with the $\tau$-topology and the $\sigma$-algebra $\sigmatau$, with speed $\mdpfactor^2 / \ldpspeed $ and rate function $\Lambda$.
\end{lemma}

\begin{proof}
We apply the Dawson--G\"{a}rtner theorem for topological vector spaces (see Theorem 4.6.9 in \cite{dz2009ldp}).

For any bounded measurable functions $f_1, \dots, f_M \in \bddmble(E)$,
define $\map{\pi_{f_1, \dots, f_{M}}}{\signedmeas(E)}{\reals^M} $ by $\pi_{f_1, \dots, f_{M}}(\nu) \coloneqq (\pairing{f_i}{\nu})_{i=1,\dots,M} $.
For all $\lambda \in \reals^M$, \cref{prop:ge.eta} gives
\begin{align}
  \lim_{n \to \infty} \frac{\ldpspeed}{\mdpfactor^2} \log \expect[Big]{\exp\ab\Big(\frac{\mdpfactor}{\ldpspeed} \iprod[big]{\lambda}{\pi_{f_1, \dots, f_{M}}(\centered{\eta_n})})}
  = \frac{1}{2} \iprod{\lambda}{\Gamma_{f_1, \dots, f_M}' \lambda},
\end{align}
where $\Gamma_{f_1, \dots, f_M}' $ is the $M \times M$ symmetric matrix given by
\begin{align}
  \Gamma_{f_1, \dots, f_M}' \coloneqq \ab\Big(\int_{E} f_i f_j \,d\tau)_{1 \leq i, j \leq M}.
\end{align}
Thus, by the G\"{a}rtner--Ellis theorem (Theorem 2.3.6 in \cite{dz2009ldp}), the sequence $\Bab{\mdpfactor^{-1} \pi_{f_1, \dots, f_{M}}(\centered{\eta_n}) }_{n \geq 1} $ satisfies the LDP in $\reals^M $ with speed $\mdpfactor^2 / \ldpspeed $ and the good rate function
\begin{align}
  I_{f_1, \dots, f_M}'(\bfu)
  \coloneqq \sup_{\bfa \in \reals^M} \ab\Big(\iprod{\bfa}{\bfu} - \frac{1}{2} \iprod{\bfa}{\Gamma_{f_1, \dots, f_M}' \bfa}), \quad \bfu \in \reals^M.
\end{align}
Therefore, Theorem 4.6.9 in \cite{dz2009ldp} ensures that $\Bab{\mdpfactor^{-1} \centered{\eta_n} }_{n \geq 1} $ satisfies the LDP in $\signedmeas(E)$
endowed with the $\tau$-topology and the $\sigma$-field $\sigmatau$,
with the good rate function
\begin{align}
  \Lambda'(\nu) \coloneqq \sup_{M \geq 1} \sup_{f_1 \dots, f_M \in \bddmble(E)} I_{f_1, \dots, f_M}'(\pi_{f_1, \dots, f_{M}}(\nu)), \quad \nu \in \signedmeas(E).
\end{align}

We next show that $\Lambda' = \Lambda $.
According to Lemma 4.6.5 in \cite{dz2009ldp},
it suffices to verify that $\Lambda $ is a good rate function,
and that
\begin{align}
  I_{f_1, \dots, f_M}'(\bfu) = \inf \Bab{\Lambda(\nu) : \nu \in \signedmeas(E), \, \pi_{f_1, \dots, f_{M}}(\nu) = \bfu } \label{eq:proflim.eta.ratefunc}
\end{align}
holds for all $M \geq 1$, $f_1, \dots, f_M \in \bddmble(E) $ and $\bfu \in \reals^M$.

To show that $\Lambda$ is a good rate function,
it is required to prove that the level set $\Lambda^{-1}([0, a])$ is closed and relatively compact in the $\tau$-topology
(or simply $\tau$-closed and relatively $\tau$-compact, respectively) for any $a \ge 0$.
By the criterion for relative $\tau$-compactness (see Theorem 4.7.25 in \cite{bogachev2007measure1}),
it suffices to verify that $\Lambda^{-1}([0, a])$ is uniformly bounded in the total variation norm $\tvnorm{\cdot}$,
that is, $\sup  \{\tvnorm{\nu}: \nu \in \Lambda^{-1}([0, a]) \} < \infty$,
and is uniformly $\tau$-continuous, in the sense of uniform absolute continuity with respect to $\tau$.
To establish the uniform boundedness,
let $\nu \in \Lambda^{-1}([0, a])$.
Since $\nu \abscont \tau$, applying Jensen's inequality (or the Cauchy--Schwarz inequality) yields
\begin{align}
  \tvnorm{\nu} = \int_{E} \abs[Big]{\frac{d\nu}{d\tau}} \,d\tau
  \leq \tau(E)^{1/2} \ab\Big(\int_{E} \abs[Big]{\frac{d\nu}{d\tau}}^2 \,d\tau)^{1/2}
  \leq \tau(E)^{1/2} (2a)^{1/2},
\end{align}
which asserts that $\Lambda^{-1}([0, a])$ is uniformly bounded in the total variation norm.
Next, we establish the uniform $\tau$-continuity.
For any given $\epsilon > 0$, we choose $R = 4a/\epsilon$ and set $\delta = \epsilon / (2R)$.
For any Borel set $A \subset E$ satisfying $\tau(A) \leq \delta$ and any $\nu \in \Lambda^{-1}([0, a])$, we have
\begin{align}
  \abs{\nu(A)} = \abs[Big]{\int_{A} \frac{d\nu}{d\tau} \,d\tau}
  &\leq \int_{A} \abs[Big]{\frac{d\nu}{d\tau}} \ab\Big(\indicator[Big]{\abs[Big]{\frac{d\nu}{d\tau}} \leq R} + \indicator[Big]{\abs[Big]{\frac{d\nu}{d\tau}} > R}) \,d\tau \\
  &\leq R \tau(A) + \frac{1}{R} \int_{A} \abs[Big]{\frac{d\nu}{d\tau}}^2 \,d\tau
  \leq R \frac{\epsilon}{2R} + \frac{2a}{R}
  = \epsilon.
\end{align}
Hence, $\Lambda^{-1}([0, a])$ is uniformly $\tau$-continuous,
and is consequently relatively $\tau$-compact.
To demonstrate that $\Lambda^{-1}([0, a])$ is $\tau$-closed,
let $L^1 = L^1(E, \calB(E), \tau)$ denote the $L^1$-space on $E$ with respect to $\tau$, where $\calB(E)$ denotes the Borel $\sigma$-algebra of $E$.
Consider the set and the map
\begin{align}
  A_a \coloneqq \Bab{g \in L^1: \frac{1}{2} \int_{E} g^2 \,d\tau \leq a },\\
  \map{\phi}{A_a}{\Lambda^{-1}([0, a])},  \quad d\phi(g) \coloneqq g \, d\tau.
\end{align}
Here, $\phi$ is a homeomorphism from $A_a$ equipped with the weak topology $\sigma(L^1, L^{\infty})$ to $\Lambda^{-1}([0, a])$ endowed with the $\tau$-topology.
Since $A_a$ is convex and closed in the norm topology $\norm{\cdot}_{L^1}$,
it is also closed in the weak topology $\sigma(L^1, L^{\infty})$.
Consequently, $\Lambda^{-1}([0, a])$ is $\tau$-closed,
which, combined with the relative $\tau$-compactness, implies that $\Lambda^{-1}([0, a])$ is $\tau$-compact.

Finally, we verify \eqref{eq:proflim.eta.ratefunc}.
For any $\bfa \in \reals^M$ and $\nu \in \signedmeas(E) $ such that $\pi_{f_1, \dots, f_{M}}(\nu) = \bfu $ and $\Lambda(\nu) < \infty $, setting $g = \sum_{j=1}^M a_j f_j$ yields
\begin{align}
  \iprod{\bfa}{\bfu} - \frac{1}{2} \iprod{\bfa}{\Gamma_{f_1, \dots, f_M}' \bfa}
  &= \int_E g \,d\nu - \frac{1}{2} \int_E g^2 \,d\tau \\
  &\leq \int_{E}\frac{1}{2}\ab\Big(g^2 + \ab\Big(\frac{d\nu}{d\tau})^2) \,d\tau - \frac{1}{2} \int_E g^2 \,d\tau = \Lambda(\nu),
\end{align}
which leads to $ I_{f_1, \dots, f_M}'(\bfu) \leq \inf_{\nu} \Lambda(\nu)$.
Conversely, for any $\bfu \in \reals^M$ with $I_{f_1, \dots, f_M}'(\bfu) < \infty $, there exists $\bfa_{\bfu} \in \reals^M$ such that $\Gamma_{f_1, \dots, f_M}' \bfa_{\bfu} = \bfu$ and $I_{f_1, \dots, f_M}'(\bfu) = \iprod{\bfa_{\bfu}}{\bfu} /2$.
Defining $\nu$ via $d\nu \coloneqq (\sum_{j=1}^M a_{\bfu, j}f_j) d\tau$, we obtain $\bfu = \Gamma_{f_1, \dots, f_M}' \bfa_{\bfu} = \pi_{f_1, \dots, f_M}(\nu)$ and
\begin{align}
  I_{f_1, \dots, f_M}'(\bfu) = \frac{1}{2} \iprod{\bfa_{\bfu}}{\bfu}
  = \frac{1}{2} \sum_{j=1}^M a_{\bfu, j} \int_E f_j \,d\nu
  = \frac{1}{2} \int_E \ab\Big(\frac{d\nu}{d\tau})^2 \,d\tau
  = \Lambda(\nu).
\end{align}
This establishes \eqref{eq:proflim.eta.ratefunc}, completing the proof of \cref{prop:mdp.eta}.
\end{proof}

Next, we show the following exponential equivalence between $\{\mdpfactor^{-1} \centered{\xi_n}\}_{n \geq 1}$ and $\{\mdpfactor^{-1} \centered{\eta_n}\}_{n \geq 1}$.

\begin{lemma}
\label{prop:exp.eq.eta.xi}
For any $f \in \bddmble(E)$ and any $\delta > 0$,
\begin{align}
  \limsup_{n \to \infty} \frac{\ldpspeed}{\mdpfactor^2} \log \probab[big]{\abs{\pairing[none]{f}{\centered{\xi_n}} - \pairing{f}{\centered{\eta_n}}} > \delta \mdpfactor} = - \infty. \label{eq:exp.eq.eta.xi}
\end{align}
\end{lemma}

\begin{proof}
For simplicity, we give the proof for the case $d=2$.
The extension of this argument to a general dimension $d$ is straightforward by combining the techniques in \cite{ho2023ldp}.
Our proof fundamentally relies on the method from \cite{ho2023ldp}, but we carefully decompose the sum into disjoint components to utilize the linearity of centering, $\centered{Z_1 + Z_2} = \centered{Z_1} + \centered{Z_2} $.
Furthermore, as in \cite{ho2023ldp}, to simplify the exposition, we assume that $\unitcube $ can be partitioned exactly into the union of cubes $\{Q_{n,z}^{*}\}_{z \in I_n^{*}} $
so that $\unitcube = \bigcup_{z \in I_n^{*}} Q_{n,z}^{*} $ and $\card{I_n^{*}} = \ldpspeed $.

We write the difference as
\begin{align}
  &\pairing{f}{\xi_n} - \pairing{f}{\eta_n} \\
  &\quad = \sum_{\calY \subset \ppp_n}
  \indicator[bigg]{\begin{lgathered}\exists z_1 \neq z_2 \in I_n^{*} \text{ such that} \\ \calY \cap Q_{n, z_1}^{*} \neq \emptyset \text{ and }  \calY \cap Q_{n, z_2}^{*} \neq \emptyset\end{lgathered}}
   \isolated[n]{\calY}{\ppp_n}{t}  \zeroed{f}(H_n(\calY)) \\
  &\qquad - \sum_{z \in I_n^{*}} \sum_{\calY \subset \ppp_n \cap Q_{n,z}^{*}} \ab(\isolated[n]{\calY}{\ppp_n \cap Q_{n,z}^{*}}{t} - \isolated[n]{\calY}{\ppp_n}{t}) \zeroed{f}(H_n(\calY)) \\
  &\quad \eqqcolon Z_n^{(1)} - Z_n^{(2)}.
\end{align}
Thus, to prove \eqref{eq:exp.eq.eta.xi}, it suffices to show that for $i=1,2$,
\begin{align}
  \limsup_{n \to \infty} \frac{\ldpspeed}{\mdpfactor^2} \log \probab[Big]{\abs[Big]{\centered{Z_n^{(i)}}} \geq \delta \mdpfactor} = -\infty. \label{eq:eta.xi.block}
\end{align}

Since \eqref{eq:eta.xi.block} can be shown in a similar manner for both $i=1$ and $i=2$,
we only verify it for $i=1$.
For each $z \in I_n^{*}$, let
\begin{align}
  R_{1, z} &\coloneqq \ab\big([-Lr_n, Lr_n] \times [0, \ldpspeed^{-1/2}] + \ldpspeed^{-1/2}z) \cap \unitcube, \\
  R_{2, z} &\coloneqq \ab\big([0, \ldpspeed^{-1/2}] \times [-Lr_n, Lr_n] + \ldpspeed^{-1/2}z) \cap \unitcube, \\
  R_{3, z} &\coloneqq \ab\big([-Lr_n, Lr_n] \times [-Lr_n, Lr_n] + \ldpspeed^{-1/2}z) \cap \unitcube.
\end{align}
We also define $J_j \coloneqq \{z \in \integers^d : R_{j, z} \neq \emptyset\}$ and $R_j \coloneqq \bigcup_{z \in J_j} R_{j, z}$ for $j=1,2,3$, noting that $R_1 \cap R_2 = R_3 $.
If $\calY \subset \ppp_n $ satisfies the indicator conditions in $Z_n^{(1)}$,
then $\calY \subset \bigcup_{z \in I_n^{*}} (\partial Q_{n,z}^{*})^{(Lr_n)} $,
where $\partial Q_{n,z}^{*} $ denotes the boundary of $Q_{n,z}^{*} $.
This implies that $\calY \subset R_j $ for some $j \in \{1,2,3\} $.
We define
\begin{align}
  j(\calY) \coloneqq
  \begin{cases}
    1 & \text{ if } \calY \subset R_1 \text{ and } \calY \not\subset R_2,\\
    2 & \text{ if } \calY \not\subset R_1 \text{ and } \calY \subset R_2,\\
    3 & \text{ if } \calY \subset R_3.\\
  \end{cases}
\end{align}
If $j(\calY) = 1$, there exists a unique $z \in J_1$ such that $\calY \subset R_{1, z}$.
Indeed, otherwise $\calY \subset R_2$, which contradicts $j(\calY) = 1$.
This uniqueness property also holds for $j(\calY) = 2, 3$.

We decompose $Z_n^{(1)}$ as
\begin{align}
  Z_n^{(1)}
  &= \sum_{j=1}^3 \sum_{z \in J_j} \sum_{\calY \subset \ppp_n \cap R_{j, z}} 
  % \indicator[none]{\exists z_1 \neq z_2 \in I_n^{*} \text{ such that } \calY \cap Q_{n, z_1}^{*} \neq \emptyset \text{ and }  \calY \cap Q_{n, z_2}^{*} \neq \emptyset} 
  \indicator[bigg]{\begin{lgathered}\exists z_1 \neq z_2 \in I_n^{*} \text{ such that} \\ \calY \cap Q_{n, z_1}^{*} \neq \emptyset \text{ and }  \calY \cap Q_{n, z_2}^{*} \neq \emptyset\end{lgathered}}
  \isolated[n]{\calY}{\ppp_n}{t} \zeroed{f}(H_n(\calY)) \\
  &\eqqcolon \sum_{j=1}^3 \sum_{z \in J_j} W_{j, z}.
\end{align}
To prove \eqref{eq:eta.xi.block} for $i=1$, it is sufficient to show that for each $j=1, 2, 3$,
\begin{align}
  \limsup_{n \to \infty} \frac{\ldpspeed}{\mdpfactor^2} \log \probab[bigg]{\abs[bigg]{\sum_{z \in J_j} \centered{ W_{j,z} }} \geq \delta \mdpfactor} = -\infty.
\end{align}

By Chebyshev's inequality, for any $a > 0$, we have
\begin{align}
  \frac{\ldpspeed}{\mdpfactor^2} \log \probab[bigg]{\abs[bigg]{\sum_{z \in J_j} \centered{W_{j, z}}} \geq \delta \mdpfactor}
  \leq -a\delta + \frac{\ldpspeed}{\mdpfactor^2} \log \expect[bigg]{\exp\ab\bigg(a\frac{\mdpfactor}{\ldpspeed} \abs[bigg]{\sum_{z \in J_j} \centered{W_{j, z}}})}.
\end{align}
Let $J_{j, w} \coloneqq J_j \cap (2\integers^d + w) $ for $w \in \{0,1\}^d$.
H\"{o}lder's inequality together with the independence of $\{W_{j, z} \}_{z \in J_{j, w}} $ implies
\begin{align}
  \expect[Big]{\exp\ab\Big(a\frac{\mdpfactor}{\ldpspeed} \sum_{z \in J_j} \centered{W_{j, z}})}
  \leq \prod_{w \in \Bab{0, 1}^d} \prod_{z \in J_{j, w}} \expect[bigg]{\exp\ab\bigg(2^{d} a\frac{\mdpfactor}{\ldpspeed} \centered{W_{j, z}})}^{2^{-d}}.
\end{align}
Applying \cref{prop:cov} with $A=R_{j, z}$ and $A'=\unitcube$, 
we obtain $\expect[none]{(\centered{W_{j, z}})^2} = \lorder(\ldpspeed \lebmeas(R_{j, z})) = \lorder((\nrnd)^{k/d}) $ for all $j \in \{1,2,3\} $ and $z \in J_j$ (recalling $d=2$).
Furthermore, since $\expect[none]{\exp(\abs{W_{j, z}})} $ remains bounded as $n \to \infty$ (see p.4027 in \cite{ho2023ldp}), \cref{prop:exp.moment.finite} yields $\expect[none]{\exp\ab(2^d a \mdpfactor / \ldpspeed \centered{W_{j, z}})} - (1 + (2^d a \mdpfactor / \ldpspeed)^2 \expect[none]{(\centered{W_{j, z}})^2}) = \lorder\ab((\mdpfactor / \ldpspeed)^3) $, which also holds for $- W_{j, z} $.
Consequently, since $\card{J_{j, w}} \leq 2^d \ldpspeed $, we conclude that
\begin{align}
  \limsup_{n \to \infty} \frac{\ldpspeed}{\mdpfactor^2} \log \probab[bigg]{\abs[bigg]{\sum_{z \in J_j} \centered{W_{j, z}}} \geq \delta \mdpfactor}
  \leq -a\delta + \limsup_{n \to \infty} \frac{\ldpspeed}{\mdpfactor^2} \card{J_{j, w}} \ab\ab\Big(\frac{\mdpfactor}{\ldpspeed})^2 \ab\bigg(\lorder((\nrnd)^{k/d}) + \lorder\ab\Big(\frac{\mdpfactor}{\ldpspeed})) = -a\delta.
\end{align}
Taking the limit as $a \to \infty$ completes the proof of \eqref{eq:eta.xi.block} for $i=1$.
\end{proof}

Next, by using several techniques developed in \cite{ho2023ldp},
we establish a general bound for estimating various exponential moments.
We use the following bound
not only to prove exponential equivalence and justify the contraction argument,
but also to prove the LIL.
For this purpose, we introduce additional notation.
Let $\pppentire' $ be an independent copy of $\pppentire$. For a family of cubes $\{Q_z\}_{z \in \integers^d}$, we define
\begin{align}
  \pppentire_{z}'
  \coloneqq (\pppentire' \cap Q_z) \cup (\pppentire \setminus Q_z), 
  \quad \pppentire_{z, n}' \coloneqq \pppentire_{z}' \cap W_n, \label{eq:ppp.indep}
\end{align}
and
\begin{align}
  \calF_{z} \coloneqq
  % \sigma\ab\bigg(\bigcup_{w \in \integers^d \setminus \integers_{\geq 0}^d, w \lexineq z} (\pppentire \cap Q_{w}) \cup \bigcup \Bab{\pppentire \cap Q_{w} : w \in \integers_{\geq 0}^d, w_1 \leq z_1, \dots, w_d \leq z_d }),
  \sigma\ab\bigg(\pppentire \cap Q_{w} : \begin{aligned} & (w \in \integers^d \setminus \integers_{\geq 0}^d,\text{ and } w \lexineq z) \\ &\text{or } (w \in \integers_{\geq 0}^d \text{ and } w_j \leq z_j \text{ for all } 1 \leq j \leq d) \end{aligned}),
\end{align}
where $\integers_{\geq 0} $ denotes the set of all nonnegative integers.

\begin{proposition}
\label{prop:exp.moment.common}
Let $v_1, v_2, v_3 > 0$ satisfy $v_1 \leq v_2 \leq v_3 \leq 2v_1$,
and $L > t > 0 $.
Let $Q_z \coloneqq Q(v_1 t/\sqrt{d}, z) $ for $z \in \integers^d$.
Let $F(\calZ, \calX)$ be a functional defined on finite sets $\calX \subset \reals^d$ and $\calZ \subset \calX$ with $\card{\calZ} = k$.
% Assume that $F$ satisfies $F(\calZ + y, \calX + y) = F(\calZ, \calX)$ for all $y \in \reals^d $.
For a finite set $\calX \subset \reals^d$ and $z \in \integers^d$, let
\begin{align}
  S_z(\calX) \coloneqq \sum_{\calZ \subset \calX \cap Q_z^{(2Lv_3)}} \isolated{\calZ}{\calX \cap Q_z^{(3Lv_3)}}{v_2 t} F(\calZ, \calX \cap Q_z^{(3Lv_3)}).
\end{align}
Assume further that for every $p \geq 2$, there exists a constant $C_p  > 0$ such that for all $z \in \integers^d$,
\begin{align}
  \expect[Bigg]{\sum_{\calZ \subset \pppentire\cap Q_z^{(2Lv_3)}} \abs{F(\calZ, \pppentire \cap Q_{z}^{(3Lv_3)})}^p }
  \leq C_p. \label{eq:p.moment.boundness}
\end{align}
Then, for all $\lambda > 0$ and any finite set $I \subset \integers^d $,
\begin{align}
  \expect[bigg]{\exp\ab\bigg(\lambda \sum_{z \in I} \condexpect[none]{S_z(\pppentire) - S_z(\pppentire_{z}')}{\calF_z })}
  \leq \exp\ab\bigg(C \card{I} \sum_{p = 2}^\infty \frac{C_p (C\lambda)^p}{p!} ), \label{eq:exp.moment.block.mart.general}
\end{align}
where $C > 0$ is a constant depending only on $d, L, t$, and $k$.
\end{proposition}

\begin{proof}
Let $B_1 \coloneqq \roundup{12\sqrt{d}L/t}$.
We define $I_{w} \coloneqq I \cap (B_1\integers^d + w) $ for $w \in \{0,1,\dots, B_1-1\}^d$.
Since $\condexpect[none]{S_z(\pppentire) - S_z(\pppentire_z')}{\calF_z}$ is $\sigma (\pppentire \cap Q_z^{(3Lv_3)}) $-measurable,
the random variables $\{\condexpect[none]{S_z(\pppentire) - S_z(\pppentire_z')}{\calF_z} \}_{z \in I_{w}} $ are independent for each $w$.
Using H\"{o}lder's inequality, the independence of $\{\condexpect[none]{S_z(\pppentire) - S_z(\pppentire_z')}{\calF_z}\}_{z \in I_{w}} $, and Jensen's inequality for conditional expectations sequentially, we obtain
\begin{align}
  &\expect[Big]{\exp\ab\Big(\lambda \sum_{z \in I} \condexpect[none]{S_z(\pppentire) - S_z(\pppentire_{z}')}{\calF_z})} \\
  &\quad \leq \prod_{w \in \Bab{0, \dots, B_1-1}^d} \expect[Big]{\exp\ab\Big(B_1^d \lambda \sum_{z \in I_{w} } \condexpect[none]{S_z(\pppentire) - S_z(\pppentire_{z}')}{\calF_z})}^{B_1^{-d}} \\
  &\quad = \prod_{w \in \Bab{0, \dots, B_1-1}^d} \prod_{z \in I_{w}} \expect[Big]{\exp\ab\big(B_1^d \lambda \condexpect[none]{S_z(\pppentire) - S_z(\pppentire_{z}')}{\calF_z})}^{B_1^{-d}} \\
  &\quad \leq \prod_{w \in \Bab{0, \dots, B_1-1}^d} \prod_{z \in I_{w}} \expect[Big]{\exp\ab\big(B_1^d \lambda (S_z(\pppentire) - S_z(\pppentire_{z}')) )}^{B_1^{-d}}.
\end{align}

Now, fix $z \in \integers^d$. Since $\expect[none]{S_z(\pppentire) - S_z(\pppentire_{z}')} = 0 $, it follows that
\begin{align}
  \expect[none]{\exp\ab(B_1^d \lambda (S_z(\pppentire) - S_z(\pppentire_{z}')))} 
  &= 1 + \sum_{p =2}^\infty \frac{(B_1^d \lambda)^p}{p!} \expect[none]{(S_z(\pppentire) - S_z(\pppentire_{z}'))^p} \\
  &\leq 1 + \sum_{p = 2}^\infty \frac{(2B_1^d \lambda)^p}{p!} \expect[none]{\abs{S_z(\pppentire)}^p}.
\end{align}
Let $J \coloneqq \{z' \in \integers^d : Q_{\bfzero_d}^{(2Lv_3)} \cap Q_{z'} \neq \emptyset\} $.
Note that $\card{J}$ is bounded by a constant $B_2 > 0$ depending only on $d, L$, and $t$.
We can then bound $\abs{S_z(\pppentire)}$ by
\begin{align}
  \abs{S_z(\pppentire)}
  \leq \sum_{z_1', \dots, z_k' \in J} \sum_{\calZ \subset \pppentire \cap Q_z^{(2Lv_3)}}
  \isolated{\calZ}{\pppentire \cap Q_z^{(3Lv_3)}}{v_2 t}
  \prod_{i=1}^k \indicator[none]{\lex_i(\calZ) \in Q_{z + z_i'}}
  \abs{F(\calZ, \pppentire \cap Q_z^{(3Lv_3)})}.
\end{align}
Furthermore, there exists at most one subset $\calZ \subset \pppentire \cap Q_z^{(2Lv_3)}$ such that $\isolated{\calZ}{\pppentire \cap Q_z^{(3Lv_3)}}{v_2t} = 1$ and for all $i=1,\dots,k,$ $\indicator[none]{\lex_i(\calZ) \in Q_{z + z_i'}} = 1 $.
Thus, for $p \geq 2$, we have
\begin{align}
  &\expect[none]{\abs{S_z(\pppentire)}^p} \\
  &\quad \leq \begin{aligned}[t]
   \card{J}^{k(p-1)} \sum_{z_1', \dots, z_k' \in J} \expectsymbol\Biggl[\Biggl(\sum_{\calZ \subset \pppentire \cap Q_z^{(2Lv_3)}}
  &\isolated{\calZ}{\pppentire \cap Q_z^{(3Lv_3)}}{v_2t} \\
  &\quad \times \prod_{i=1}^k \indicator[none]{\lex_i(\calZ) \in Q_{z + z_i'}}
  \abs{F(\calZ, \pppentire \cap Q_z^{(3Lv_3)})} \Biggr)^p \Biggr] \end{aligned}\\
  &\quad = \begin{aligned}[t]
  \card{J}^{k(p-1)} \sum_{z_1', \dots, z_k' \in J} \expectsymbol\Biggl[ \sum_{\calZ \subset \pppentire \cap Q_z^{(2Lv_3)}}
  &\isolated{\calZ}{\pppentire \cap Q_z^{(3Lv_3)}}{v_2t} \\
  &\quad \times \prod_{i=1}^k \indicator[none]{\lex_i(\calZ) \in Q_{z + z_i'}}
  \abs{F(\calZ, \pppentire \cap Q_z^{(3Lv_3)})}^p \Biggr]\end{aligned}\\
  % &\quad \leq \card{J}^{k(p-1)} \sum_{z_1', \dots, z_k' \in J} \expect[Bigg]{\ab\bigg(\sum_{\calZ \subset \pppentire \cap Q_z^{(2Lv_3)}}
  % \isolated{\calZ}{\pppentire \cap Q_z^{(3Lv_3)}}{v_2t}
  % \prod_{i=1}^k \indicator[none]{\lex_i(\calZ) \in Q_{z + z_i'}}
  % \abs{F(\calZ, \pppentire \cap Q_z^{(3Lv_3)})})^p}\\
  % &\quad = \card{J}^{k(p-1)} \sum_{z_1', \dots, z_k' \in J} \expect[Bigg]{\sum_{\calZ \subset \pppentire \cap Q_z^{(2Lv_3)}}
  % \isolated{\calZ}{\pppentire \cap Q_z^{(3Lv_3)}}{v_2t}
  % \prod_{i=1}^k \indicator[none]{\lex_i(\calZ) \in Q_{z + z_i'}}
  % \abs{F(\calZ, \pppentire \cap Q_z^{(3Lv_3)})}^p}\\
  &\quad \leq \card{J}^{kp} \expect[Bigg]{\sum_{\calZ \subset \pppentire \cap Q_z^{(2Lv_3)}}
  \abs{F(\calZ, \pppentire \cap Q_z^{(3Lv_3)})}^p}
  \leq B_2^{kp} C_p, \label{eq:exp.moment.common.p.moment}
\end{align}
where the uniqueness of $\calZ $ is utilized in the third line.

Consequently, the left-hand side of \eqref{eq:exp.moment.block.mart.general} is bounded by
\begin{align}
  \prod_{w \in \Bab{0, \dots, B_1-1}^d} \ab\bigg(1 + \sum_{p = 2}^\infty \frac{(2B_1^d \lambda)^p}{p!} B_2^{kp} C_p)^{B_1^{-d}\card{I_{w}}}
  \leq \exp\ab\bigg(B_1^{-d} \card{I}  \sum_{p = 2}^\infty \frac{C_p (2B_1^d B_2^k \lambda)^p}{p!}), \label{eq:exp.moment.bound.common.lambda}
\end{align}
where we used $\sum_{w \in \Bab{0, \dots, B_1-1}^d} \card{I_{w}} = \card{I} $.
This proves \eqref{eq:exp.moment.block.mart.general}.
\end{proof}

\begin{remark}
(i)
\cref{prop:exp.moment.common} also holds 
when $\pppentire$ is replaced by $\pppentire \cap A$, 
where $A \subset \reals^d$ is any Borel set.
That is, 
if the condition
\begin{align}
  \expect[Bigg]{\sum_{\calZ \subset \pppentire \cap A \cap Q_z^{(2Lv_3)}} \abs{F(\calZ, \pppentire \cap A \cap Q_{z}^{(3Lv_3)})}^p } \leq C_p
\end{align}
is satisfied, 
then we have
\begin{align}
  \expect[bigg]{\exp\ab\bigg(\lambda \sum_{z \in I} \condexpect[none]{S_z(\pppentire \cap A) - S_z(\pppentire_{z}' \cap A)}{\calF_z })}
  \leq \exp\ab\bigg(C \card{I} \sum_{p = 2}^\infty \frac{C_p (C\lambda)^p}{p!} ).
\end{align}
Typically, 
we apply this result with $A = W_n$.

(ii)
In particular, 
if $I = I_n \coloneqq \{z \in \integers^d : W_n \cap Q_z \neq \emptyset \} $ for $n \geq 1$ 
and $v_1 \leq t d^{-1/2} $, 
then we obtain
\begin{align}
  \expect[bigg]{\exp\ab\bigg(\lambda \sum_{z \in I_n} \condexpect[none]{S_z(\pppentire_n) - S_z(\pppentire_{z, n}')}{\calF_z })}
  \leq \exp\ab(C n v_1^{-d} \sum_{p=2}^\infty \frac{C_p (C\lambda)^p}{p!} ). \label{eq:exp.moment.block.mart}
\end{align}
Indeed, 
\eqref{eq:exp.moment.block.mart} follows from the bound 
$\card{I_n} = \roundup{n^{1/d} / v_1 t d^{-1/2}}^d \leq (2 n^{1/d} / v_1 t d^{-1/2})^d $ 
and \eqref{eq:exp.moment.bound.common.lambda}.
\end{remark}

Applying \cref{prop:exp.moment.common},
we establish the exponential equivalence between $\{\mdpfactor^{-1} \centered{\xi}_n\}_{n \geq 1}$ and $\{\mdpfactor^{-1} \centered{\kappa_n}\}_{n \geq 1}$.

\begin{lemma}
\label{prop:exp.eq.xi.kappa}
Let $f \in \bddmble(E)$ or $f = \iprod{\bfa}{\cdot}$ for some $\bfa = (a_1, \dots, a_m) \in \reals^m$.
Let $0 < t_1 \leq \dots \leq t_m$ denote the time parameters for $\{\kappa_n \}_{n \geq 1}$,
and let $t = t_1$ be the time parameter for $\{\xi_n \}_{n \geq 1}$,
and assume that $L > t_m $.
Then, there exists a constant $C > 0$, independent of $n$ and $f$,
such that for all sufficiently large $n$ and all $R > 0$,
\begin{align}
  \probab[big]{\abs{\pairing{f}{\centered{\xi_n}} - \pairing{f}{\centered{\kappa_n}}} > R}
  \leq 2\exp\ab\Big(- C C_f \nrnd \ldpspeed \tailfunc\ab\Big(\frac{R}{C^2 C_f \nrnd \ldpspeed})), \label{eq:xi.kappa.diff.tail}
\end{align}
where $\tailfunc$ is given by \eqref{eq:tailfunc}, and $C_f$ is a constant defined by
\begin{align}
  C_f \coloneqq
  \begin{cases}
    e^{2\supnorm{f}} \tau(E) & \text{if } f \in \bddmble(E), \\
    \displaystyle
    \int_E \exp\ab\bigg({\sum_{i=1}^m \abs{a_i} u_i}) \,\tau(d\bfu) & \text{if } f = \iprod{\bfa}{\cdot}. \label{eq:xi.kappa.tail.const.f}
  \end{cases}
\end{align}
In particular, for any $\delta > 0$,
\begin{align}
  \limsup_{n \to \infty} \frac{\ldpspeed}{\mdpfactor^2} \log \probab[big]{\abs{\pairing{f}{\centered{\xi_n}} - \pairing{f}{\centered{\kappa_n}}} > \delta \mdpfactor} = - \infty. \label{eq:xi.kappa.exp.eq}
\end{align}
\end{lemma}

\begin{proof}
Note that $\pairing{f}{\xi_n}$ can be expressed as
\begin{align}
  \pairing{f}{\xi_n}
  = \sum_{\calZ \subset \pppentire_n} \isolated{\calZ}{\pppentire_n}{(\nrnd)^{1/d} t_1} \zeroed{f}(H((\nrnd)^{-1/d}\calZ)).
\end{align}
Similarly, by setting $v = v_1 = v_2 = v_3 \coloneqq (\nrnd)^{1/d} $, $Q_z \coloneqq Q(v_1 t_1 / \sqrt{d}, z) $ and 
\begin{align}
  F(\calZ, \calX)
  &\coloneqq \zeroed{f}(H(v^{-1}\calZ)) - \zeroed{f}(G(v^{-1}\calZ, v^{-1}\calX; \bft)),
\end{align}
we can rewrite the difference $\pairing[none]{f}{\centered{\xi_n}} - \pairing{f}{\centered{\kappa_n}}$ in the form $\sum_{z \in I_n} \condexpect[none]{S_z(\pppentire_n) - S_z(\pppentire_{z,n}')}{\calF_z}$ using the notation from \cref{prop:exp.moment.common},
which is a decomposition also employed in \cite{py2001clt,krebs2021lilsip}.

We now check the condition \eqref{eq:p.moment.boundness} for $p \geq 2$ to apply \cref{prop:exp.moment.common}. 
Observe that $\abs{F(\calZ, \pppentire_n)}$ is bounded by
\begin{align}
  \indicator[none]{H(v^{-1} \calZ) \neq \bfzero_m} (1 - \isolated{\calZ}{\pppentire_n \cap Q_{z}^{(3Lv)}}{v t_m}) \abs{f(H(v^{-1} \calZ)) - f(G(v^{-1}\calZ, v^{-1}\pppentire_n; \bft))}.
\end{align}
When $f$ is bounded, the inequality $\abs{f(H(v^{-1} \calZ)) - f(G(v^{-1}\calZ, v^{-1}\pppentire_n; \bft))} \leq 2\supnorm{f}$ holds. Thus, the left-hand side of \eqref{eq:p.moment.boundness} is bounded by
\begin{align}
  &\expect[Bigg]{\sum_{\calZ \subset \pppentire_n \cap Q_z^{(2Lv)}} \indicator[none]{H(v^{-1} \calZ) \neq \bfzero_m} (1 - \isolated{\calZ}{\pppentire_n \cap Q_{z}^{(3Lv)}}{v t_m}) (2\supnorm{f})^p } \\
  &\quad= \frac{(2\supnorm{f})^p}{k!} \int_{(W_n \cap Q_z^{(2Lv)})^{k}} \indicator[none]{H(v^{-1} \bfx) \neq \bfzero_m}
  \expect[none]{1 - \isolated{\bfx}{\pppentire_n \cap Q_{z}^{(3Lv)}}{v t_m}} \,d\bfx \\
  &\quad\leq \frac{p! e^{2\supnorm{f}}}{k!} \int_{(Q_z^{(2Lv)})^{k}} \indicator[none]{H(v^{-1} \bfx) \neq \bfzero_m}
  \expect[none]{1 - \isolated{\bfx}{\pppentire_n \cap Q_{z}^{(3Lv)}}{v t_m}}\,d\bfx , \label{eq:xi.kappa.pth.moment.bdd}
\end{align}
where we have used the elementary inequality $x^p \leq p! e^x$ for $x > 0$.
When $f = \iprod{\bfa}{\cdot} $, we have
\begin{align}
  \abs{f(H(v^{-1} \calZ)) - f(G(v^{-1}\calZ, v^{-1}\pppentire_n; \bft))}
  &= \abs{\sum_{i=1}^m a_i \ab(1 - \isolated{\calZ}{\pppentire_n \cap Q_{z}^{(3Lv)}}{v t_i})h^{(i)}(v^{-1} \calZ) } \\
  &\leq \sum_{i=1}^m \abs{a_i} h^{(i)}(v^{-1} \calZ) 
  = \iprod{\bfa'}{H(v^{-1} \calZ)},
\end{align}
where $\bfa' \coloneqq (\abs{a_1}, \dots, \abs{a_m})$.
In this case, the left-hand side of \eqref{eq:p.moment.boundness} is bounded by
\begin{align}
  &\expect[Bigg]{\sum_{\calZ \subset \pppentire_n \cap Q_z^{(2Lv)}} (1 - \isolated{\calZ}{\pppentire_n \cap Q_{z}^{(3Lv)}}{v t_m}) \iprod{\bfa'}{H(v^{-1} \calZ)}^p } \\
  &\quad = \frac{1}{k!} \int_{(W_n \cap Q_z^{(2Lv)})^{k}}
  \iprod{\bfa'}{H(v^{-1} \bfx)}^p
  \expect[none]{1 - \isolated{\bfx}{\pppentire_n \cap Q_{z}^{(3Lv)}}{v t_m}} \,d\bfx \\
  &\quad \leq \frac{p!}{k!} \int_{(Q_z^{(2Lv)})^{k}}
  e^{\iprod{\bfa'}{H(v^{-1} \bfx)}} \indicator[none]{H(v^{-1} \bfx) \neq \bfzero_m}
  \expect[none]{1 - \isolated{\bfx}{\pppentire_n \cap Q_{z}^{(3Lv)}}{v t_m}} \,d\bfx .\label{eq:xi.kappa.pth.moment.linear}
\end{align}
Since the relation
\begin{align}
  &\expect[none]{1 - \isolated{\bfx}{\pppentire_n \cap Q_{z}^{(3Lv)}}{v t_m}} 
  = 1 - \exp\ab\big(-\lebmeas(W_n \cap Q_{z}^{(3Lv)} \cap \ball{\bfx}{v t_m})) \\
  &\quad \leq \lebmeas\ab\big(W_n \cap Q_{z}^{(3Lv)} \cap \ball{\bfx}{v t_m})
  \leq k \lebmeas(\ball{\bfzero_d}{1}) (v t_m)^d
\end{align}
holds for all $\bfx = (x_1, \dots, x_k) \in (\reals^d)^k$,
it follows that both \eqref{eq:xi.kappa.pth.moment.bdd} and \eqref{eq:xi.kappa.pth.moment.linear} are bounded by
\begin{align}
  k \lebmeas(\ball{\bfzero_d}{1}) (v t_m)^d v^{d(k-1)} \lebmeas(Q_{z}^{(2Lv)}) p! C_f
  = p! C C_f v^{d(k+1)},
\end{align}
where $C > 0$ is a constant depending only on $d, k, L, t_1$, and $t_m$,
and $C_f$ is given by \eqref{eq:xi.kappa.tail.const.f}.

Applying \cref{prop:exp.moment.common}, for $n $ such that $v = (\nrnd)^{1/d} < td^{-1/2} $, we obtain
\begin{align}
  \probab[big]{\abs{\pairing{f}{\centered{\xi_n}} - \pairing{f}{\centered{\kappa_n}}} > R} 
  &\leq \probab[big]{\pairing{f}{\centered{\xi_n}} - \pairing{f}{\centered{\kappa_n}} > R} + \probab[big]{\pairing{-f}{\centered{\xi_n}} - \pairing{-f}{\centered{\kappa_n}} > R}\\
  &\leq 2e^{- \lambda R} \exp\ab\bigg(C n v^{-d} \sum_{p=2}^\infty C C_f v^{d(k+1)} (C\lambda)^p)  \\
  &\leq 2e^{- \lambda R} \exp\ab(CC_f \nrnd \ldpspeed \frac{C^2 \lambda^2}{1 - C\lambda}), \label{eq:xi.kappa.exp.eq.tail.lambda}
\end{align}
where the last inequality holds for any $0 < \lambda < C^{-1}$.
Optimizing with respect to $\lambda$,
we arrive at \eqref{eq:xi.kappa.diff.tail}.

Furthermore, for any $a > 0 $, by setting $R = \delta \mdpfactor$ and $\lambda = a\mdpfactor/\ldpspeed $ in \eqref{eq:xi.kappa.exp.eq.tail.lambda}, we find
\begin{align}
  \limsup_{n \to \infty} \frac{\ldpspeed}{\mdpfactor^2} \log \probab[big]{\abs{\pairing{f}{\centered{\xi_n}} - \pairing{f}{\centered{\kappa_n}}} > \delta \mdpfactor} 
  &\leq -a\delta + \limsup_{n \to \infty} \frac{\ldpspeed}{\mdpfactor^2} CC_f \nrnd \ldpspeed \frac{C^2 a^2 \mdpfactor^2/\ldpspeed^2}{1 - C a \mdpfactor/\ldpspeed}
  = -a\delta.
\end{align}
Letting $a \to \infty$ yields \eqref{eq:xi.kappa.exp.eq}.
\end{proof}

At the end of this subsection,
we prepare to apply an extension of the contraction principle
(see, e.g., Theorem 4.2.23 in \cite{dz2009ldp}) to the sequence $\{\mdpfactor^{-1} \centered{\kappa_n}\}_{n \geq 1}$.
Let $\map{\Phi}{\signedmeas(E)}{\reals^m}$ be defined by
\begin{align}
  \Phi(\nu) \coloneqq
  \begin{cases}
    \displaystyle
    \ab(\int_{E} u_i \,\nu(d\bfu))_{i=1,\dots,m} & \displaystyle \text{for all } i=1,\dots, m,  \int_{E} u_i \,\abs{\nu}(d\bfu) < \infty,\\
    \bfzero_m & \text{otherwise}.
  \end{cases}
  \label{eq:map.contract.meas.to.vec}
\end{align}
Note that $\Phi(\kappa_n) = T_n$.
Since $\Phi$ is not continuous on $\signedmeas(E)$,
we cannot directly apply the classical contraction principle.
To circumvent this issue, we introduce a continuous cutoff approximation $\Phi_K$ of $\Phi$.
For $K > 0$, we define
\begin{align}
  \Phi_K(\nu) \coloneqq \ab(\int_{E} \phi_K(u_i) \,\nu(d\bfu))_{i=1,\dots,m},
\end{align}
where
\begin{align}
  \phi_K(x) =
  \begin{cases}
      x & \text{ if } 0 \leq x < K, \\
      -K^2(x-K) + K & \text{ if } K \leq x \leq K + K^{-1},\\
      0 & \text{otherwise}.
  \end{cases}
\end{align}

The following lemma establishes the exponential equivalence between $\Phi(\mdpfactor^{-1} \centered{\kappa_n})$ and $\Phi_K(\mdpfactor^{-1} \centered{\kappa_n})$,
as well as the uniformity on level sets required to invoke the extension of the contraction principle.

\begin{lemma}
\label{lem:contraction_cutoff}

\textup{(i)} For any $\delta > 0$,
\begin{align}
  \limsup_{K \to \infty} \limsup_{n \to \infty} \frac{\ldpspeed}{\mdpfactor^2} \log \probab[big]{\norm{\Phi(\mdpfactor^{-1} \centered{\kappa_n}) - \Phi_K(\mdpfactor^{-1} \centered{\kappa_n})}_{\reals^m} \geq \delta } = -\infty. \label{eq:cutoff.exp}
\end{align}

\textup{(ii)} For any $a > 0$,
\begin{align}
  \limsup_{K \to \infty} \sup_{\nu \in \Lambda^{-1}([0, a])} \norm{\Phi(\nu) - \Phi_K(\nu)}_{\reals^m} = 0. \label{eq:cutoff.sup}
\end{align}

\textup{(iii)} For any $\bfu \in \reals^m$,
\begin{align}
  I(\bfu) = \inf \Bab{\Lambda(\nu) : \nu \in \signedmeas(E), \, \Phi(\nu) = \bfu}. \label{eq:rate.func.contractio.principle}
\end{align}

\end{lemma}

\begin{proof}
Let $\varphi_K(x) \coloneqq x - \phi_K(x) $. Note that $\varphi_K(x) \leq \indicator[none]{x > K}x$ and $\varphi_K(cx) = c\varphi_K(x)$ for $c \in \{0, 1\}$.

(i) Since the bound
\begin{align}
  \norm{\Phi(\centered{\kappa_n}) - \Phi_K(\centered{\kappa_n})}_{\reals^m}
  \leq \sum_{i=1}^m \abs{\centered{\sum_{\calY \subset \ppp_n} \isolated[n]{\calY}{\ppp_n}{t_i} \varphi_K(h_n^{(i)}(\calY)) }}
\end{align}
holds, to prove \eqref{eq:cutoff.exp} it is enough to show that for each $i=1,\dots, m$,
\begin{align}
  \limsup_{K \to \infty} \limsup_{n \to \infty} \frac{\ldpspeed}{\mdpfactor^2} \log \probab[Bigg]{\abs[bigg]{\centered{\sum_{\calY \subset \ppp_n} \isolated[n]{\calY}{\ppp_n}{t_i} \varphi_K(h_n^{(i)}(\calY)) }} \geq \delta \mdpfactor} = -\infty. \label{eq:cutoff.exp.neg.element}
\end{align}

% Without loss of generality, we consider only the case $i=1$.
Setting $v = v_1 = v_2 = v_3 \coloneqq (\nrnd)^{1/d}$, $t \coloneqq t_i$ and $F(\calZ) \coloneqq \varphi_K(h^{(i)}(v^{-1}\calZ))$, we can express the centered sum as
\begin{align}
  \centered{\sum_{\calY \subset \ppp_n} \isolated[n]{\calY}{\ppp_n}{t_i} \varphi_K(h_n^{(i)}(\calY))} 
  &= \centered{\sum_{\calZ \subset \pppentire_n} \isolated{\calZ}{\pppentire_n}{v t_i} F(\calZ)}
  = \sum_{z \in I_n} \condexpect[none]{S_z(\pppentire_n) - S_z(\pppentire_{z,n}')}{\calF_z}.
\end{align}
For $p \geq 2$, the inequality
\begin{align}
  \expect[Bigg]{\sum_{\calZ \subset \pppentire_n \cap Q_z^{(2Lv)}} \abs{F(\calZ)}^p } 
  % \leq \frac{1}{k!} \int_{(Q_z^{(2Lv)})^k}  \varphi_K(h_n^{(i)} (\bfx))^p d\bfx \\
  &\leq v^{d(k-1)} \lebmeas(Q_{z}^{(2Lv)}) \frac{1}{k!} \int_{(\reals^d)^{k-1}} \varphi_K(h^{(i)} (\bfzero, \bfy))^p \,d\bfy \\
  &\leq p! (L+2t_i)^d v^{dk} \int_{(\reals^d)^{k-1}} e^{h^{(i)}(\bfzero_d, \bfy)} \indicator[none]{h^{(i)}(\bfzero_d, \bfy) > K} \,d\bfy
\end{align}
holds.
Thus, \cref{prop:exp.moment.common} implies that
\begin{align}
  \expect[bigg]{\exp\ab\bigg(\lambda \centered{\sum_{\calY \subset \ppp_n} \isolated[n]{\calY}{\ppp_n}{t_i} \varphi_K(h_n^{(i)}(\calY))} )}
  \leq \exp\ab\bigg(C C_{i, K} \ldpspeed \frac{C^2 \lambda^2}{1 - C\lambda}),
\end{align}
where $0 < \lambda < C^{-1}$, $C > 0$ is a constant independent of $n$ and $K$, 
and $C_{i, K}$ is the constant defined by
\begin{align}
  C_{i, K} \coloneqq \int_{(\reals^d)^{k-1}} e^{h^{(i)}(\bfzero_d, \bfy)} \indicator[none]{h^{(i)}(\bfzero_d, \bfy) > K} \,d\bfy.
\end{align}
Note that by condition (H4), $C_{i, K}$ is finite and vanishes as $K \to \infty$.
Therefore, we obtain for all sufficiently large $n$,
\begin{align}
  \probab[Bigg]{\abs[bigg]{\centered{\sum_{\calY \subset \ppp_n} \isolated[n]{\calY}{\ppp_n}{t_i} \varphi_K(h_n^{(i)}(\calY)) }} \geq \delta \mdpfactor}
  \leq 2\exp\ab\Big(- C C_{i,K} \ldpspeed \tailfunc\ab\Big(\frac{\delta \mdpfactor}{C^2 C_{i,K} \ldpspeed})),
\end{align}
and
\begin{align}
  &\limsup_{K \to \infty} \limsup_{n \to \infty} \frac{\ldpspeed}{\mdpfactor^2} \log \probab[Bigg]{\centered{\sum_{\calY \subset \ppp_n} \isolated[n]{\calY}{\ppp_n}{t_i} \varphi_K(h_n^{(i)}(\calY)) } \geq \delta \mdpfactor} \\
  &\quad \leq \limsup_{K \to \infty} \limsup_{n \to \infty} \frac{\ldpspeed}{\mdpfactor^2} \ab\bigg(-C C_{i,K} \ldpspeed \frac{1}{4}\ab\Big(\frac{\delta \mdpfactor}{C^2 C_{i,K} \ldpspeed})^2 )
  = \limsup_{K \to \infty} (- \delta^2 4^{-1} C^{-3} C_{i,K}^{-1})
  = - \infty.
\end{align}

(ii) Fix $R > 0$.
For any $\nu \in \Lambda^{-1}([0, a])$, we have
\begin{align}
  \norm{\Phi(\nu) - \Phi_K(\nu)}_{\reals^m}
  &\leq \sum_{i=1}^m \int_{E} \varphi_K(u_i) \,\abs{\nu} (d\bfu) \\
  &\leq \sum_{i=1}^m \int_{E} \indicator[none]{u_i \geq K} u_i
  \ab\Big(\indicator[Big]{\abs[Big]{\frac{d\nu}{d\tau}(\bfu)} \geq R u_i} + \indicator[Big]{\abs[Big]{\frac{d\nu}{d\tau}(\bfu)} < R u_i})
  \abs[Big]{\frac{d\nu}{d\tau}(\bfu)} \,\tau(d\bfu) \\
  &\leq \frac{2am}{R} + R\sum_{i=1}^m \int_{E} \indicator[none]{u_i \geq K} u_i^2 \,\tau(d\bfu).
\end{align} 
Under condition (H4), the integral $\int_{E} \indicator[none]{u_i \geq K} u_i^2 \tau(d\bfu)$ is finite and vanishes as $K \to \infty$. Consequently,
\begin{align}
  \limsup_{K \to \infty} \sup_{\nu \in \Lambda^{-1}([0, a])}
  \norm{\Phi(\nu) - \Phi_K(\nu)}_{\reals^m}
  \leq \frac{2am}{R}.
\end{align}
Taking the limit as $R \to \infty$ yields \eqref{eq:cutoff.sup}.

(iii)
The proof follows arguments similar to those in the proof of \eqref{eq:proflim.eta.ratefunc} and is therefore omitted.
\end{proof}

\subsection{The binomial case}

In this subsection, we establish the exponential equivalence between $\{\mdpfactor^{-1} \centered{\kappa_n(\ppp_n)}\}_{n \geq 1}$
and $\{\mdpfactor^{-1} \centered{\kappa_n(\bpp_n)}\}_{n \geq 1}$,
as well as between $\{\mdpfactor^{-1} \centered{T_n(\ppp_n)}\}_{n \geq 1}$ and $\{\mdpfactor^{-1} \centered{T_n(\bpp_n)}\}_{n \geq 1}$.
Without loss of generality, we assume that the time parameters $t_1,\ldots, t_m$
for $\{\kappa_n(\ppp_n)\}_{n \geq 1}$ and $\{\kappa_n(\bpp_n)\}_{n \geq 1}$
satisfy $0 < t_1 \leq \dots \leq t_m < \infty $.
For convenience,
we adopt the convention that $t_{m+1} \coloneqq \infty$
and $\isolated[n]{\calY}{\calX}{t_{m+1}} = 0$ for all finite sets $\calY, \calX \subset \reals^d$.

To establish the exponential equivalences,
we begin by estimating the difference between the mean measures of $\kappa_n(\mathcal P_n)$ and $\kappa_n(\mathcal B_n)$.
The following proposition allows us to reduce the estimate
for the difference between the centered versions to that for the corresponding uncentered versions.

\begin{proposition}
\label{prop:ppp.bpp.expect.vanish}
Let $f \in \bddmble(E)$ or
$f = \iprod{\bfa}{\cdot}$ for some $\bfa = (a_1, \dots, a_m) \in \reals^m$.
Then, 
there exists a constant $C > 0$ independent of $n$ and $f$ such that for all $n \geq 1$,
\begin{align}
  \abs[big]{\expect[none]{\pairing{f}{\kappa_n(\ppp_n)}} - \expect[none]{\pairing{f}{\kappa_n(\bpp_n)}}} 
  \leq C C_f \ab\big((\nrnd)^{k-1} + (\nrnd)^k (\nrnd + 1)) , \label{eq:kappa.ppp.bpp.expect.bound}
\end{align}
where
\begin{align}
  C_f \coloneqq \int_{E} g_f \, d\tau, 
  \quad g_f(\bfu) \coloneqq \sum_{i=1}^m \abs{f(u_1, \dots, u_i, \bfzero_{m-i})}, \quad \bfu = (u_1, \dots, u_m) \in E. \label{eq:kappa.f.projection.sum}
\end{align}
In particular,
\begin{align}
  \lim_{n \to \infty} \ab\big(\expect[none]{\pairing{f}{\kappa_n(\ppp_n)}} - \expect[none]{\pairing{f}{\kappa_n(\bpp_n)}}) = 0. \label{eq:kappa.ppp.bpp.expect.vanish}
\end{align}
\end{proposition}

\begin{proof}
Let $H_n^{(i)}(\calY) \coloneqq ( h_n^{(1)}(\calY), \dots, h_n^{(i)}(\calY)) \in [0, \infty)^i $,
and observe that
\begin{align}
  \kappa_n(\calX)
  &= \sum_{\calY \subset \calX, \card{\calY}=k} \indicator[none]{G_n(\calY,\calX;\bft) \neq \bfzero_{m}} \diracdelta{G_n(\calY,\calX;\bft)} \\
  &= \sum_{i=1}^m \sum_{\calY \subset \calX, \card{\calY}=k} (\isolated[n]{\calY}{\calX}{t_{i}} - \isolated[n]{\calY}{\calX}{t_{i+1}}) \indicator[none]{H_n^{(i)}(\calY) \neq \bfzero_{i}} \diracdelta{(H_n^{(i)}(\calY), \bfzero_{m-i})}. \label{eq:kappa.decompose}
\end{align}

From the Mecke formula, 
we can write
\begin{align}
  \expect[none]{\pairing{f}{\kappa_n(\ppp_n)}}
  &= \sum_{i=1}^m \expect[Bigg]{ \sum_{\calY \subset \ppp_n, \card{\calY}=k} (\isolated[n]{\calY}{\ppp_n}{t_{i}} - \isolated[n]{\calY}{\ppp_n}{t_{i+1}}) \zeroed{f}(H_n^{(i)}(\calY), \bfzero_{m-i})}\\
  &= \sum_{i=1}^m \frac{n^k}{k!} \int_{(\unitcube)^k} d\bfx\, \zeroed{f}(H_n^{(i)}(\bfx), \bfzero_{m-i})
   \expect[none]{ \isolated[n]{\bfx}{\ppp_n}{t_{i}} - \isolated[n]{\bfx}{\ppp_n}{t_{i+1}} } \\
  &= \sum_{i=1}^m \frac{n^k}{k!} \int_{(\unitcube)^k} d\bfx\, \zeroed{f}(H_n^{(i)}(\bfx), \bfzero_{m-i})\\
  &\quad\times \ab\Big(\exp\ab\big(- n \lebmeas(\unitcube \cap \ball{\bfx}{r_n t_i})) - \indicator{i < m}\exp\ab\big(- n \lebmeas(\unitcube \cap \ball{\bfx}{r_n t_{i+1}}))).
\end{align}
Similarly, 
for the binomial point process, 
we have
\begin{align}
  \expect[none]{\pairing{f}{\kappa_n(\bpp_n)}}
  &= \sum_{i=1}^m \expect[Bigg]{ \sum_{\calY \subset \bpp_n, \card{\calY}=k} (\isolated[n]{\calY}{\bpp_n}{t_{i}} - \isolated[n]{\calY}{\bpp_n}{t_{i+1}}) \zeroed{f}(H_n^{(i)}(\calY), \bfzero_{m-i})}\\
  &= \sum_{i=1}^m \frac{n!}{k!(n-k)!} \expect[bigg]{ (\isolated[n]{\bpp_k}{\bpp_n}{t_{i}} - \isolated[n]{\bpp_k}{\bpp_n}{t_{i+1}}) \zeroed{f}(H_n^{(i)}(\bpp_k), \bfzero_{m-i})}\\
  &= \sum_{i=1}^m \frac{n!}{k!(n-k)!} \int_{(\unitcube)^k} d\bfx\, \zeroed{f}(H_n^{(i)}(\bfx), \bfzero_{m-i})\\
  & \quad \times \ab\Big(\lebmeas\ab\big(\unitcube \setminus \ball{\bfx}{r_n t_i})^{n-k} - \indicator{i < m}\lebmeas\ab\big(\unitcube \setminus \ball{\bfx}{r_n t_{i+1}})^{n-k}).
\end{align}
Thus, 
the difference $\abs[big]{\expect[none]{\pairing{f}{\kappa_n(\ppp_n)}} - \expect[none]{\pairing{f}{\kappa_n(\bpp_n)}}}$ is bounded by
\begin{align}
\begin{split}
  &\sum_{i=1}^m \frac{1}{k!} \ab\Big(n^k - \frac{n!}{(n-k)!}) \int_{(\unitcube)^k} d\bfx\, \abs{\zeroed{f}(H_n^{(i)}(\bfx), \bfzero_{m-i})}\\
  &+ \sum_{i=1}^m \frac{n^k}{k!} \int_{(\unitcube)^k} d\bfx\, \abs{\zeroed{f}(H_n^{(i)}(\bfx), \bfzero_{m-i})}\\
  &\quad \times \Bigl(\abs[big]{\exp\ab\big(- n \lebmeas(\unitcube \cap \ball{\bfx}{r_n t_i})) - \lebmeas\ab\big(\unitcube \setminus \ball{\bfx}{r_n t_i})^{n-k}} \Bigr.\\
  &\qquad \bigl. + \indicator{i < m}\abs[big]{\exp\ab\big(- n \lebmeas(\unitcube \cap \ball{\bfx}{r_n t_{i+1}})) - \lebmeas(\unitcube \setminus \ball{\bfx}{r_n t_{i+1}})^{n-k}}\Bigr).
\end{split} \label{eq:kappa.ppp.bpp.expect.bound.halfway}
\end{align}

Since $ n^k - n!/(n-k)! \leq k^2 n^{k-1}$ and
\begin{align}
  \int_{(\unitcube)^k} \abs{\zeroed{f}(H_n^{(i)}(\bfx), \bfzero_{m-i})} \,d\bfx
  \leq r_n^{d(k-1)} \int_{(\reals^d)^{k-1}} \abs{\zeroed{f}(H^{(i)}(\bfzero_d, \bfy), \bfzero_{m-i})} \,d\bfy, \label{eq:kappa.ppp.bpp.int.f.order}
\end{align}
the first term of \eqref{eq:kappa.ppp.bpp.expect.bound.halfway} is bounded by
$C C_f (\nrnd)^{k-1}$.
Next, 
we estimate the second term of \eqref{eq:kappa.ppp.bpp.expect.bound.halfway}.
For each $i=1, \dots, m$, 
let $\Delta_{n,i,\bfx} \coloneqq \lebmeas(\unitcube \cap \ball{\bfx}{r_n t_i}) $.
Note that $\Delta_{n, i, \bfx} \leq k \lebmeas(\ball{\bfzero_d}{r_n t_i}) \leq k \lebmeas(\ball{\bfzero_d}{1}) (r_n t_m)^d $ holds uniformly in $\bfx$.
By the elementary inequality $\abs{e^{-x} - (1 - x)} \leq x^2 $ 
and the fact that $\abs{x^p - y^p} \leq p\abs{x-y} $ for any $x, y \in (0, 1) $ and $p \geq 1$, 
it follows that
\begin{align}
  &\abs{\exp(- n \lebmeas(\unitcube \cap \ball{\bfx}{r_n t_i}))
   - \lebmeas(\unitcube \setminus \ball{\bfx}{r_n t_i})^{n-k}} \\
  &\quad = \abs{\exp(-n\Delta_{n, i, \bfx}) - (1 - \Delta_{n, i, \bfx})^{n-k}}
  \leq n \Delta_{n,i,\bfx}^2 + k \Delta_{n,i,\bfx}
  \leq C r_n^d (\nrnd + 1)
\end{align}
for $i=1, \dots, m $.
Thus, 
combining this with \eqref{eq:kappa.ppp.bpp.int.f.order}, 
we can bound the second term in \eqref{eq:kappa.ppp.bpp.expect.bound.halfway} by
$C C_f (\nrnd)^{k} (\nrnd + 1)$.
Therefore, 
we obtain the desired bound \eqref{eq:kappa.ppp.bpp.expect.bound}.
Consequently, 
as $n \to \infty$, 
we have 
$\expect[none]{\pairing{f}{\kappa_n(\ppp_n)}} - \expect[none]{\pairing{f}{\kappa_n(\bpp_n)}}
  = \lorder\ab\big((\nrnd)^{k-1}) \to 0$, 
which yields \eqref{eq:kappa.ppp.bpp.expect.vanish}.
\end{proof}

To establish the exponential equivalence between the Poisson and binomial models,
we introduce auxiliary Poisson point processes.
For each $l \geq 1$,
let $N_{2n, l}'$ be a Poisson random variable with mean $2n$,
and let $X_{n, l, 1}'', X_{n, l, 2}'', \dots$ be independent and identically distributed random vectors uniformly distributed on $\unitcube$,
both independent of $\{\bpp_n\}_{n \geq 1}$ and $\{\ppp_n\}_{n \geq 1}$.
We define
\begin{align}
  \ppp_{2n, l}' \coloneqq \Bab{X_i: i \in \Bab{1, \dots, n \land N_{2n, l}'} } \cup \Bab{X_{n, l, i}'' : i \in \Bab{1, \dots, (N_{2n, l}' - n) \lor 0}}.
\end{align}
Similarly, for $l_1 < l_2$ and $\epsilon > 0$,
let $\ppp_{2n\epsilon, l_1, l_2}'$ be a Poisson point process on $\unitcube$ with intensity $2n\epsilon$
constructed from $\bpp_{l_2} \setminus \bpp_{l_1}$.

With these auxiliary processes,
we obtain the following tail estimate for the difference between $\kappa_n(\bpp_{l_1})$ and $\kappa_n(\bpp_{l_2})$.

\begin{proposition}
\label{prop:tail.bppl1.bppl2}
Let $f \in \bddmble(E)$ or
$f = \iprod{\bfa}{\cdot}$ for some $\bfa = (a_1, \dots, a_m) \in \reals^m$.
Let $\epsilon \in (0, 1) $ and $l_1, l_2 \geq 1$ such that $l_1 < l_2 \leq l_1 + \epsilon n $.
Then, there exist constants $C_1, C_2 > 0$ independent of $n, l_1, l_2, \epsilon$ and $f$,
such that for sufficiently large $n$ and $R > (C_1^2 C_f' \epsilon \ldpspeed) \lor (C_1^2 C_f' \epsilon \nrnd \ldpspeed ) $,
\begin{align}
\begin{split}
  &\probab[Big]{\abs[big]{\pairing{f}{\kappa_n(\bpp_{l_1})} - \pairing{f}{\kappa_n(\bpp_{l_2})}} > R, \, \bpp_{l_1} \subset \ppp_{2n, l_1}', \, \bpp_{l_2} \setminus \bpp_{l_1} \subset \ppp_{2n\epsilon, l_1, l_2}'} \\
  &\quad \leq \sum_{j=0}^{k-1} \exp\ab\Big(- C_1 \ab\big((C_1 C_2)^{-1} R^{1/2} - (C_f' \epsilon^{k-j} \ldpspeed)^{1/2})^2 ) \\
  &\qquad + \exp\ab\Big(- C_1 \ab\big((C_1 C_2)^{-1} R^{1/2} - (C_f' \epsilon \nrnd \ldpspeed)^{1/2})^2),
\end{split} \label{eq:tail.bppl1.bppl2}
\end{align}
where
\begin{align}
  C_f' \coloneqq \int_{E} e^{g_f} \,d\tau, \label{eq:bpp.ppp.const.2}
\end{align}
and $g_f $ is defined by \eqref{eq:kappa.f.projection.sum}.
\end{proposition}

\begin{proof}
\textit{Step} 1.
By \eqref{eq:kappa.decompose}, we observe that
\begin{align}
  &\abs{\pairing{f}{\kappa_n(\bpp_{l_2})} - \pairing{f}{\kappa_n(\bpp_{l_1})}}\\
  &\quad \leq \sum_{j=0}^{k-1} \sum_{\substack{\calY \subset \bpp_{l_1}, \calY' \subset \bpp_{l_2} \setminus \bpp_{l_1} \\ \card{\calY}=j, \card{\calY'}= k-j}}
  \sum_{i=1}^m \ab\big(\isolated[n]{\calY \cup \calY'}{\bpp_{l_2}}{t_{i}} - \isolated[n]{\calY \cup \calY'}{\bpp_{l_2}}{t_{i+1}}) \abs{\zeroed{f}(H_n^{(i)}(\calY \cup \calY'), \bfzero_{m-i})} \\
  &\qquad + \sum_{\calY \subset \bpp_{l_1}, \card{\calY}=k} \sum_{i=1}^m  \abs[Big]{\ab\big(\isolated[n]{\calY}{\bpp_{l_2}}{t_{i}} - \isolated[n]{\calY}{\bpp_{l_2}}{t_{i+1}}) - \ab\big(\isolated[n]{\calY}{\bpp_{l_1}}{t_{i}} - \isolated[n]{\calY}{\bpp_{l_1}}{t_{i+1}})} \\
  &\quad \qquad \times \abs{\zeroed{f}(H_n^{(i)}(\calY), \bfzero_{m-i})} \\
  &\quad \leq \sum_{j=0}^{k-1} \sum_{\substack{\calY \subset \bpp_{l_1}, \calY' \subset \bpp_{l_2} \setminus \bpp_{l_1} \\ \card{\calY}=j, \card{\calY'}= k-j}}
  \isolated[n]{\calY \cup \calY'}{\bpp_{l_2}}{t_{1}} \zeroed{g_f}(H_n(\calY \cup \calY')) \\
  &\qquad + 2\sum_{\calY \subset \bpp_{l_1}, \card{\calY}=k} \isolated[n]{\calY}{\bpp_{l_1}}{t_{1}} \ab\big(1 - \isolated[n]{\calY}{\bpp_{l_2}\setminus \bpp_{l_1}}{t_{m}}) \zeroed{g_f}(H_n(\calY)), \label{eq:bpp.l1l2.1}
\end{align}
where we used the inequality
\begin{align}
  & \abs[Big]{\ab\big(\isolated[n]{\calY}{\bpp_{l_2}}{t_{i}} - \isolated[n]{\calY}{\bpp_{l_2}}{t_{i+1}}) - \ab\big(\isolated[n]{\calY}{\bpp_{l_1}}{t_{i}} - \isolated[n]{\calY}{\bpp_{l_1}}{t_{i+1}})} \\
  &\quad \leq \isolated[n]{\calY}{\bpp_{l_2}}{t_{i}}
  \isolated[n]{\calY}{\bpp_{l_1}}{t_{i+1}}
  \ab\big(1 - \isolated[n]{\calY}{\bpp_{l_2}\setminus \bpp_{l_1}}{t_{i+1}}) \\
  &\qquad + \isolated[n]{\calY}{\bpp_{l_1}}{t_{i}}
  \ab\big(1 - \isolated[n]{\calY}{\bpp_{l_1}}{t_{i+1}})
  \ab\big(1 - \isolated[n]{\calY}{\bpp_{l_2}\setminus \bpp_{l_1}}{t_{i}}) \\
  &\quad \leq 2\isolated[n]{\calY}{\bpp_{l_1}}{t_{1}} \ab\big(1 - \isolated[n]{\calY}{\bpp_{l_2}\setminus \bpp_{l_1}}{t_{m}})
\end{align}
for $i=1, \dots, m$.

Let $Q_{n, z} \coloneqq Q(6Lr_n, z)$ for $z \in \integers^d $ and
$I_n \coloneqq \Bab{z \in \integers^d : Q_{n,z} \cap \unitcube \neq \emptyset} $.
Note that the number of $\calY \subset \bpp_{l_1} $ satisfying $\isolated[n]{\calY}{\bpp_{l_1}}{t_{1}} = 1$ and $\lex_1(\calY) \in Q_{n, z}$ is bounded by a constant, 
which is independent of $n$ and depends only on $d, k, L$.
Similarly, 
the number of pairs $(\calY, \calY')$ with $\calY \subset \bpp_{l_1}$ and $\calY' \subset \bpp_{l_2} \setminus \bpp_{l_1} $ such that $\isolated[n]{\calY \cup \calY'}{\bpp_{l_2}}{t_{1}} = 1$ and $\lex_1(\calY \cup \calY') \in Q_{n,z} $ is also bounded by a constant.
Let $M$ be such a constant.
Then, on the event $\{\bpp_{l_1} \subset \ppp_{2n, l_1}'\} \cap \{\bpp_{l_2} \setminus \bpp_{l_1} \subset \ppp_{2n\epsilon, l_1, l_2}'\}$, 
\eqref{eq:bpp.l1l2.1} is bounded by
\begin{align}
  &\sum_{j=0}^{k-1} \sum_{z \in I_n} \max_{\substack{\calY \subset \bpp_{l_1}, \calY' \subset \bpp_{l_2} \setminus \bpp_{l_1} \\ \card{\calY}=j, \card{\calY'}= k-j}}
  M \indicator[none]{\lex_1(\calY \cup \calY') \in Q_{n,z}} \zeroed{g_f}(H_n(\calY \cup \calY'))\\
  &\qquad + 2\sum_{z \in I_n} \max_{\calY \subset \bpp_{l_1}, \card{\calY}=k} M \ab\big(1 - \isolated[n]{\calY}{\bpp_{l_2}\setminus \bpp_{l_1}}{t_{m}})
  \indicator[none]{\lex_1(\calY) \in Q_{n,z}} \zeroed{g_f}(H_n(\calY)) \\
  &\quad \leq M\sum_{j=0}^{k-1} \sum_{z \in I_n} \max_{\substack{\calY \subset \ppp_{2n, l_1}', \calY' \subset \ppp_{2n\epsilon, l_1, l_2}'\\ \card{\calY}=j, \card{\calY'}= k-j}}
  \indicator[none]{\lex_1(\calY \cup \calY') \in Q_{n,z}} \zeroed{g_f}(H_n(\calY \cup \calY')) \\
  &\qquad + 2M\sum_{z \in I_n} \max_{\calY \subset \ppp_{2n, l_1}', \card{\calY}=k}
  \ab\big(1 - \isolated[n]{\calY}{\bpp_{l_2}\setminus \bpp_{l_1}}{t_{m}})
  \indicator[none]{\lex_1(\calY) \in Q_{n,z}} \zeroed{g_f}(H_n(\calY)) \\
  &\quad \eqqcolon M \sum_{j=0}^{k-1} Z_{n, j} + 2M V_{n}.
\end{align}
Applying the union bound, 
we see that the left-hand side of \eqref{eq:tail.bppl1.bppl2} is bounded by
\begin{align}
  \sum_{j=0}^{k-1} \probab[Big]{Z_{n, j} > \frac{R}{2Mk}}
   + \probab[Big]{V_{n} > \frac{R}{4M}}
  \leq \sum_{j=0}^{k-1} \probab[Big]{Z_{n, j} > \frac{R}{2Mk}}
   + \probab[Big]{V_{n} > \frac{R}{2Mk}}
\end{align}

\textit{Step} 2.
Now we estimate $Z_{n,  j}$.
Let $R' \coloneqq R/2Mk $ and $I_{n, w} \coloneqq I_n \cap (2\integers^d + w) $ for $w \in \{0, 1\}^d $.
For any $\lambda > 0$, 
by Chebyshev's inequality, 
$\probab[none]{Z_{n, j} > R'}$ is bounded by
\begin{align}
  e^{-\lambda R' } \expect[bigg]{\exp\ab\bigg(\lambda\sum_{z \in I_n} \max_{(\calY, \calY') \in \calI}
  \indicator[none]{\lex_1(\calY \cup \calY') \in Q_{n, z}} \zeroed{g_f}(H_n(\calY \cup \calY')))}, \label{eq:ppp.bpp.exp.spread}
\end{align}
where $\calI \coloneqq \{(\calY, \calY') : \calY \subset \ppp_{2n,  l_1}',\, \calY' \subset \ppp_{2n\epsilon, l_1, l_2}',\, \card{\calY}=j,\, \card{\calY'}= k-j \} $.
Using Jensen's inequality, the independence of the summands in \eqref{eq:ppp.bpp.exp.spread} indexed by $z \in I_{n, w} $
and an argument similar to that in the proof of \cref{prop:exp.moment.common},
we can bound the expectation in \eqref{eq:ppp.bpp.exp.spread} by 
\begin{align}
  &\prod_{w \in \Bab{0, 1}^d} \prod_{z \in I_{n, w}} \expect[bigg]{\exp\ab\bigg(2^d\lambda \max_{(\calY, \calY') \in \calI}
  \indicator[none]{\lex_1(\calY \cup \calY') \in Q_{n, z}} \zeroed{g_f}(H_n(\calY \cup \calY')))}^{2^{-d}}, \label{eq:ppp.bpp.exp.spread.1}
\end{align}
and the expectation in \eqref{eq:ppp.bpp.exp.spread.1} is bounded by
\begin{align}
  &\ab\bigg(1 + \sum_{p=1}^{\infty}\frac{(2^d\lambda)^p}{p!} \expect[bigg]{\max_{(\calY, \calY') \in \calI}
  \indicator[none]{\lex_1(\calY \cup \calY') \in Q_{n,z}} \zeroed{g_f}(H_n(\calY \cup \calY'))^p})^{2^{-d}} \\
  &\quad \leq \ab\Bigg(1 + \sum_{p=1}^{\infty}\frac{(2^d\lambda)^p}{p!} \expect[Bigg]{\sum_{(\calY, \calY') \in \calI}
  \indicator[none]{\lex_1(\calY \cup \calY') \in Q_{n,z}} \zeroed{g_f}(H_n(\calY \cup \calY'))^p})^{2^{-d}}. \label{eq:ppp.bpp.exp.spread.2}
\end{align}
Furthermore, 
by the Mecke formula, 
for all $p \geq 1$, 
the expectation in \eqref{eq:ppp.bpp.exp.spread.2} is bounded by
\begin{align}
  \frac{(2n)^j}{j!} \frac{(2n\epsilon)^{k-j}}{(k-j)!} \int_{(\unitcube)^k} \indicator[none]{\lex_1(\bfx) \in Q_{n,z}} \zeroed{g_f}(H_n(\bfx))^p \,d\bfx
  \leq \frac{k! 2^k n^k \epsilon^{k-j}}{j!(k-j)!} \lebmeas(Q_{n,z}) r_n^{d(k-1)} p! \int_{E} e^{g_f} \,d\tau.
\end{align}

Using the relations $\lebmeas(Q_{n,z}) = (6L r_n)^d $ and $\card{I_n} = \sum_{w \in \Bab{0, 1}^d} \card{I_{n, w}} \leq 2^d (6Lr_n)^{-d} $, 
and combining the estimates above, 
we can bound \eqref{eq:ppp.bpp.exp.spread} by
\begin{align}
  &e^{- \lambda R'} \prod_{w \in \Bab{0, 1}^d} \ab\bigg(1 + C C_f' (\nrnd)^{k}\epsilon^{k-j} \sum_{p=1}^{\infty} (2^d \lambda)^p )^{2^{-d} \card{I_{n, w}}} \\
  &\quad \leq e^{- \lambda R'} \exp\ab\bigg(C C_f' 2^{-d} \card{I_{n}} (\nrnd)^{k}\epsilon^{k-j} \sum_{p=1}^{\infty} (2^d \lambda)^p ) 
  \leq e^{- \lambda R'} \exp\ab\Big(C C_f' \epsilon^{k-j} \ldpspeed \frac{C\lambda}{1 - C\lambda} ),
\end{align}
where $C > 0$ is a constant depending only on $k$, $j$, and $L$.
Recall that for any $c > 0 $, $s > a > 0$, 
\begin{align}
  \inf_{\lambda \in (0, c^{-1})} \ab(-\lambda s + \frac{a\lambda}{1 - c\lambda}) = - \frac{\ab\big(s^{1/2} - a^{1/2})^2}{c}.
\end{align}
Optimizing over $\lambda$, 
we obtain for $R' > C^2 C_f' \ldpspeed \epsilon^{k-j}$,
\begin{align}
  \probab[none]{Z_{n, j} > R'}
  \leq \exp\ab\big(- C (C^{-1} R'^{1/2} - (C_f' \epsilon^{k-j} \ldpspeed)^{1/2})^2 ).
\end{align}

\textit{Step} 3.
Next, 
we estimate $V_{n}$.
We write $V_{n} \eqqcolon \sum_{z \in I_n} U_{n, z} $ for simplicity.
For any $\lambda > 0$, 
by Chebyshev's and Jensen's inequalities, 
$\probab[none]{V_{n} > R'}$ is bounded by
\begin{align}
  e^{-\lambda R'} \expect[bigg]{\exp\ab\bigg(\lambda \sum_{z \in I_n} U_{n, z} )}
  \leq e^{-\lambda R'} \prod_{w \in \Bab{0, 1}^d} \expect[bigg]{\exp\ab\bigg(2^d \lambda \sum_{z \in I_{n, w}} U_{n, z})}^{2^{-d}}.
  \label{eq:ppp.bpp.exp.2}
\end{align}
Let $J_{n, w}(\bpp_{l_2}\setminus \bpp_{l_1}) \coloneqq \{z \in I_{n, w} : Q_{n,z}^{(2Lr_n)} \cap (\bpp_{l_2}\setminus \bpp_{l_1}) \neq \emptyset \} $.
Since for $z \in I_{n, w} \setminus J_{n, w}(\bpp_{l_2}\setminus \bpp_{l_1}) $, there is no subset $\calY \subset \ppp_{2n, l_1}'$ with $\card{\calY}=k$ 
such that $\isolated[n]{\calY}{\bpp_{l_2}\setminus \bpp_{l_1}}{t_{m}} = 0$ and $\lex_1(\calY) \in Q_{n,z} $, 
we have $\sum_{z \in I_{n, w}} U_{n, z} = \sum_{z \in J_{n, w}(\bpp_{l_2}\setminus \bpp_{l_1})} U_{n, z}$.
Furthermore, 
the random variables $\{U_{n, z} : z \in J_{n, w}(\bpp_{l_2}\setminus \bpp_{l_1}) \} $ are conditionally independent given $\bpp_{l_2}\setminus \bpp_{l_1}$.
Thus, 
the expectation on the right-hand side of \eqref{eq:ppp.bpp.exp.2} can be rewritten as
\begin{align}
  \expect[Bigg]{ \prod_{z \in J_{n, w}(\bpp_{l_2}\setminus \bpp_{l_1})} \condexpect[big]{\exp(2^d \lambda U_{n, z})}{\bpp_{l_2}\setminus \bpp_{l_1}}}^{2^{-d}}.
\end{align}
As in Step 2,
each conditional expectation $\condexpect[big]{\exp(2^d \lambda U_{n, z})}{\bpp_{l_2}\setminus \bpp_{l_1}}$ is bounded by
\begin{align}
  &1 + \sum_{p=1}^{\infty}\frac{(2^d\lambda)^p}{p!}
  \expect[bigg]{\max_{\calY \subset \ppp_{2n, l_1}', \card{\calY}=k} 
  \indicator[none]{\lex_1(\calY) \in Q_{n,z}} \zeroed{g_f}(H_n(\calY))^p} \\
  &\quad \leq 1 + \sum_{p=1}^{\infty}\frac{(2^d\lambda)^p}{p!} 
  \expect[Bigg]{\sum_{\calY \subset \ppp_{2n, l_1}', \card{\calY}=k} 
  \indicator[none]{\lex_1(\calY) \in Q_{n,z}} \zeroed{g_f}(H_n(\calY))^p} \\
  &\quad \leq 1 + \sum_{p=1}^{\infty} (2^d\lambda)^p 
  \frac{(2n)^k}{k!} \lebmeas(Q_{n,z}) r_n^{d(k-1)} C_f'.
\end{align}
Moreover, 
since the cubes $\{Q_{n,z}^{(2Lr_n)}\}_{z \in I_{n, w}} $ are disjoint, 
we have the bound $\card{J_{n, w}(\bpp_{l_2}\setminus \bpp_{l_1})} \leq \card{\bpp_{l_2}\setminus \bpp_{l_1}} \leq n \epsilon $.
Combining these inequalities, 
we can bound \eqref{eq:ppp.bpp.exp.2} by
\begin{align}
  e^{-\lambda R'} \prod_{w \in \Bab{0, 1}^d}
  \expect[bigg]{\ab\Big(1 + C C_f' (\nrnd)^{k} \sum_{p=1}^\infty (2^d \lambda)^p )^{\card{J_{n, w}(\bpp_{l_2}\setminus \bpp_{l_1})}}}^{2^{-d}} 
   \leq e^{-\lambda R'} \exp\ab(C C_f' \epsilon \nrnd \ldpspeed \frac{C\lambda}{1 - C\lambda}).
\end{align}
Consequently, 
we obtain for $R' > C^2 C_f' \epsilon \nrnd \ldpspeed $,
\begin{align}
  \probab[none]{V_{n} > R'}
  \leq \exp\ab\big(- C (C^{-1} R'^{1/2} - (C_f' \epsilon \nrnd \ldpspeed )^{1/2})^2),
\end{align}
Combining the estimates from the three steps completes the proof of \cref{prop:tail.bppl1.bppl2}.
\end{proof}

By combining the estimates above,
we obtain the desired exponential equivalences examined at the beginning of this subsection.

\begin{lemma}
\label{prop:exp.eq.ppp.bpp}
For any $f \in \bddmble(E)$ and $\delta > 0$,
\begin{align}
  \lim_{n \to \infty} \frac{\ldpspeed}{\mdpfactor^2} \log \probab[big]{\abs[big]{\pairing[big]{f}{\centered{\kappa_n(\ppp_n)}} - \pairing[big]{f}{\centered{\kappa_n(\bpp_n)}}} > \delta \mdpfactor}
  = -\infty. \label{eq:kappa.ppp.bpp.exp.eq}
\end{align}
Furthermore, for any $\delta > 0$,
\begin{align}
  \lim_{n \to \infty} \frac{\ldpspeed}{\mdpfactor^2} \log \probab[big]{\norm[big]{\centered{T_n(\ppp_n)} - \centered{T_n(\bpp_n)}}_{\reals^m} > \delta \mdpfactor}
  = -\infty. \label{eq:T.ppp.bpp.exp.eq}
\end{align}
\end{lemma}

\begin{proof}
We first note that the proof of \eqref{eq:T.ppp.bpp.exp.eq} reduces to that of \eqref{eq:kappa.ppp.bpp.exp.eq}.
Indeed, setting $f_i \coloneqq \iprod{\bfe_i}{\cdot}$ for $i=1,\dots,m$,
we have
\begin{align}
  \norm{\centered{T_n(\ppp_n)} - \centered{T_n(\bpp_n)}}_{\reals^m}
  \leq \sum_{i=1}^m \abs{\pairing[big]{f_i}{\centered{\kappa_n(\ppp_n)}} - \pairing[big]{f_i}{\centered{\kappa_n(\bpp_n)}}}.
\end{align}
Thus, it suffices to prove \eqref{eq:kappa.ppp.bpp.exp.eq}.

By \eqref{eq:kappa.ppp.bpp.expect.vanish},
in order to establish \eqref{eq:kappa.ppp.bpp.exp.eq},
it is sufficient to show that
\begin{align}
  \lim_{n \to \infty} \frac{\ldpspeed}{\mdpfactor^2} \log \probab[big]{\abs[big]{\pairing{f}{\kappa_n(\ppp_n)} - \pairing{f}{\kappa_n(\bpp_n)}} > \delta \mdpfactor} = -\infty.
\end{align}
Let
\begin{align}
  \epsilon_n \coloneq (\nrnd)^{(k-1)/4} \frac{\mdpfactor}{\ldpspeed}. \label{eq:pp.bpp.epsilon}
\end{align}
By the union bound, 
we have
\begin{align}
  &\probab[big]{\abs[big]{\pairing{f}{\kappa_n(\ppp_n)} - \pairing{f}{\kappa_n(\bpp_n)}} > \delta \mdpfactor} \\
  &\quad \leq \probab{\abs{N_n - n} > n\epsilon_n}
   + \sum_{l: \abs{n-l} \leq n\epsilon_n} \condprobab[big]{\abs[big]{\pairing{f}{\kappa_n(\ppp_n)} - \pairing{f}{\kappa_n(\bpp_n)}} > \delta \mdpfactor}{N_n = l} \probab{N_n = l} \\
  &\quad \leq \probab{\abs{N_n - n} > n\epsilon_n}
   + \max_{l: \abs{n-l} \leq n\epsilon_n} \probab[big]{\abs[big]{\pairing{f}{\kappa_n(\bpp_l)} - \pairing{f}{\kappa_n(\bpp_n)}} > \delta \mdpfactor}. \label{eq:tail.ppp.bpp.union.bound}
\end{align}
Let $l_1 \coloneqq l \land n$ and $ l_2 \coloneqq l \lor n $.
Then we see that \eqref{eq:tail.ppp.bpp.union.bound} is bounded by
\begin{align}
  \begin{split}
  &\probab{\abs{N_n - n} > n\epsilon_n}
  + \max_{l: \abs{n-l} \leq n\epsilon_n} \ab\Big(\probab[big]{\bpp_{l_1} \not\subset \ppp_{2n, l_1}'} + \probab[big]{\bpp_{l_2} \setminus \bpp_{l_1} \not\subset \ppp_{2n\epsilon_n, l_1, l_2}'}) \\
  &\quad + \max_{l: \abs{n-l} \leq n\epsilon_n} \probab[Big]{\abs[big]{\pairing{f}{\kappa_n(\bpp_{l_1})} - \pairing{f}{\kappa_n(\bpp_{l_2})}} > \delta \mdpfactor, \, \bpp_{l_1} \subset \ppp_{2n, l_1}', \, \bpp_{l_2} \setminus \bpp_{l_1} \subset \ppp_{2n\epsilon_n, l_1, l_2}'}.
  \end{split} \label{eq:ppp.bpp.prob.decomposition}
\end{align}

By the well-known tail probability bounds for the Poisson distribution (see, e.g., Lemma 1.2 in \cite{penrose2003rgg}), 
we have
\begin{align}
  \probab{\abs{N_n - n} > n\epsilon_n}
  &= \probab{N_n > (1 + \epsilon_n)n} + \probab{N_n < (1 - \epsilon_n)n} \\
  &\leq \exp\ab\big(- n\varphi\ab((1+\epsilon_n)n/n) ) + \exp\ab\big(- n\varphi\ab((1-\epsilon_n)n/n) ),
\end{align}
where $\varphi(x) \coloneqq x \log x - x + 1$.
Furthermore, 
for $l $ satisfying $\abs{l - n} \leq n\epsilon_n$, 
since $l_1 = l \land n \leq n $ and $l_2- l_1 = \abs{n - l} \leq n\epsilon_n $, 
it holds that
\begin{align}
  \probab[big]{\bpp_{l_1} \not\subset \ppp_{2n, l_1}'} + \probab[big]{\bpp_{l_2} \setminus \bpp_{l_1} \not\subset \ppp_{2n\epsilon_n, l_1, l_2}'}
  &= \probab[big]{\ppp_{2n, l_1}'(\unitcube) < l_1} + \probab[big]{\ppp_{2n\epsilon_n, l_1, l_2}'(\unitcube) < l_2 - l_1} \\
  &\leq \exp\ab\big(- 2n\varphi\ab(l_1/2n) ) + \exp\ab\big(- 2n\epsilon_n \varphi\ab((l_2 - l_1)/2n\epsilon_n) ) \\
  % &\leq \exp\ab(- 2n\varphi\ab(n/2n) ) + \exp\ab(- 2n\epsilon_n \varphi\ab(n\epsilon_n/2n\epsilon_n) ) \\
  &\leq \exp\ab(- 2\varphi\ab(1/2)n ) + \exp\ab(- 2 \varphi\ab(1/2)n\epsilon_n ).
\end{align}
Since $\varphi(1 + x) \geq x^2/3$ for all $x \in (-1, 1)$, 
for all sufficiently large $n \geq 1$,
\begin{align}
  \probab{\abs{N_n - n} > n\epsilon_n}
  + \max_{l: \abs{n-l} \leq n\epsilon_n} \ab\Big(\probab[big]{\bpp_{l_1} \not\subset \ppp_{2n, l_1}'} + \probab[big]{\bpp_{l_2} \setminus \bpp_{l_1} \not\subset \ppp_{2n\epsilon_n, l_1, l_2}'}) 
  \leq 4 \exp\ab(- n\epsilon_n^2 / 3 ). \label{eq:ppp.num.prob.bound}
\end{align}

Furthermore, since $\mdpfactor / (\epsilon_n^{k-j} \ldpspeed) \to \infty$ for $j=0,\dots,k-1$ and $\mdpfactor / (\epsilon_n \nrnd \ldpspeed) \to \infty$ as $n \to \infty$,
applying \eqref{eq:ppp.bpp.prob.decomposition} and \eqref{eq:tail.bppl1.bppl2} yields, for all sufficiently large $n$,
\begin{align}
  \probab[big]{\abs[big]{\pairing{f}{\kappa_n(\ppp_n)} - \pairing{f}{\kappa_n(\bpp_n)}} > \delta \mdpfactor}
  \leq 4 \exp\ab(- n\epsilon_n^2 / 3 )
   + k \exp(- C \delta \mdpfactor). \label{eq:tail.ppp.bpp}
\end{align}
For either choice $a_n = - n\epsilon_n^2/3$ or $a_n = - C \delta \mdpfactor $ in \eqref{eq:tail.ppp.bpp}, we have
\begin{align}
  \limsup_{n \to \infty} \frac{\ldpspeed}{\mdpfactor^2} a_n = -\infty,
\end{align}
which yields \eqref{eq:kappa.ppp.bpp.exp.eq}.
\end{proof}

\section{Proof of the laws of the iterated logarithm}
\label{sec:lil.upper}
In this section, 
we prove \cref{thm:lil.measure} and \cref{cor:lil.scalar},
assuming for the moment the estimates to be established later.
Throughout this section,
we set $\mdpfactor = \sqrt{\ldpspeed \log \log \ldpspeed}$,
and let $r_n$ be given by \eqref{eq:rn.decay.speed}.
Let $\accptmeas \coloneqq \Lambda^{-1}([0, 1])$ and $\accptvec \coloneqq I^{-1}([0, 1])$.
For $\nu \in \signedmeas(E)$ and $\bfu \in \reals^m$, we define the distances to these sets by
$\distancemeas(\nu, \accptmeas) \coloneqq \inf\{\distancemeas(\nu, \mu) : \mu \in \accptmeas\}$
and $\distancevec(\bfu, \accptvec) \coloneqq \inf\{\norm{\bfu - \bfv}_{\reals^m} : \bfv \in \accptvec\}$.

\begin{proof}[Proof of \cref{thm:lil.measure} and \cref{cor:lil.scalar}]
First, we establish the LIL for $\{\mdpfactor^{-1} \centered{\xi_n(\ppp_n)}\}_{n \geq 1}$ and $\{\mdpfactor ^{-1} \Phi(\centered{\xi_n(\ppp_n)})\}_{n \geq 1}$,
where $\Phi$ is defined by \eqref{eq:map.contract.meas.to.vec}.
Specifically, we will show that almost surely,
\begin{align}
  \lim_{n \to \infty} \distancemeas(\mdpfactor^{-1} \centered{\xi_{n}(\ppp_n)}, \accptmeas) &= 0, \label{eq:lil.upper.meas} \\
  \lim_{n \to \infty} \distancevec(\mdpfactor^{-1} \Phi(\centered{\xi_{n}(\ppp_n)}), \accptvec) &= 0, \label{eq:lil.upper.vec}
\end{align}
under condition (H6) for \eqref{eq:lil.upper.meas} and \eqref{eq:lil.upper.vec}.
Furthermore, we will show that almost surely,
\begin{align}
  \liminf_{n \to \infty} \distancemeas(\mdpfactor^{-1} \centered{\xi_{n}(\ppp_n)}, \theta) = 0 \quad \text{for all } \theta \in \accptmeas, \label{eq:lil.lower.meas} \\
  \liminf_{n \to \infty} \norm[big]{\mdpfactor^{-1} \Phi(\centered{\xi_{n}(\ppp_n)}) - \bfu}_{\reals^m} = 0 \quad \text{for all } \bfu \in \accptvec. \label{eq:lil.lower.vec}
\end{align}

Next, by establishing that almost surely,
\begin{align}
  \lim_{n \to \infty} \distancemeas\ab\big(\mdpfactor^{-1} \centered{\xi_n(\ppp_n)}, \mdpfactor^{-1} \centered{\kappa_n(\ppp_n)}) &= 0, \label{eq:lil.equivalent.xi.kappa} \\
  \lim_{n \to \infty} \distancemeas\ab\big(\mdpfactor^{-1} \centered{\kappa_n(\ppp_n)}, \mdpfactor^{-1} \centered{\kappa_n(\bpp_n)}) &= 0, \label{eq:lil.equivalent.kappappp.kappabpp}
\end{align}
we deduce the LIL for $\{\mdpfactor^{-1} \centered{\kappa_n(\ppp_n)}\}_{n \geq 1}$ and $\{\mdpfactor^{-1} \centered{\kappa_n(\bpp_n)}\}_{n \geq 1}$,
thereby completing the proof of \cref{thm:lil.measure}.
Here, the time parameters for $\{\kappa_n\}_{n \geq 1}$ are given by $0 < t_1 \leq \dots \leq t_m$,
while $t = t_1$ is taken to be the time parameter for $\{\xi_n\}_{n \geq 1}$.
Similarly, by proving that almost surely,
\begin{align}
  \lim_{n \to \infty} \norm[big]{\mdpfactor^{-1} \Phi(\centered{\xi_n(\ppp_n)}) - \mdpfactor^{-1} \centered{T_n(\ppp_n)}}_{\reals^m} &= 0, \label{eq:lil.equivalent.phixi.t}\\
  \lim_{n \to \infty} \norm[big]{\mdpfactor^{-1} \centered{T_n(\ppp_n)} - \mdpfactor^{-1} \centered{T_n(\bpp_n)}}_{\reals^m} &= 0, \label{eq:lil.equivalent.tppp.tbpp}
\end{align}
we obtain the LIL for $\{\mdpfactor^{-1} \centered{T_n(\ppp_n)}\}_{n \geq 1}$ and $\{\mdpfactor^{-1} \centered{T_n(\bpp_n)}\}_{n \geq 1}$,
which completes the proof of \cref{cor:lil.scalar}.
\end{proof}

\subsection{Proof of \eqref{eq:lil.upper.meas} and \eqref{eq:lil.upper.vec}}
In this subsection, we focus on the Poisson point process $\ppp_n$ only.
Thus, we simplify the notation by omitting $\ppp_n$,
and write $\xi_n$ instead of $\xi_n(\ppp_n)$.
First, we prove \eqref{eq:lil.upper.meas},
and then we establish \eqref{eq:lil.upper.vec}.

For any $s > 1$, we define the subsequence $\lilsubseq{s}{l} \coloneqq \rounddown{s^{l/(k - \alpha(k-1))}}$,
which is chosen so that $\ldpspeed[\lilsubseq{s}{l}]$ is close to $s^l$.
Recall that $\ldpspeed = n^{k - \alpha(k-1)}$ under \eqref{eq:rn.decay.speed}.
Furthermore, throughout this section, we assume without loss of generality that $\supnorm{f_q} = 1$.

By a standard LIL argument,
we obtain the following convergence along the subsequence $\{\lilsubseq{s}{l}\}_{l \geq 1}$.

\begin{lemma}
\label{lem:lil.subseqence.cptness}
For all $s > 1$,
it holds almost surely that
\begin{align}
  \lim_{l \to \infty} \distancemeas(\mdpfactor[\lilsubseq{s}{l}]^{-1} \centered{\xi_{\lilsubseq{s}{l}}}, \accptmeas) = 0. \label{eq:lil.subseqence.cptness.meas}
\end{align}
\end{lemma}

\begin{proof}
For any $\delta > 0$,
since $\Lambda$ is a good rate function, we have
\begin{align}
  \inf \Bab{\Lambda(\nu) : \nu \in \signedmeas(E), \, \distancemeas(\nu, \accptmeas) \geq \delta} > 1.
\end{align}
Thus, by \cref{thm:mdp.measure},
there exists a constant $\beta > 1$ such that for all sufficiently large $l$,
\begin{align}
  \probab[none]{\distancemeas(\mdpfactor[\lilsubseq{s}{l}]^{-1} \centered{\xi_{\lilsubseq{s}{l}}}, \accptmeas) \geq \delta }
  \leq \exp(- \beta \log \log \ldpspeed[\lilsubseq{s}{l}])
  \leq \exp(- \beta \log (2^{-1}\log s^l))
  = (2^{-1} l \log s )^{- \beta}.
\end{align}
Therefore, by the first Borel--Cantelli lemma,
we obtain \eqref{eq:lil.subseqence.cptness.meas}.
\end{proof}

Using \cref{lem:lil.subseqence.cptness},
we have for all $s > 1$, almost surely,
\begin{align}
  &\limsup_{n \to \infty} \distancemeas(\mdpfactor^{-1} \centered{\xi_{n}}, \accptmeas) \\
  &\quad \leq \limsup_{l \to \infty} \mdpfactor[\lilsubseq{s}{l}]^{-1} \max_{\lilsubseq{s}{l} \leq n \leq \lilsubseq{s}{l+1}} \distancemeas(\centered{\xi_{n}}, \centered{\xi_{\lilsubseq{s}{l}}})
   + (1 - s^{-1/2})\sup_{\nu \in \accptmeas} \distancemeas(\nu, 0). \label{eq:lil.upper.subseq.bound}
\end{align}
Indeed, for $s > 1$ and $\lilsubseq{s}{l} \leq n \leq \lilsubseq{s}{l+1}$,
\begin{align}
  &\distancemeas(\mdpfactor^{-1} \centered{\xi_{n}}, \accptmeas)\\
  &\quad \leq \distancemeas(\mdpfactor^{-1} \centered{\xi_{n}}, \mdpfactor^{-1} \centered{\xi_{\lilsubseq{s}{l}}}) 
   + \distancemeas(\mdpfactor^{-1} \centered{\xi_{\lilsubseq{s}{l}}}, \mdpfactor[\lilsubseq{s}{l}]^{-1} \centered{\xi_{\lilsubseq{s}{l}}})
   + \distancemeas(\mdpfactor[\lilsubseq{s}{l}]^{-1} \centered{\xi_{\lilsubseq{s}{l}}}, \accptmeas) \\
  &\quad \leq \mdpfactor[\lilsubseq{s}{l}]^{-1} \distancemeas(\centered{\xi_{n}}, \centered{\xi_{\lilsubseq{s}{l}}}) 
   + (1 - \mdpfactor[\lilsubseq{s}{l}]/\mdpfactor[\lilsubseq{s}{l+1}])\distancemeas(\mdpfactor[\lilsubseq{s}{l}]^{-1}\centered{\xi_{\lilsubseq{s}{l}}}, 0)
   + \distancemeas(\mdpfactor[\lilsubseq{s}{l}]^{-1} \centered{\xi_{\lilsubseq{s}{l}}}, \accptmeas).
\end{align}
Thus, together with
\begin{align}
  \distancemeas(\mdpfactor[\lilsubseq{s}{l}]^{-1}\centered{\xi_{\lilsubseq{s}{l}}}, 0)
  \leq \distancemeas(\mdpfactor[\lilsubseq{s}{l}]^{-1}\centered{\xi_{\lilsubseq{s}{l}}}, \accptmeas) + \sup_{\nu \in \accptmeas} \distancemeas(\nu, 0),
\end{align}
applying \cref{lem:lil.subseqence.cptness} yields \eqref{eq:lil.upper.subseq.bound}.
Note that the second term on the right-hand side of \eqref{eq:lil.upper.subseq.bound} vanishes as $s \to 1$.
Consequently, to establish \eqref{eq:lil.upper.meas},
it suffices to show that for every $\delta > 0$,
there exists $s > 1$ such that almost surely,
\begin{align}
  \limsup_{l \to \infty} \mdpfactor[\lilsubseq{s}{l}]^{-1} \max_{\lilsubseq{s}{l} \leq n \leq \lilsubseq{s}{l+1}} \distancemeas(\centered{\xi_{n}}, \centered{\xi_{\lilsubseq{s}{l}}}) &\leq \delta. \label{eq:lil.upper.max.limsup}
\end{align}

To establish \eqref{eq:lil.upper.max.limsup},
we introduce some additional notation.
Let $v_{n} \coloneqq (\nrnd)^{1/d}$,
$Q_{z} \coloneqq Q(v_{\lilsubseq{s}{l+1}}t/\sqrt{d}, z)$,
\begin{align}
  \calF_z \coloneqq
  \sigma\ab\bigg(\pppentire \cap Q_{w} : \begin{aligned} & (w \in \integers^d \setminus \integers_{\geq 0}^d,\text{ and } w \lexineq z) \\ &\text{or } (w \in \integers_{\geq 0}^d \text{ and } w_j \leq z_j \text{ for all } 1 \leq j \leq d) \end{aligned}),
\end{align}
% $\calF_{z} \coloneqq \sigma(\bigcup_{w \in \integers^d, w \lexineq z} \pppentire \cap Q_{w})$,
and $I_{n} \coloneqq \{z \in \integers^d : Q_{z} \cap W_{n} \neq \emptyset \}$.
For a function $\map{f}{E}{\reals}$ and $u, v > 0$,
let
\begin{gather}
  S_{z,f}^{(u, v)}(\calX) \coloneqq \sum_{\calZ \subset \calX \cap Q_{z}^{(2Lv_{\lilsubseq{s}{l}})}} \isolated{\calZ}{\calX \cap Q_{z}^{(3Lv_{\lilsubseq{s}{l}})}}{u t} \zeroed{f}(H(v^{-1} \calZ)), \\
  \Delta_{n,z,f}^{(u, v)} \coloneqq S_{z,f}^{(u, v)}(\pppentire_n) - S_{z,f}^{(u, v)}(\pppentire_{z, n}'),
  \quad \Delta_{\infty,z,f}^{(u, v)} \coloneqq S_{z,f}^{(u, v)}(\pppentire) - S_{z,f}^{(u, v)}(\pppentire_{z}'),
\end{gather}
where $\pppentire_{z, n}' $ and $\pppentire_z' $ is given by \eqref{eq:ppp.indep}.
Then, we can write
\begin{align}
  \pairing{f}{\centered{\xi_n}}
  &= \centered{\sum_{\calZ \subset \pppentire_n} \isolated{\calZ}{\pppentire_n}{v_n t} \zeroed{f}(H(v_n^{-1} \calZ))}
  = \sum_{z \in I_{n}} \condexpect[Big]{\Delta_{n,z,f}^{(v_{n}, v_{n})}}{\calF_z} \\
  &= \sum_{z \in I_{n}} \condexpect[Big]{\Delta_{\infty,z,f}^{(v_n, v_n)}}{\calF_z}
  + Z_{n, f}^{(1)},
\end{align}
where
\begin{align}
  Z_{n,f}^{(1)} \coloneqq
  \sum_{z \in I_{n}} \condexpect[none]{\Delta_{n,z,f}^{(v_n, v_n)} - \Delta_{\infty,z,f}^{(v_n, v_n)}}{\calF_z}. \label{eq:lil.upper.z1.def}
\end{align}
Furthermore, we decompose the difference as
\begin{align}
  \pairing{f}{\centered{\xi_n}} - \pairing{f}{\centered{\xi_{\lilsubseq{s}{l}}}} 
  = Z_{n,f}^{(1)} - Z_{\lilsubseq{s}{l},f}^{(1)} + Z_{n,f}^{(2)} + Z_{n,f}^{(3)} + Z_{n,f}^{(4)}, \label{eq:lil.upper.fxi.decomposition}
\end{align}
where
\begin{align}
  Z_{n,f}^{(2)} &\coloneqq \sum_{z \in I_{n} \setminus I_{\lilsubseq{s}{l}}} \condexpect[Big]{\Delta_{\infty,z,f}^{(v_{\lilsubseq{s}{l}}, v_{\lilsubseq{s}{l}})}}{\calF_z}, \label{eq:lil.upper.z2.def} \\
  Z_{n,f}^{(3)} &\coloneqq \sum_{z \in I_{n}} \condexpect[Big]{\Delta_{\infty,z,f}^{(v_{n}, v_{n})} - \Delta_{\infty,z,f}^{(v_{\lilsubseq{s}{l}}, v_{n})}}{\calF_z}, \label{eq:lil.upper.z3.def}\\
  Z_{n,f}^{(4)} &\coloneq \sum_{z \in I_{n}} \condexpect[Big]{\Delta_{\infty,z,f}^{(v_{\lilsubseq{s}{l}}, v_{n})} - \Delta_{\infty,z,f}^{(v_{\lilsubseq{s}{l}}, v_{\lilsubseq{s}{l}})}}{\calF_z}. \label{eq:lil.upper.z4.def}
\end{align}

For $Z_{n,f}^{(1)}, Z_{n,f}^{(2)}, Z_{n,f}^{(3)}$, and $Z_{n,f}^{(4)}$,
the following upper bounds hold.
Recall that $\tailfunc$ is given by \eqref{eq:tailfunc}.

\begin{lemma}
\label{lemma:lil.upper.z1}
Let $s \in (1, 2^{d(k - \alpha(k-1))/(\alpha - 1)})$, and let $f \in \bddmble(E)$ or $f = \iprod{\bfa}{\cdot}$ for some $\bfa \in \reals^m$.
Then, there exists a constant $C > 0$, independent of $s$, $n$, and $f$,
such that for all sufficiently large $l$, all $\lilsubseq{s}{l} \leq n \leq \lilsubseq{s}{l+1}$, and all $R > 0$,
\begin{align}
  \probab[big]{Z_{n,f}^{(1)} > R}
  \leq \exp\ab\bigg(- C C_f r_{\lilsubseq{s}{l}} \ldpspeed[\lilsubseq{s}{l}] \tailfunc\ab\bigg(\frac{R}{C^2 C_f r_{\lilsubseq{s}{l}} \ldpspeed[\lilsubseq{s}{l}] })), \label{eq:lil.upper.z1.tail}
\end{align}
where
\begin{align}
  C_f \coloneqq \int_{E} e^{\abs{f}} \,d\tau. \label{eq:lil.upper.const.f.1}
\end{align}
\end{lemma}

\begin{lemma}
\label{lemma:lil.upper.z2}
Let $s \in (1, 2^{d(k - \alpha(k-1))/(\alpha - 1)})$ and $f \in \bddmble(E)$ or $f = \iprod{\bfa}{\cdot}$ for some $\bfa \in \reals^m$.
Then, there exists a constant $C > 0$, independent of $s$, $l$ and $f$,
such that for all sufficiently large $l$ and all $R > 0$,
\begin{align}
\begin{split}
  &\probab[Big]{\max_{\lilsubseq{s}{l} \leq n \leq \lilsubseq{s}{l+1}} Z_{n,f}^{(2)} > R} \\
  &\quad \leq \exp\ab\bigg(- C C_f (1 - r_{\lilsubseq{s}{l+1}}/r_{\lilsubseq{s}{l}}) \ldpspeed[\lilsubseq{s}{l}] \tailfunc\ab\bigg(\frac{R}{C^2 C_f (1 - r_{\lilsubseq{s}{l+1}}/r_{\lilsubseq{s}{l}}) \ldpspeed[\lilsubseq{s}{l}]})),
\end{split} \label{eq:lil.upper.z2.tail.max}
\end{align}
where $C_f$ is given by \eqref{eq:lil.upper.const.f.1}.
\end{lemma}

\begin{lemma}
\label{lemma:lil.upper.z3}
Let $s \in (1, 2^{d(k - \alpha(k-1))/(\alpha - 1)})$ and $f \in \bddmble(E)$ or $f = \iprod{\bfa}{\cdot}$ for some $\bfa \in \reals^m$.
Then, there exists a constant $C > 0$, independent of $s, n$ and $f$,
such that for all sufficiently large $l$, $\lilsubseq{s}{l} \leq n \leq \lilsubseq{s}{l+1}$, and all $R > 0$,
\begin{align}
\begin{split}
  \probab[big]{Z_{n,f}^{(3)} > R}
  \leq \exp\ab\bigg(- C C_f \lilsubseq{s}{l} r_{\lilsubseq{s}{l}}^d \ldpspeed[\lilsubseq{s}{l}] \tailfunc\ab\bigg(\frac{R}{C^2 C_f \lilsubseq{s}{l} r_{\lilsubseq{s}{l}}^d \ldpspeed[\lilsubseq{s}{l}] })),
\end{split} \label{eq:lil.upper.z3.tail.max}
\end{align}
where $C_f$ is given by \eqref{eq:lil.upper.const.f.1}.
\end{lemma}

\begin{lemma}
\label{lemma:lil.upper.z4}
Let $s \in (1, 2^{d(k - \alpha(k-1))/(\alpha - 1)})$ and $f \in \bddcont(E)$.
Assume that $f$ is continuously differentiable with bounded gradient $\norm{\nabla f}_{\reals^m}$.
% and satisfies $f(\bfzero_m) = 0$.
Further assume that $H$ satisfies condition \textup{(H6)}.
Then, there exists a constant $C > 0$, independent of $s, l$ and $f$,
such that for all sufficiently large $l$ and all $R > 2 C (\suprdnorm{\nabla f} + \supnorm{f}) \ldpspeed[\lilsubseq{s}{l}]^{1/2}$,
\begin{align}
\begin{split}
  &\probab[Big]{\max_{\lilsubseq{s}{l} \leq n \leq \lilsubseq{s}{l+1}} Z_{n,f}^{(4)} > R} \\
  &\quad \leq \exp\ab\bigg(- C C_f' (1 - v_{\lilsubseq{s}{l+1}}/v_{\lilsubseq{s}{l}}) \ldpspeed[\lilsubseq{s}{l}] \tailfunc\ab\bigg(\frac{R}{2 C^2 C_f' (1 - v_{\lilsubseq{s}{l+1}}/v_{\lilsubseq{s}{l}})\ldpspeed[\lilsubseq{s}{l}] } )),
\end{split} \label{eq:lil.upper.z4.tail.max}
\end{align}
where
\begin{align}
  C_f' \coloneqq e^{\supnorm{f}} (\supnorm{f}^{-1}\suprdnorm{\nabla f} + \supnorm{f}^{-2}\suprdnorm{\nabla f}^2 + 2). \label{eq:lil.upper.const.f.2}
\end{align}
\end{lemma}

We will prove \cref{lemma:lil.upper.z1,lemma:lil.upper.z2,lemma:lil.upper.z3,lemma:lil.upper.z4} in \cref{sec:proof.zi.bound}.
Using these results, we establish \eqref{eq:lil.upper.meas}.

\begin{proof}[Proof of \eqref{eq:lil.upper.meas}]
It remains to show \eqref{eq:lil.upper.max.limsup}.
By the decomposition \eqref{eq:lil.upper.fxi.decomposition},
it suffices to verify that for every $\delta > 0$ and $i=1,2,3,4$,
\begin{align}
  \sum_{l=1}^{\infty} \probab[bigg]{\sum_{q=1}^{\infty} \frac{1}{2^q} \max_{\lilsubseq{s}{l} \leq n \leq \lilsubseq{s}{l+1}} \abs[big]{Z_{n,f_q}^{(i)}} > \delta \mdpfactor[\lilsubseq{s}{l}] } < \infty. \label{eq:lil.upper.max.z.summable}
\end{align}
Take $s > 1$ sufficiently close to $1$.
By \cref{lemma:lil.upper.z1,lemma:lil.upper.z2,lemma:lil.upper.z3},
there exists a constant $C > 0$, independent of $s, l$, and $q$,
such that for sufficiently large $l$ and all $R > 0$,
\begin{align}
  &\probab[bigg]{\max_{\lilsubseq{s}{l} \leq n \leq \lilsubseq{s}{l+1}} \abs[big]{Z_{n,f_q}^{(1)}} > R }\\
  &\quad \leq \probab[bigg]{\max_{\lilsubseq{s}{l} \leq n \leq \lilsubseq{s}{l+1}} Z_{n,f_q}^{(1)} > R } + \probab[bigg]{\max_{\lilsubseq{s}{l} \leq n \leq \lilsubseq{s}{l+1}} (-Z_{n,f_q}^{(1)}) > R } \\
  &\quad \leq 4(\lilsubseq{s}{l+1} - \lilsubseq{s}{l})\exp\ab\bigg(- e \tau(E)C r_{\lilsubseq{s}{l}} \ldpspeed[\lilsubseq{s}{l}] \tailfunc\ab\bigg(\frac{R}{e \tau(E)C^2 r_{\lilsubseq{s}{l}} \ldpspeed[\lilsubseq{s}{l}] })),
\end{align}
and similarly,
\begin{align}
  &\probab[bigg]{\max_{\lilsubseq{s}{l} \leq n \leq \lilsubseq{s}{l+1}} \abs[big]{Z_{n,f_q}^{(2)}} > R } \\
   &\quad \leq 4\exp\ab\bigg(- e \tau(E)C (1 - r_{\lilsubseq{s}{l+1}}/r_{\lilsubseq{s}{l}}) \ldpspeed[\lilsubseq{s}{l}] \tailfunc\ab\bigg(\frac{R}{e \tau(E)C^2 (1 - r_{\lilsubseq{s}{l+1}}/r_{\lilsubseq{s}{l}}) \ldpspeed[\lilsubseq{s}{l}]})), \\
  &\probab[bigg]{\max_{\lilsubseq{s}{l} \leq n \leq \lilsubseq{s}{l+1}} \abs[big]{Z_{n,f_q}^{(3)}} > R } \\
   &\quad \leq 4(\lilsubseq{s}{l+1} - \lilsubseq{s}{l})\exp\ab\bigg(- e \tau(E)C \lilsubseq{s}{l} r_{\lilsubseq{s}{l}}^d \ldpspeed[\lilsubseq{s}{l}] \tailfunc\ab\bigg(\frac{R}{e \tau(E)C^2 \lilsubseq{s}{l} r_{\lilsubseq{s}{l}}^d \ldpspeed[\lilsubseq{s}{l}] })).
\end{align}
Note that the constant $C_{f_q}$ given by \eqref{eq:lil.upper.const.f.1} is uniformly bounded as $C_{f_q} \leq e \tau(E) < \infty$ for all $q \geq 1$, since $\supnorm{f_q} = 1$.

Next, we address the term $Z_{n,f_q}^{(4)}$.
By \eqref{eq:test.func.grad.bound.cond}, for all sufficiently large $l$, we have
$a_{2\fqexp} q^{2\fqexp+1} \delta \mdpfactor[\lilsubseq{s}{l}] / 4 \geq 2 C (q^{\fqexp} + 1) \ldpspeed[\lilsubseq{s}{l}]^{1/2}
\geq 2 C (\suprdnorm{\nabla f_q} + \supnorm{f_q}) \ldpspeed[\lilsubseq{s}{l}]^{1/2}$,
where $a_{2\fqexp}$ is defined by \eqref{eq:lil.upper.tail.sum.const}.

Thus, applying \cref{lemma:lil.upper.z4},
we have for $R \geq a_{2\fqexp} q^{2\fqexp+1} \delta \mdpfactor[\lilsubseq{s}{l}] /4$,
\begin{align}
  &\probab[bigg]{\max_{\lilsubseq{s}{l} \leq n \leq \lilsubseq{s}{l+1}} \abs[big]{Z_{n,f_q}^{(4)}} \geq R } \\
  &\quad \leq 4\exp\ab\bigg(- 4e C (1 - v_{\lilsubseq{s}{l+1}}/v_{\lilsubseq{s}{l}}) \ldpspeed[\lilsubseq{s}{l}] q^{2\fqexp} \tailfunc\ab\bigg(\frac{R}{8e C^2 (1 - v_{\lilsubseq{s}{l+1}}/v_{\lilsubseq{s}{l}})\ldpspeed[\lilsubseq{s}{l}]q^{2\fqexp}} )).
\end{align}
Note that the constant $C_{f_q}'$ given by \eqref{eq:lil.upper.const.f.2} is bounded as $C_{f_q}' \leq 4e q^{2\fqexp}$ for $q \geq 1$.

We choose the variables in \cref{prop:lil.upper.tail.sum.finite} as follows:
$Y_{l, q} =
\max_{\lilsubseq{s}{l} \leq n \leq \lilsubseq{s}{l+1}}
\abs{Z_{n,f_q}^{(i)}}$, $i=1,2,3,4$, and
\begin{align}
  (A_l, B_l, C_l, \beta, a') =
  \begin{cases}
    \ab\big(\delta \mdpfactor[\lilsubseq{s}{l}],\, e \tau(E)C r_{\lilsubseq{s}{l}} \ldpspeed[\lilsubseq{s}{l}],\, 4(\lilsubseq{s}{l+1} - \lilsubseq{s}{l}),\, 0,\, C) & i=1, \\
    \ab\big(\delta \mdpfactor[\lilsubseq{s}{l}],\, e \tau(E)C (1 - r_{\lilsubseq{s}{l+1}}/r_{\lilsubseq{s}{l}}) \ldpspeed[\lilsubseq{s}{l}],\, 4,\, 0,\, C) & i=2, \\
    \ab\big(\delta \mdpfactor[\lilsubseq{s}{l}],\, e \tau(E)C \lilsubseq{s}{l} r_{\lilsubseq{s}{l}}^d \ldpspeed[\lilsubseq{s}{l}],\, 4(\lilsubseq{s}{l+1} - \lilsubseq{s}{l}),\, 0,\, C) & i=3, \\
    \ab\big(\delta \mdpfactor[\lilsubseq{s}{l}],\, 4e C (1 - v_{\lilsubseq{s}{l+1}}/v_{\lilsubseq{s}{l}}) \ldpspeed[\lilsubseq{s}{l}],\, 4,\, 2\fqexp,\, 2C) & i=4.
  \end{cases}
\end{align}
With these choices and the bounds above, assumption \eqref{eq:lil.upper.tail.sum.ass} is satisfied for $s > 1$ sufficiently close to $1$.
Note that $r_{\lilsubseq{s}{l+1}}/r_{\lilsubseq{s}{l}} \to s^{-\alpha/d(k - \alpha(k-1))}$ and $v_{\lilsubseq{s}{l+1}}/v_{\lilsubseq{s}{l}} \to s^{- (\alpha-1)/d(k - \alpha(k-1))}$ as $l \to \infty$.
Thus, we obtain \eqref{eq:lil.upper.max.z.summable}.
\end{proof}

To establish the summability from the obtained tail estimates,
we use the following proposition.

\begin{proposition}
\label{prop:lil.upper.tail.sum.finite}
Let $\{Y_{j, q}\}_{j, q \geq 1}$ be real-valued random variables.
Let $\beta \geq 0$ and
\begin{align}
  a_{\beta} \coloneqq \min_{q \in \naturals} \frac{2^q}{q^{\beta + 3}}. \label{eq:lil.upper.tail.sum.const}
\end{align}
Assume that there exist sequences of positive numbers $\{A_j\}_{j \geq 1}$, $\{B_j\}_{j \geq 1}$, $\{C_j\}_{j\geq 1}$, and a constant $a' > 0$ such that for all $j, q \geq 1$ and $R \geq a_{\beta} q^{\beta+1} A_j/4$,
\begin{align}
  \probab[none]{Y_{j, q} \geq R} \leq C_{j} \exp\ab\Big(- q^{\beta} B_j \tailfunc\ab\Big(\frac{R}{a' q^{\beta} B_j})),
\end{align}
where $\tailfunc$ is given by \eqref{eq:tailfunc}.
Further assume that $\inf_{j \geq 1} C_j > 0$ and
\begin{align}
  \sum_{j = 1}^{\infty} C_j \exp\ab\Big(- \frac{a_{\beta}^2 A_j^2}{64 a'^2 B_j}) < \infty \quad
  \text{and} \quad 
  \sum_{j = 1}^{\infty} C_j \exp\ab\Big(- \frac{a_{\beta} A_j}{16 a'}) < \infty. \label{eq:lil.upper.tail.sum.ass}
\end{align}
Then,
\begin{align}
  \sum_{j = 1}^{\infty} \probab[bigg]{\sum_{q = 1}^{\infty} \frac{1}{2^q} Y_{j, q} \geq A_j } < \infty. \label{eq:lil.upper.tail.sum.finite}
\end{align}

\end{proposition}

\begin{proof}
Since $2^{q}/ q^{2} \geq a_{\beta} q^{\beta+1}$ for $q \in \naturals$,
the following set inclusions hold:
\begin{align}
  \Bab{\sum_{q=1}^{\infty} \frac{1}{2^q} Y_{j, q} > A_j }
  \subset \bigcup_{q=1}^{\infty} \Bab{\frac{1}{2^q} Y_{j, q} > \frac{6}{\pi^2 q^2} A_j  }
  \subset \bigcup_{q=1}^{\infty} \Bab{Y_{j, q} > \frac{6 a_{\beta} q^{\beta+1} A_j}{\pi^2}}
  \subset \bigcup_{q=1}^{\infty} \Bab{Y_{j, q} > \frac{a_{\beta} q^{\beta+1} A_j}{4}}.
\end{align}
Thus, we have
\begin{align}
  \probab[bigg]{\sum_{q = 1}^{\infty} \frac{1}{2^q}Y_{j, q} \geq A_j }
  &\leq \sum_{q= 1}^{\infty} \probab[bigg]{ Y_{j, q} \geq \frac{a_{\beta} q^{\beta+1} A_j}{4}  }
  \leq \sum_{q= 1}^{\infty}C_j \exp\ab\Big(- q^{\beta} B_j \tailfunc\ab\Big(\frac{a_{\beta} q^{\beta+1} A_j}{4 a' q^{\beta} B_j})). \label{eq:tail.summable.1}
\end{align}
By \eqref{eq:tailfunc.lower.bound},
\eqref{eq:tail.summable.1} is bounded by
\begin{align}
  &C_j\sum_{q = 1}^{\infty} \ab\bigg(\exp\ab\Big(- q^{\beta} B_j \frac{1}{4}\ab\Big(\frac{a_{\beta} q^{\beta+1} A_j}{4 a' q^{\beta} B_j})^2)
   + \exp\ab\Big(- q^{\beta} B_j \frac{1}{4}\ab\Big(\frac{a_{\beta} q^{\beta+1} A_j}{4 a' q^{\beta} B_j}))) \\
  &\quad \leq C_j\sum_{q = 1}^{\infty} \exp\ab\Big(- \frac{a_{\beta}^2 A_j^2}{64 a'^2 B_j} q)
  + C_j\sum_{q = 1}^{\infty} \exp\ab\Big(- \frac{a_{\beta} A_j}{16 a'} q) \\
  % &\leq C_j \frac{\exp\ab\Big(- \frac{a_{\beta}^2 A_j^2}{64 a'^2 B_j})}{1 - \exp\ab\Big(- \frac{a_{\beta}^2 A_j^2}{64 a'^2 B_j})}
  % + C_j \frac{\exp\ab\Big(- \frac{a_{\beta} A_j}{16 a'})}{1 - \exp\ab\Big(- \frac{a_{\beta} A_j}{16 a'})}.
  &\quad \leq C_j \frac{\exp(- a_{\beta}^2 A_j^2/(64 a'^2 B_j))}{1 - \exp(- a_{\beta}^2 A_j^2/(64 a'^2 B_j))}
  + C_j \frac{\exp(- a_{\beta} A_j/(16 a'))}{1 - \exp(- a_{\beta} A_j/(16 a'))}.
\end{align}
Since \eqref{eq:lil.upper.tail.sum.ass} holds and $\inf_{j \geq 1} C_j > 0$,
we have $\exp(- a_{\beta}^2 A_j^2 / (64 a'^2 B_j)) \to 0$ and $\exp(- a_{\beta} A_j / (16 a')) \to 0$ as $j \to \infty$.
Summing over $j \ge 1$ thus yields \eqref{eq:lil.upper.tail.sum.finite}.
\end{proof}

Next, we prove \eqref{eq:lil.upper.vec}.
We need the following lemma, which modifies \cref{lemma:lil.upper.z4} for the choice $f = \iprod{\bfe_i}{\cdot}$.

\begin{lemma}
\label{lemma:lil.upper.z4.ei}
Let $s \in (1, 2^{d(k - \alpha(k-1))/(\alpha - 1)})$ and $f = \iprod{\bfe_i}{\cdot}$, $i=1,\dots, m$.
Assume that $H$ satisfies condition \textup{(H6)}.
Then, there exists a constant $C > 0$, independent of $s$, $l$ and $i$,
such that for all sufficiently large $l$ and all $R > C \ldpspeed[\lilsubseq{s}{l}]^{1/2}$,
\begin{align}
  \probab[Big]{\max_{\lilsubseq{s}{l} \leq n \leq \lilsubseq{s}{l+1}} Z_{n,f}^{(4)} > R}
  \leq C\exp\ab\bigg(- \frac{R^2}{C (1 - v_{\lilsubseq{s}{l+1}}/v_{\lilsubseq{s}{l}}) \ldpspeed[\lilsubseq{s}{l}] })
  + C\exp\ab\bigg(- \frac{R}{Cl}). \label{eq:lil.upper.z4.ei.tail.max}
\end{align}
\end{lemma}

We will prove \cref{lemma:lil.upper.z4.ei} in \cref{sec:proof.zi.bound}.
With \cref{lemma:lil.upper.z4.ei},
we are now ready to prove \eqref{eq:lil.upper.vec}.

\begin{proof}[Proof of \eqref{eq:lil.upper.vec}]
The proof is similar to that of \eqref{eq:lil.upper.meas}.
As in \cref{lem:lil.subseqence.cptness},
for all $s > 1$, it holds almost surely that
\begin{align}
  \lim_{l \to \infty} \distancevec(\mdpfactor[\lilsubseq{s}{l}]^{-1 } \Phi(\centered{\xi_{\lilsubseq{s}{l}}}), \accptvec) &= 0
\end{align}
and
\begin{align}
  &\limsup_{n \to \infty} \distancevec(\mdpfactor^{-1} \Phi(\centered{\xi_{n}}), \accptvec) \\
  &\quad \leq \limsup_{l \to \infty} \mdpfactor[\lilsubseq{s}{l}]^{-1} \max_{\lilsubseq{s}{l} \leq n \leq \lilsubseq{s}{l+1}} \norm{\Phi(\centered{\xi_{n}}) - \Phi(\centered{\xi_{\lilsubseq{s}{l}}})}_{\reals^m}
   + (1 - s^{-1/2})\sup_{\bfu \in \accptvec} \norm{\bfu}_{\reals^m} \\
  &\quad \leq \sum_{i=1}^m \limsup_{l \to \infty} \mdpfactor[\lilsubseq{s}{l}]^{-1} \max_{\lilsubseq{s}{l} \leq n \leq \lilsubseq{s}{l+1}} \abs{\pairing{f_i}{\centered{\xi_{n}}} - \pairing{f_i}{\centered{\xi_{\lilsubseq{s}{l}}}}}
  + (1 - s^{-1/2})\sup_{\bfu \in \accptvec} \norm{\bfu}_{\reals^m}, \label{eq:lil.upper.subseq.bound.vec}
\end{align}
where $f_i \coloneqq \iprod{\bfe_i}{\cdot}$.
Consequently, it remains to show that for all $\delta > 0$ and $i=1,\dots, m$,
there exists $s > 1$ such that almost surely,
\begin{align}
  \limsup_{l \to \infty} \mdpfactor[\lilsubseq{s}{l}]^{-1} \max_{\lilsubseq{s}{l} \leq n \leq \lilsubseq{s}{l+1}} \abs[big]{\pairing{f_i}{\centered{\xi_{n}}} - \pairing{f_i}{\centered{\xi_{\lilsubseq{s}{l}}}}} &\leq \delta. \label{eq:lil.upper.max.limsup.vec}
\end{align}
Furthermore, to establish \eqref{eq:lil.upper.max.limsup.vec},
it is enough to show that for all $\delta > 0$ and $j=1,2,3,4$,
\begin{align}
  \sum_{l=1}^{\infty} \probab[bigg]{\max_{\lilsubseq{s}{l} \leq n \leq \lilsubseq{s}{l+1}} \abs[big]{Z_{n,f_i}^{(j)}} > \delta \mdpfactor[\lilsubseq{s}{l}] } < \infty. \label{eq:lil.upper.max.z.summable.vec}
\end{align}
As in the proof of \eqref{eq:lil.upper.meas},
for $s > 1$ sufficiently close to $1$,
by \cref{lemma:lil.upper.z1,lemma:lil.upper.z2,lemma:lil.upper.z3},
we have for sufficiently large $l$,
\begin{align}
  &\probab[bigg]{\max_{\lilsubseq{s}{l} \leq n \leq \lilsubseq{s}{l+1}} \abs[big]{Z_{n,f_i}^{(1)}} > \delta \mdpfactor[\lilsubseq{s}{l}] } \\
   &\quad \leq 4(\lilsubseq{s}{l+1} - \lilsubseq{s}{l})\exp\ab\bigg(- C C_{f_i}  r_{\lilsubseq{s}{l}} \ldpspeed[\lilsubseq{s}{l}] \tailfunc\ab\bigg(\frac{\delta \mdpfactor[\lilsubseq{s}{l}]}{C^2 C_{f_i} r_{\lilsubseq{s}{l}} \ldpspeed[\lilsubseq{s}{l}] })), \\
  &\probab[bigg]{\max_{\lilsubseq{s}{l} \leq n \leq \lilsubseq{s}{l+1}} \abs[big]{Z_{n,f_i}^{(2)}} > \delta \mdpfactor[\lilsubseq{s}{l}] } \\
   &\quad \leq 4\exp\ab\bigg(- C C_{f_i} (1 - r_{\lilsubseq{s}{l+1}}/r_{\lilsubseq{s}{l}}) \ldpspeed[\lilsubseq{s}{l}] \tailfunc\ab\bigg(\frac{\delta \mdpfactor[\lilsubseq{s}{l}]}{C^2 C_{f_i} (1 - r_{\lilsubseq{s}{l+1}}/r_{\lilsubseq{s}{l}}) \ldpspeed[\lilsubseq{s}{l}]})), \\
  &\probab[bigg]{\max_{\lilsubseq{s}{l} \leq n \leq \lilsubseq{s}{l+1}} \abs[big]{Z_{n,f_i}^{(3)}} > \delta \mdpfactor[\lilsubseq{s}{l}] } \\
   &\quad \leq 4(\lilsubseq{s}{l+1} - \lilsubseq{s}{l})\exp\ab\bigg(- C C_{f_i} \lilsubseq{s}{l} r_{\lilsubseq{s}{l}}^d \ldpspeed[\lilsubseq{s}{l}] \tailfunc\ab\bigg(\frac{\delta \mdpfactor[\lilsubseq{s}{l}]}{C^2 C_{f_i} \lilsubseq{s}{l} r_{\lilsubseq{s}{l}}^d \ldpspeed[\lilsubseq{s}{l}] })).
\end{align}
Next, we bound $Z_{n,f_i}^{(4)}$.
For sufficiently large $l$, we have $\delta \mdpfactor[\lilsubseq{s}{l}] > C \ldpspeed[\lilsubseq{s}{l}]^{1/2}$,
and by \cref{lemma:lil.upper.z4.ei},
\begin{align}
  \probab[bigg]{\max_{\lilsubseq{s}{l} \leq n \leq \lilsubseq{s}{l+1}} \abs[big]{Z_{n,f_i}^{(4)}} > \delta \mdpfactor[\lilsubseq{s}{l}] }
  \leq C\exp\ab\bigg(- \frac{\delta^2 \mdpfactor[\lilsubseq{s}{l}]^2}{C (1 - v_{\lilsubseq{s}{l+1}}/v_{\lilsubseq{s}{l}}) \ldpspeed[\lilsubseq{s}{l}] })
  + C\exp\ab\bigg(- \frac{\delta \mdpfactor[\lilsubseq{s}{l}]}{Cl}).
\end{align}
Consequently, for $s > 1$ sufficiently close to $1$,
these probabilities are summable over $l \geq 1$.
Thus, we obtain \eqref{eq:lil.upper.max.z.summable.vec}.

\end{proof}

\subsection{Proof of the bounds for $Z_{n, f}^{(1)}, Z_{n, f}^{(2)}, Z_{n, f}^{(3)}, Z_{n, f}^{(4)}$}
\label{sec:proof.zi.bound}

In this subsection, we prove \cref{lemma:lil.upper.z1,lemma:lil.upper.z2,lemma:lil.upper.z3,lemma:lil.upper.z4,lemma:lil.upper.z4.ei},
which provide upper bounds for $Z_{n,f}^{(i)}$ ($i=1,2,3,4$).
Note that $v_{\lilsubseq{s}{l}} / v_{\lilsubseq{s}{l+1}} \to s^{(\alpha - 1)/d(k - \alpha(k-1))}$ as $l \to \infty$.
Thus, for any $s \in (1, 2^{d(k - \alpha(k-1))/(\alpha - 1)})$,
we have $v_{\lilsubseq{s}{l}} \leq 2 v_{\lilsubseq{s}{l+1}}$ for all sufficiently large $l$.

\begin{proof}[Proof of \cref{lemma:lil.upper.z1}]
If $z \in \integers^d$ satisfies $Q_{z}^{(3L v_{\lilsubseq{s}{l}})} \subset W_n $,
then $\Delta_{n,z,f}^{(v_n, v_n)} = \Delta_{\infty,z,f}^{(v_n, v_n)} $.
Thus, defining $J_{n} \coloneqq \{z \in I_n : Q_{z}^{(3L v_{\lilsubseq{s}{l}})} \not\subset W_n \} $,
we can write $Z_{n,f}^{(1)} = \sum_{z \in J_n} \condexpect[none]{\Delta_{n,z,f}^{(v_n, v_n)} - \Delta_{\infty,z,f}^{(v_n, v_n)}}{\calF_z} $.
Observe that for any $\lilsubseq{s}{l} \leq n \leq \lilsubseq{s}{l+1} $ and any integer $p \geq 2$,
we have the moment bound
\begin{align}
  &\expect[biggg]{\sum_{\calZ \subset \pppentire \cap Q_{z}^{(2Lv_{\lilsubseq{s}{l}})}} \abs{\zeroed{f}(H(v_n^{-1} \calZ))}^p } \\
  &\quad \leq \frac{v_n^{d(k-1)}}{k!} \int_{Q_{z}^{(2Lv_{\lilsubseq{s}{l}})}} dx \int_{(\reals^d)^{k-1}} d\bfy\, \abs{\zeroed{f}(H(\bfzero_d, \bfy))}^p \\
  &\quad \leq \frac{v_{\lilsubseq{s}{l}}^{d(k-1)}}{k!} \ab\big(t d^{-1/2}v_{\lilsubseq{s}{l+1}} +  2Lv_{\lilsubseq{s}{l}})^d p! \int_{(\reals^d)^{k-1}}  e^{\abs{f(H(\bfzero_d, \bfy))}} \indicator[none]{H(\bfzero_d, \bfy) \neq \bfzero_m} \,d\bfy\\
  % = C v^{dk} \int_{E} (e^{f(x)} - 1) \tau(dx)
  &\quad \leq p! C C_f v_{\lilsubseq{s}{l}}^{dk}. \label{eq:lil.p.moment.common}
\end{align}
Therefore, combining \eqref{eq:lil.p.moment.common} with \cref{prop:exp.moment.common},
we obtain that for all $n$ in the range $\lilsubseq{s}{l} \leq n \leq \lilsubseq{s}{l+1} $ and all $\lambda > 0$,
\begin{align}
 \probab[big]{Z_{n,f}^{(1)} > R}
 \leq e^{-\lambda R}\expect[none]{\exp(\lambda Z_{n,f}^{(1)})}
 \leq e^{-\lambda R} \exp\ab\bigg(C C_f \card{J_n}v_{\lilsubseq{s}{l}}^{dk} \sum_{p=2}^{\infty} (C\lambda)^p ). \label{eq:lil.upper.boundary.tail.bound}
\end{align}
Since $\card{J_n} \leq C (n^{1/d}/v_{\lilsubseq{s}{l}})^{d-1} \leq C r_{\lilsubseq{s}{l}}^{-(d-1)} $,
the right-hand side of \eqref{eq:lil.upper.boundary.tail.bound} is bounded by
\begin{align}
  e^{-\lambda R}\exp\ab\bigg(C C_f \lilsubseq{s}{l}^{k} r_{\lilsubseq{s}{l}}^{d(k-1) +1} \sum_{p=2}^{\infty} (C\lambda)^p)
  = e^{-\lambda R}\exp\ab\Big(C C_f r_{\lilsubseq{s}{l}} \ldpspeed[\lilsubseq{s}{l}] \frac{C^2 \lambda^2}{1 - C\lambda}),
\end{align}
for $0 < \lambda < C^{-1} $.
Optimizing over $\lambda$,
we obtain \eqref{eq:lil.upper.z1.tail}.
\end{proof}

\begin{proof}[Proof of \cref{lemma:lil.upper.z2}]
Observe that the sequence $\{Z_{n,f}^{(2)}\}_{\lilsubseq{s}{l} \leq n \leq \lilsubseq{s}{l+1}} $
forms a martingale with respect to the filtration
\begin{align}
  \Bab{\calG_{n}}_{\lilsubseq{s}{l} \leq n \leq \lilsubseq{s}{l+1}} \coloneqq
  \Bab{\sigma\ab\bigg(\bigcup_{z \in I_{n} } \calF_{z}) }_{\lilsubseq{s}{l} \leq n \leq \lilsubseq{s}{l+1}}. \label{eq:lil.filtration}
\end{align}
Thus, applying Doob's inequality, we see that for all $\lambda > 0$,
\begin{align}
  \probab[Big]{\max_{\lilsubseq{s}{l} \leq n \leq \lilsubseq{s}{l+1}} Z_{n,f}^{(2)} > R}
  \leq e^{- \lambda R} \expect[none]{\exp(\lambda Z_{\lilsubseq{s}{l+1}, f}^{(2)})}. \label{eq:lil.upper.mart.tail.bound}
\end{align}
It then follows from \eqref{eq:lil.p.moment.common} and \cref{prop:exp.moment.common} that
\begin{align}
  \expect[none]{\exp(\lambda Z_{\lilsubseq{s}{l+1}, f}^{(2)})}
  \leq \exp\ab\bigg(C C_f \card{I_{\lilsubseq{s}{l+1}} \setminus I_{\lilsubseq{s}{l}}} v_{\lilsubseq{s}{l+1}}^{dk} \sum_{p=2}^{\infty} (C\lambda)^p ). 
\end{align}
Furthermore, for sufficiently large $l$, 
\begin{align}
  \card{I_{\lilsubseq{s}{l+1}} \setminus I_{\lilsubseq{s}{l}}}
  &= \roundup{{\lilsubseq{s}{l+1}}^{1/d}/v_{\lilsubseq{s}{l+1}}td^{-1/2}}^d - \roundup{{\lilsubseq{s}{l}}^{1/d}/v_{\lilsubseq{s}{l+1}}td^{-1/2}}^d \\
  &= \roundup{t^{-1}d^{1/2} r_{\lilsubseq{s}{l+1}}^{-1}}^d - \roundup{t^{-1}d^{1/2} r_{\lilsubseq{s}{l}}^{-1}}^d \\
  &\leq C r_{\lilsubseq{s}{l+1}}^{-(d-1)} (r_{\lilsubseq{s}{l+1}}^{-1} - r_{\lilsubseq{s}{l}}^{-1}).
\end{align}
Consequently, \eqref{eq:lil.upper.mart.tail.bound} is bounded by
\begin{align}
  &e^{-\lambda R} \exp\ab\bigg(C C_f \lilsubseq{s}{l+1}^k r_{\lilsubseq{s}{l+1}}^{d(k-1)}  (1 - r_{\lilsubseq{s}{l+1}} / r_{\lilsubseq{s}{l}})\sum_{p=2}^{\infty} (C\lambda)^p ) \\
  &\quad = e^{-\lambda R} \exp\ab\Big(C C_f (1 - r_{\lilsubseq{s}{l+1}} / r_{\lilsubseq{s}{l}}) \ldpspeed[\lilsubseq{s}{l}] \frac{C^2 \lambda^2}{1 - C\lambda}),
\end{align}
for any $0 < \lambda < C^{-1} $.
By choosing the optimal $\lambda$,
we arrive at \eqref{eq:lil.upper.z2.tail.max}.
\end{proof}

\begin{proof}[Proof of \cref{lemma:lil.upper.z3}]
We begin by observing the identity
\begin{align}
  &S_{z,f}^{(v_{n}, v_{n})}(\calX) - S_{z,f}^{(v_{\lilsubseq{s}{l}}, v_{n})}(\calX) \\
  &= \sum_{\calZ \subset \calX \cap Q_{z}^{(2Lv_{\lilsubseq{s}{l}})}} 
  \isolated{\calZ}{\calX \cap Q_{z}^{(3Lv_{\lilsubseq{s}{l}})}}{v_{n} t} 
  \ab\big(1 - \isolated{\calZ}{\calX \cap Q_{z}^{(3Lv_{\lilsubseq{s}{l}})}}{v_{\lilsubseq{s}{l}} t}) \zeroed{f}(H(v_{n}^{-1} \calZ)),
\end{align}
which implies that for any integer $p \geq 2$,
\begin{align}
  &\expect[biggg]{\sum_{\calZ \subset \pppentire \cap Q_{z}^{(2Lv_{\lilsubseq{s}{l}})}} 
  \ab\big(1 - \isolated{\calZ}{\pppentire \cap Q_{z}^{(3Lv_{\lilsubseq{s}{l}})}}{v_{\lilsubseq{s}{l}} t}) \zeroed{f}(H(v_{n}^{-1} \calZ))^p } \\
  &\quad \leq \frac{1}{k!} \int_{\ab\big(Q_{z}^{(2Lv_{\lilsubseq{s}{l}})})^{k}} \zeroed{f}(H(v_n^{-1}\bfx))^p \expect[big]{1 - \isolated{\bfx}{\pppentire \cap Q_{z}^{(3Lv_{\lilsubseq{s}{l}})}}{v_{\lilsubseq{s}{l}} t}} \,d\bfx. \label{eq:lil.upper.rad.p.moment}
\end{align}
Since
\begin{align}
  \expect[big]{1 - \isolated{\bfx}{\pppentire \cap Q_{z}^{(3Lv_{\lilsubseq{s}{l}})}}{v_{\lilsubseq{s}{l}} t}}
  = 1 - \exp\ab\big(- \lebmeas(Q_{z}^{(3Lv_{\lilsubseq{s}{l}})} \cap \ball{\bfx}{v_{\lilsubseq{s}{l}} t}))
  \leq C v_{\lilsubseq{s}{l}}^{d},
\end{align}
we can bound \eqref{eq:lil.upper.rad.p.moment} by $p! C C_f v_{\lilsubseq{s}{l}}^{d(k+1)}$,
using a calculation similar to that in \eqref{eq:lil.p.moment.common}.
Therefore, by applying \cref{prop:exp.moment.common} to \eqref{eq:lil.upper.rad.p.moment},
we find that
\begin{align}
  \probab[big]{ Z_{n,f}^{(3)} > R}
  \leq e^{-\lambda R} \exp\ab\bigg(C C_f \card{I_{n}} v_{\lilsubseq{s}{l}}^{d(k+1)} \sum_{p=2}^{\infty} (C\lambda)^p) 
  \leq e^{-\lambda R} \exp\ab\Big(C C_f \lilsubseq{s}{l} r_{\lilsubseq{s}{l}}^d  \ldpspeed[\lilsubseq{s}{l}] \frac{C^2 \lambda^2}{1 - C\lambda})
\end{align}
for all $0 < \lambda< C^{-1}$.
Optimizing with respect to $\lambda$,
we obtain \eqref{eq:lil.upper.z3.tail.max}.
\end{proof}

\begin{proof}[Proof of \cref{lemma:lil.upper.z4}]
\textit{Step} 1.
To evaluate the sum in $Z_{n,f}^{(4)}$ over an index set independent of $n$,
we make the following observation.
For each $z \in \integers^d$, define
\begin{align}
  Y_{n, z} \coloneqq \condexpect[Big]{\Delta_{\infty,z,f}^{(v_{\lilsubseq{s}{l}}, v_{n})} - \Delta_{\infty,z,f}^{(v_{\lilsubseq{s}{l}}, v_{\lilsubseq{s}{l}})}}{\calF_z}, \label{eq:lil.upper.z4.y}
\end{align}
and for $j = \lilsubseq{s}{l}, \dots, \lilsubseq{s}{l+1}$,
let
\begin{align}
  V_{j} \coloneqq \max_{\lilsubseq{s}{l} \leq n \leq \lilsubseq{s}{l+1}} \sum_{z \in I_{j}} Y_{n, z}.
\end{align}
We claim that the process $\{V_{j}\}_{\lilsubseq{s}{l} \leq j \leq \lilsubseq{s}{l+1}}$ forms a submartingale with respect to the filtration \eqref{eq:lil.filtration}.
Indeed, since $\condexpect[none]{Y_{n, z}}{\calG_{j}} = 0$ for any $z \in I_{j+1} \setminus I_j$,
it follows that for all $j = \lilsubseq{s}{l}, \dots, \lilsubseq{s}{l+1}-1$
and all $n = \lilsubseq{s}{l}, \dots, \lilsubseq{s}{l+1}$,
\begin{align}
  \condexpect[big]{V_{j+1}}{\calG_j}
  \geq \condexpect[Bigg]{\sum_{z \in I_{j+1}} Y_{n, z} }{\calG_{j}}
  = \sum_{z \in I_{j}} Y_{n, z}.
\end{align}
Taking the maximum over $n$ on both sides,
we obtain $\condexpect[none]{V_{j+1}}{\calG_j} \geq V_{j}$.
Consequently, by Doob's inequality, for any $\lambda > 0$, we obtain
\begin{align}
  \probab[Big]{\max_{\lilsubseq{s}{l} \leq n \leq \lilsubseq{s}{l+1}} Z_{n,f}^{(4)} > R}
  \leq \probab[Big]{\max_{\lilsubseq{s}{l} \leq j \leq \lilsubseq{s}{l+1}} V_{j} > R}
  % &\leq e^{-\lambda R} \expect[bigg]{\max_{\lilsubseq{s}{l} \leq j \leq \lilsubseq{s}{l+1}} \exp\ab\big(\lambda W_{j})}
  \leq e^{-\lambda R} \expect[Big]{\exp\ab\big(\lambda V_{\lilsubseq{s}{l+1}})}. \label{eq:lil.upper.h.diff.bound}
\end{align}

\textit{Step} 2.
We next evaluate \eqref{eq:lil.upper.h.diff.bound}.
Let $B \coloneqq \roundup{12 \sqrt{d}L/t}$,
and for each $w \in \{0, 1, \dots, B-1\}^{d}$,
define $J_{w} \coloneqq I_{\lilsubseq{s}{l+1}} \cap (B \integers^d + w)$ and
\begin{align}
  U_{w}
  \coloneqq \max_{\lilsubseq{s}{l} \leq n \leq \lilsubseq{s}{l+1}} \sum_{z \in J_w} Y_{n, z}.
\end{align}
Then, by H\"{o}lder's inequality, we have
\begin{align}
  \log \expect[Big]{\exp\ab\big(\lambda V_{\lilsubseq{s}{l+1}})}
  &\leq B^{-d} \sum_{w \in \{0, 1, \dots, B-1\}^{d}} \log \expect[Big]{\exp\ab\big(B^d \lambda U_{w})}.  \label{eq:lil.upper.v.exp.moment}
\end{align}
For any fixed $n$ in the range $\lilsubseq{s}{l} \leq n \leq \lilsubseq{s}{l+1}$,
the random variables $\{Y_{n, z}\}_{z \in J_{w}}$ are independent and identically distributed.
Furthermore, by an argument analogous to that in \cref{prop:exp.eq.xi.kappa},
we have $\abs{Y_{n, z}} \leq C \supnorm{f}$ for some constant $C > 0$ independent of $s, n, z$, and $f$.
Therefore, applying Bousquet's inequality (\cref{prop:bousquet.ineq}) yields
\begin{align}
  \log \expect[Big]{\exp\ab\big(B^d \lambda U_{w})}
  &\leq B^d\expect[big]{U_{w}}\lambda
   + A_{w}\ab\big(e^{B^d C\supnorm{f}\lambda} - B^d C\supnorm{f} \lambda - 1), \label{eq:lil.upper.u.exp.moment}
\end{align}
where
\begin{align}
  A_{w} \coloneqq 2 (C \supnorm{f})^{-1} \expect[big]{U_{w}} + (C \supnorm{f})^{-2} \max_{\lilsubseq{s}{l} \leq n \leq \lilsubseq{s}{l+1}} \sum_{z \in J_{w}} \expect[none]{Y_{n, z}^2}.
\end{align}

\textit{Step} 3.
We evaluate $A_w$.
Under condition (H6),
since $f$ is continuously differentiable,
for $\lilsubseq{s}{l} \leq n \leq \lilsubseq{s}{l+1}$, we have
\begin{align}
  \expect[none]{Y_{n, z}^2}
  &\leq C v_{\lilsubseq{s}{l}}^{dk} \int_{(\reals^d)^{k-1}} \abs{\zeroed{f}(H(\bfzero_d, \bfy)) - \zeroed{f}(H(\bfzero_d, v_n v_{\lilsubseq{s}{l}}^{-1} \bfy))}^2 \,d\bfy\\
  &\leq C v_{\lilsubseq{s}{l}}^{dk} \sum_{i=1}^m \biggl( \suprdnorm{\nabla f}^2 \int_{(\reals^d)^{k-1}} \abs{h^{(i)}(\bfzero_d, \bfy) - h^{(i)}(\bfzero_d, v_n v_{\lilsubseq{s}{l}}^{-1} \bfy)}^2 \,d\bfy  \\
  &\quad + \supnorm{f}^2 \lebmeas \ab\big(\{\bfy \in (\reals^{d})^{k-1} : h^{(i)}(\bfzero_d, \bfy) = 0 \text{ and } h^{(i)}(\bfzero_d, v_n v_{\lilsubseq{s}{l}}^{-1}\bfy) \neq 0\} )  \\
  &\quad + \supnorm{f}^2 \lebmeas \ab\big(\{\bfy \in (\reals^{d})^{k-1} : h^{(i)}(\bfzero_d, \bfy) \neq 0 \text{ and } h^{(i)}(\bfzero_d, v_n v_{\lilsubseq{s}{l}}^{-1}\bfy) = 0 \}) \biggr)\\
  &\leq C v_{\lilsubseq{s}{l}}^{dk} \ab\big(\suprdnorm{\nabla f}^2 + \supnorm{f}^2) (1 - v_{\lilsubseq{s}{l+1}} / v_{\lilsubseq{s}{l}}). \label{eq:lil.upper.z4.y.expect}
\end{align}

Furthermore, we show that for sufficiently large $l$ and all $w \in \{0, 1, \dots, B-1\}^{d}$,
\begin{align}
  \expect[big]{U_{w}} \leq C (\suprdnorm{\nabla f} + \supnorm{f}) \ldpspeed[\lilsubseq{s}{l}]^{1/2}. \label{eq:lil.upper.z4.u.expect}
\end{align}
To apply \cref{prop:chaining.modify},
for any pair of indices satisfying $\lilsubseq{s}{l} \leq n_1, n_2 \leq \lilsubseq{s}{l+1}$,
we estimate the exponential moment of
\begin{align}
  Y_{n_1, z} - Y_{n_2, z}
  = \condexpect[Big]{\Delta_{\infty,z,f}^{(v_{\lilsubseq{s}{l}}, v_{n_1})} - \Delta_{\infty,z,f}^{(v_{\lilsubseq{s}{l}}, v_{n_2})}}{\calF_z}.
\end{align}
Condition \eqref{eq:p.moment.boundness} is verified analogously to \eqref{eq:lil.upper.z4.y.expect} as follows:
\begin{align}
  &\expect[Bigg]{\sum_{\calZ \subset \pppentire \cap Q_{z}^{(2Lv_{\lilsubseq{s}{l}})}} \abs{\zeroed{f}(H(v_{n_1}^{-1}\calZ)) - \zeroed{f}(H(v_{n_2}^{-1}\calZ))}^p} \\
  &\quad \leq C^p v_{\lilsubseq{s}{l}}^{dk} \int_{(\reals^d)^{k-1}} \abs{\zeroed{f}(H(\bfzero_d, \bfy)) - \zeroed{f}(H(\bfzero_d, v_{n_1 \lor n_2} v_{n_1 \land n_2}^{-1} \bfy))}^p \,d\bfy \\
  &\quad \leq C^p v_{\lilsubseq{s}{l}}^{dk} (\suprdnorm{\nabla f}^{p} + \supnorm{f}^p) (1 - v_{n_1\lor n_2} / v_{n_1 \land n_2}) \\
  &\quad \leq C^p v_{\lilsubseq{s}{l}}^{dk} (\suprdnorm{\nabla f} + \supnorm{f})^p (1 - v_{n_1\lor n_2} / v_{n_1 \land n_2}). \label{eq:lil.upper.z4.p.moment}
\end{align}
Thus, applying \cref{prop:exp.moment.common}, we obtain
\begin{align}
  &\log \expect[Bigg]{\exp\ab\bigg(\lambda\ab\bigg(\sum_{z \in J_w} Y_{n_{1}, z} - \sum_{z \in J_w} Y_{n_{2}, z}))} \\
  &\leq \card{J_w}  v_{\lilsubseq{s}{l}}^{dk} (1 - v_{n_1 \lor n_2} / v_{n_1 \land n_2}) \sum_{p=2}^{\infty} \frac{\ab\big(C (\suprdnorm{\nabla f} + \supnorm{f})\lambda)^p}{p!}\\
  &\leq \ldpspeed[\lilsubseq{s}{l}] (1 - v_{n_1 \lor n_2} / v_{n_1 \land n_2}) \frac{\ab\big(C(\suprdnorm{\nabla f} + \supnorm{f}))^{2} \lambda^2}{2\ab\big(1 - C(\suprdnorm{\nabla f} + \supnorm{f}) \lambda)},
\end{align}
where the last inequality holds for any $0 < \lambda < (C(\suprdnorm{\nabla f} + \supnorm{f}))^{-1}$.

We now define a metric $d_l$ on $\calT_l \coloneqq\{\lilsubseq{s}{l}, \dots, \lilsubseq{s}{l+1}\}$ by
\begin{align}
  d_l(n_1, n_2) \coloneqq (1 - v_{n_1 \lor n_2}/ v_{n_1 \land n_2})^{1/2}, \quad \lilsubseq{s}{l} \leq n_1,  n_2 \leq \lilsubseq{s}{l+1}.
\end{align}
Applying \cref{prop:chaining.modify}, we obtain
\begin{align}
  \expect[big]{U_{w}}
  &\leq 12 \ab\big(C^2 (\suprdnorm{\nabla f} + \supnorm{f})^2 \ldpspeed[\lilsubseq{s}{l}])^{1/2} \int_0^{\delta_l/2} \sqrt{\log N(u, \calT_l)} \,du  \\
  &\quad + 2 C(\suprdnorm{\nabla f} + \supnorm{f}) \sum_{j=1}^{D_l} \log N(2^{-j} \delta_l, \calT_l),
\end{align}
where $N(u, \calT_l)$ denotes the $u$-packing number of $\calT_l$,
$D_l$ is a positive integer satisfying $D_l \leq \roundup{\log_2 (\delta_l \delta_l'^{-1})}$,
and $\delta_l, \delta_l'$ are given by
\begin{align}
  \delta_l
  &\coloneqq \max_{\lilsubseq{s}{l} \leq n_1 < n_2 \leq \lilsubseq{s}{l+1}} d_l(n_1, n_2)
  = (1 - v_{\lilsubseq{s}{l+1}} / v_{\lilsubseq{s}{l}})^{1/2}, \\
  \delta_l'
  &\coloneqq \min_{\lilsubseq{s}{l} \leq n_1 < n_2 \leq \lilsubseq{s}{l+1}} d_l(n_1, n_2)
  = (1 - v_{\lilsubseq{s}{l+1}}/ v_{\lilsubseq{s}{l+1}-1})^{1/2}.
\end{align}
Using the trivial bound $N(u, \calT_l) \leq \card{\calT_l} = \lilsubseq{s}{l+1} - \lilsubseq{s}{l} + 1$,
together with the asymptotic relations as $l \to \infty$,
\begin{align}
  \int_0^{\delta_l/2} \sqrt{\log N(u, \calT_l)} \,du &= \lorder(1), \label{eq:packing.num.integral.const}\\
  D_l = \roundup{\log_2 (\delta_l \delta_l'^{-1})} &= \lorder(l), \label{eq:chaining.stop.num.l}
\end{align}
which will be shown in Step 4,
we arrive at
\begin{align}
  \expect[big]{U_{w}}
  &\leq C (\suprdnorm{\nabla f} + \supnorm{f}) \ab\big(\ldpspeed[\lilsubseq{s}{l}]^{1/2} 
  + l \log (\lilsubseq{s}{l+1} - \lilsubseq{s}{l} + 1) )\\
  &\leq C (\suprdnorm{\nabla f} + \supnorm{f}) \ldpspeed[\lilsubseq{s}{l}]^{1/2}.
\end{align}

\textit{Step} 4.
We show \eqref{eq:packing.num.integral.const} and \eqref{eq:chaining.stop.num.l}.
To this end, let $a = N(u, \calT_l)$.
Then, we can choose a sequence of indices $n_1 < \dots < n_a$ in $\calT_l$ such that $d_l(n_i, n_{i+1}) > u$ for each $i$.
Since this implies $v_{n_{i+1}}/v_{n_i} < 1 - u^2$, we have
\begin{align}
  \frac{v_{n_2}}{v_{n_1}} \cdots \frac{v_{n_a}}{v_{n_{a-1}}} \leq (1 - u^2)^{a-1},
\end{align}
which gives $v_{n_{a}}/v_{n_1} \leq (1- u^2)^{a-1}$.
Consequently, we obtain $\delta_l \geq (1 - v_{n_a}/v_{n_1})^{1/2} \geq (1 - (1 - u^2)^{a-1})^{1/2}$,
which yields the bound
\begin{align}
  N(u, \calT_l) = a
  \leq 1 + \frac{\log(1 - \delta_l^2)}{\log(1 - u^2)}.
\end{align}
Thus, we have
\begin{align}
  \int_{0}^{\delta_l/2} \sqrt{\log N(u, \calT_l)} \,du
  \leq \int_{0}^{\delta_l/2} \ab\bigg(\log \ab\Big(1 + \frac{\log(1 - \delta_l^2)}{\log(1 - u^2)}))^{1/2} \,du.
\end{align}
Furthermore, as $l \to \infty$,
we have $\delta_l = (1 - v_{\lilsubseq{s}{l+1}}/v_{\lilsubseq{s}{l}})^{1/2} \to (1 - s^{-(\alpha-1)/d(k -\alpha(k-1))})^{1/2} \leq (1 - 2^{-1})^{1/2}$.
Therefore, the integral converges to a finite constant,
establishing \eqref{eq:packing.num.integral.const}.
Furthermore, using the inequality $1 - (1 - x)^{\beta} \ge \beta x$ for $0 < x, \beta < 1$,
we have
\begin{align}
  {\delta_l'}^2
  % = 1 - v_{\lilsubseq{s}{l+1}}/ v_{\lilsubseq{s}{l+1}-1}
  = 1 - \ab\Big(1 - \frac{1}{\lilsubseq{s}{l+1}})^{(\alpha-1)/d}
  \ge \frac{\alpha-1}{d}\frac{1}{\lilsubseq{s}{l+1}},
\end{align}
which implies $D_l = \roundup{\log_2 (\delta_l \delta_l'^{-1})} = \lorder(\log\lilsubseq{s}{l+1}) = \lorder(l)$ as $l \to \infty$.

\textit{Step} 5.
By \eqref{eq:lil.upper.z4.y.expect}, we have
\begin{align}
  \max_{\lilsubseq{s}{l} \leq n \leq \lilsubseq{s}{l+1}} \sum_{z \in J_{w}} \expect[none]{Y_{n, z}^2}
  \leq C \card{J_{w}} v_{\lilsubseq{s}{l}}^{dk} (\suprdnorm{\nabla f}^2 + \supnorm{f}^2) (1 - v_{\lilsubseq{s}{l+1}} / v_{\lilsubseq{s}{l}}) \label{eq:lil.upper.z4.var.bound}
\end{align}
and using \eqref{eq:lil.upper.z4.u.expect},
we can bound
\begin{align}
  &\sum_{w \in \{0, 1, \dots, B-1\}^{d}} A_{w} \\
  &\quad \leq 2 (C \supnorm{f})^{-1} B^d C (\suprdnorm{\nabla f} + \supnorm{f})  \ldpspeed[\lilsubseq{s}{l}]^{1/2}\\
   &\qquad + (C \supnorm{f})^{-2} C \card{I_{\lilsubseq{s}{l+1}}} (\suprdnorm{\nabla f}^2 + \supnorm{f}^2) v_{\lilsubseq{s}{l}}^{dk} (1 - v_{\lilsubseq{s}{l+1}} / v_{\lilsubseq{s}{l}}) \\
  &\quad \leq C(\supnorm{f}^{-1}\suprdnorm{\nabla f} + 1) \ldpspeed[\lilsubseq{s}{l}]^{1/2} + C (\supnorm{f}^{-2}\suprdnorm{\nabla f}^2 + 1) (1 - v_{\lilsubseq{s}{l+1}}/v_{\lilsubseq{s}{l}}) \ldpspeed[\lilsubseq{s}{l}].
\end{align}
Consequently, for sufficiently large $l$,
since $\ldpspeed[\lilsubseq{s}{l}]^{1/2} \leq C (1 - v_{\lilsubseq{s}{l+1}}/v_{\lilsubseq{s}{l}}) \ldpspeed[\lilsubseq{s}{l}]$,
the right-hand side of \eqref{eq:lil.upper.v.exp.moment} is bounded by
\begin{align}
  &\sum_{w \in \{0, 1, \dots, B-1\}^{d}} \expect[big]{U_{w}}\lambda
   + \sum_{w \in \{0, 1, \dots, B-1\}^{d}} A_{w} \ab\big(e^{B^d C\supnorm{f}\lambda} - B^d C\supnorm{f} \lambda - 1) \\
  &\quad \leq C (\suprdnorm{\nabla f} + \supnorm{f}) \ldpspeed[\lilsubseq{s}{l}]^{1/2} \lambda \\
   &\qquad + C \ab\big(\supnorm{f}^{-1}\suprdnorm{\nabla f} + \supnorm{f}^{-2}\suprdnorm{\nabla f}^2 + 2)  \\
   &\quad \qquad \times (1 - v_{\lilsubseq{s}{l+1}}/v_{\lilsubseq{s}{l}}) \ldpspeed[\lilsubseq{s}{l}]\ab\big(e^{B^d C\supnorm{f}\lambda} - B^d C\supnorm{f} \lambda - 1) \\
  &\quad \leq C (\suprdnorm{\nabla f} + \supnorm{f}) \ldpspeed[\lilsubseq{s}{l}]^{1/2} \lambda
   + C C_f' (1 - v_{\lilsubseq{s}{l+1}}/v_{\lilsubseq{s}{l}}) \ldpspeed[\lilsubseq{s}{l}] \frac{C^2 \lambda^2}{ 1 - C\lambda},
\end{align}
where the last inequality holds for $0 < \lambda < C^{-1}$.
Combining this with \eqref{eq:lil.upper.h.diff.bound} and choosing the optimal $\lambda$,
we obtain
\begin{align}
  &\probab[Big]{\max_{\lilsubseq{s}{l} \leq n \leq \lilsubseq{s}{l+1}} Z_{n,f}^{(4)} > R} \\
  &\quad \leq \exp\ab\bigg(- C C_f' (1 - v_{\lilsubseq{s}{l+1}}/v_{\lilsubseq{s}{l}}) \ldpspeed[\lilsubseq{s}{l}] \tailfunc\ab\bigg(\frac{R - C (\suprdnorm{\nabla f} + \supnorm{f}) \ldpspeed[\lilsubseq{s}{l}]^{1/2}}{C^2 C_f'  (1 - v_{\lilsubseq{s}{l+1}}/v_{\lilsubseq{s}{l}}) \ldpspeed[\lilsubseq{s}{l}] }))
\end{align}
for all $R > C (\suprdnorm{\nabla f} + \supnorm{f}) \ldpspeed[\lilsubseq{s}{l}]^{1/2}$.
Consequently, for $R > 2C (\suprdnorm{\nabla f} + \supnorm{f}) \ldpspeed[\lilsubseq{s}{l}]^{1/2}$,
we have $R - C (\suprdnorm{\nabla f} + \supnorm{f}) \ldpspeed[\lilsubseq{s}{l}]^{1/2} > R/2$,
which yields \eqref{eq:lil.upper.z4.tail.max}.
\end{proof}

\begin{remark}
\label{rem:lil.upper.z4.bound.f.boundness}
The boundedness of $f$ is required only for \eqref{eq:lil.upper.u.exp.moment},
and not for \eqref{eq:lil.upper.z4.y.expect} and \eqref{eq:lil.upper.z4.u.expect}.
Thus, \eqref{eq:lil.upper.z4.var.bound} and \eqref{eq:lil.upper.z4.u.expect} hold for $f = \iprod{\bfa}{\cdot}$ with $\bfa \in \reals^m$,
owing to $\norm{\nabla f}_{\reals^m} = \norm{\bfa}_{\reals^m} < \infty$.
\end{remark}

\begin{proof}[Proof of \cref{lemma:lil.upper.z4.ei}]
Unlike in \cref{lemma:lil.upper.z4}, the function $h^{(i)}$ is not necessarily bounded,
so Bousquet's concentration inequality is not applicable.
Therefore, we apply \cref{prop:adamczak.ineq} instead.

Let $Y_{n, z}$ be as defined in \eqref{eq:lil.upper.z4.y}.
Set $B \coloneqq \roundup{12 \sqrt{d}L/t}$.
For $w \in \{0, 1, \dots, B-1\}^{d}$ and $\lilsubseq{s}{l} \leq j \leq \lilsubseq{s}{l+1}$,
we define $I_{j, w} \coloneqq I_{j} \cap (B \integers^d + w)$ and
\begin{align}
  V_{j, w}
  \coloneqq \max_{\lilsubseq{s}{l} \leq n \leq \lilsubseq{s}{l+1}} \sum_{z \in I_{j, w}} Y_{n, z}.
\end{align}
By the union bound, we have
\begin{align}
  \probab[Big]{\max_{\lilsubseq{s}{l} \leq n \leq \lilsubseq{s}{l+1}} Z_{n,f}^{(4)} > R}
  &\leq \sum_{w \in \{0, 1, \dots, B-1\}^{d}} \probab[bigg]{\max_{\lilsubseq{s}{l} \leq n \leq \lilsubseq{s}{l+1}} \sum_{z \in I_{n, w}} Y_{n, z} > \frac{R}{B^d}} \\
  &\leq \sum_{w \in \{0, 1, \dots, B-1\}^{d}} \probab[Big]{\max_{\lilsubseq{s}{l} \leq j \leq \lilsubseq{s}{l+1}} V_{j, w} > \frac{R}{B^d}}. \label{eq:lil.upper.z4.ei.bound.V.w}
\end{align}

To apply \cref{prop:adamczak.ineq} to $V_{j, w}$,
we now estimate the variance of $Y_{n, z}$,
the exponential moments of $Y_{n, z}$,
and the $\psi_1$-Orlicz norm (defined by \eqref{eq:orlicz.norm.def}) of $\max_{z, n} \abs{Y_{n, z}}$.
In the following, $C_1 > 0$ denotes a generic constant independent of $s, l, i, w$,
which may change from line to line.
As noted in \cref{rem:lil.upper.z4.bound.f.boundness},
\eqref{eq:lil.upper.z4.var.bound} holds (recall that $\norm{\nabla f}_{\reals^m} = \norm{\bfe_i}_{\reals^m} = 1$).
Hence, we obtain
\begin{align}
  \max_{\lilsubseq{s}{l} \leq n \leq \lilsubseq{s}{l+1}} \sum_{z \in I_{\lilsubseq{s}{l+1}, w}} \expect[none]{Y_{n, z}^2}
  &\leq C_1 \card{I_{\lilsubseq{s}{l+1}, w}} v_{\lilsubseq{s}{l}}^{dk} (1 - v_{\lilsubseq{s}{l+1}} / v_{\lilsubseq{s}{l}})\\
  &\leq C_1 (1 - v_{\lilsubseq{s}{l+1}} / v_{\lilsubseq{s}{l}}) \ldpspeed[\lilsubseq{s}{l}].
\end{align}
Similarly, since \eqref{eq:lil.upper.z4.u.expect} also holds,
we have $\expect[none]{V_{\lilsubseq{s}{l+1}, w}} \leq C_1 \ldpspeed[\lilsubseq{s}{l}]^{1/2}$ for all $w \in \{0, \dots, B-1\}^d$.
Furthermore, using \eqref{eq:lil.upper.z4.p.moment} and a calculation similar to that leading to \eqref{eq:exp.moment.common.p.moment},
we deduce that for all $\lambda \in \reals$,
\begin{align}
  \expect[none]{\exp(\lambda Y_{n, z})}
  \leq 1 + v_{\lilsubseq{s}{l}}^{dk} (1 - v_{\lilsubseq{s}{l+1}} / v_{\lilsubseq{s}{l}}) (e^{C_1\abs{\lambda}} - C_1\abs{\lambda} - 1).
\end{align}
Thus, since $v_{\lilsubseq{s}{l}}^{dk} < 1/4$ for sufficiently large $l$,
\cref{prop:orliczc.norm} (i) yields $\orlicznorm{Y_{n, z}} \leq C_1$
for all $\lilsubseq{s}{l} \leq n \leq \lilsubseq{s}{l+1} $ and $z \in I_{\lilsubseq{s}{l+1}, w} $,
where $\orlicznorm{\cdot}$ denotes the $\psi_1$-Orlicz norm defined by \eqref{eq:orlicz.norm.def}.
Then, by \cref{prop:orliczc.norm} (iii),
we find that
\begin{align}
  \orlicznorm{\max_{\lilsubseq{s}{l} \leq n \leq \lilsubseq{s}{l+1}} \abs{Y_{n, z}}}
  &\leq \orlicznorm{\max_{z \in I_{\lilsubseq{s}{l+1}, w}} \max_{\lilsubseq{s}{l} \leq n \leq \lilsubseq{s}{l+1}} \abs{Y_{n, z}}} \\
  &\leq C_1 \ab\Big(1 + \log_2\ab\big((\lilsubseq{s}{l+1} - \lilsubseq{s}{l} + 1)\card{I_{\lilsubseq{s}{l+1}, w}}) )
  \leq C_1 l < \infty.
\end{align}

Consequently, applying \cref{prop:adamczak.ineq} with $\eta = \delta = 1/2$,
for $R > 2 B^d C_1 \ldpspeed[\lilsubseq{s}{l}]^{1/2} $,
we have $R/B^d - 3\expect[none]{V_{\lilsubseq{s}{l+1}, w}}/2 \geq R/4B^d $
for all $w \in \{0, \dots, B-1\}^d$,
and \eqref{eq:lil.upper.z4.ei.bound.V.w} can be bounded by
\begin{align}
  % &B^d\exp\ab\Big(- \frac{(\frac{R}{B^d} - \frac{3}{2}C_2 \ldpspeed[\lilsubseq{s}{l}]^{1/2} )^2}{3(C_1 \ldpspeed[\lilsubseq{s}{l}] (1 - v_{\lilsubseq{s}{l+1}} / v_{\lilsubseq{s}{l}}))^2})
  % + 3 B^d \exp\ab\bigg(- \frac{\frac{R}{B^d} - \frac{3}{2} C_2 \ldpspeed[\lilsubseq{s}{l}]^{1/2} }{C(\frac{1}{2}, \frac{1}{2}) C_3 l}) \\
  & B^d\exp\ab\Big(- \frac{(R/4B^d)^2}{3C_1 (1 - v_{\lilsubseq{s}{l+1}} / v_{\lilsubseq{s}{l}}) \ldpspeed[\lilsubseq{s}{l}]})
  + 3 B^d \exp\ab\bigg(- \frac{R/4B^d}{C(\frac{1}{2}, \frac{1}{2}) C_1 l}) \\
  &\quad = B^d\exp\ab\Big(- \frac{R^2}{48 B^{2d} C_1 (1 - v_{\lilsubseq{s}{l+1}} / v_{\lilsubseq{s}{l}}) \ldpspeed[\lilsubseq{s}{l}]})
  + 3 B^d \exp\ab\bigg(- \frac{R}{4B^d C(\frac{1}{2}, \frac{1}{2}) C_1 l}),
\end{align}
where $C(1/2, 1/2)$ is the constant provided in \cref{prop:adamczak.ineq}.
Finally, by taking $C$ as
\begin{align}
  C \coloneqq \max\Bab[Big]{2 B^d C_1, 3B^d, 48B^{2d}C_1, 4 B^d C\ab\Big(\frac{1}{2}, \frac{1}{2})C_1},
\end{align}
we establish \eqref{eq:lil.upper.z4.ei.tail.max}.
\end{proof}

\subsection{Proof of \eqref{eq:lil.lower.meas} and \eqref{eq:lil.lower.vec}}

As the proofs for \eqref{eq:lil.lower.meas} and \eqref{eq:lil.lower.vec} are essentially identical,
we present the proof of \eqref{eq:lil.lower.meas} only.

\begin{proof}[Proof of \eqref{eq:lil.lower.meas}]
\textit{Step} 1.
Since $\signedmeas(E)$ is separable with respect to the weak topology,
it is also separable with respect to the metric $\distancemeas$.
Thus, to establish \eqref{eq:lil.lower.meas},
it is sufficient to show that for all $\theta \in \accptmeas$, $g_1, \dots, g_M \in \bddmble(E)$ and $\epsilon > 0$,
there exists $s > 1$ such that
% \begin{align}
%   \liminf_{n \to \infty} \max_{1\leq i \leq M} \abs{\pairing[big]{g_i}{\mdpfactor^{-1} \centered{\xi_{n}} - \pairing{g_i}{\theta} }} = 0 \quad \text{a.s.} \label{eq:lil.lower.subseq}
% \end{align}
\begin{align}
  \liminf_{l \to \infty} \max_{1\leq i \leq M} \abs{\pairing[big]{g_i}{\mdpfactor[\lilsubseq{s}{l}]^{-1} \centered{\xi_{\lilsubseq{s}{l}}}} - \pairing{g_i}{\theta} } \leq \epsilon \quad \text{a.s.} \label{eq:lil.lower.subseq}
\end{align}
Indeed,
if we denote $\distancemeas(\cdot, \cdot)$ by $\distancemeas(\cdot, \cdot; \{f_q\}_{q \geq 1})$
to make the family of test functions $\{f_q\}_{q \geq 1}$ explicit,
we observe that for any positive integer $M \geq 1$,
\begin{align}
  \distancemeas(\mdpfactor^{-1} \centered{\xi_{n}}, \theta; \Bab{f_q}_{q \geq 1})
  &= \sum_{q=1}^M \frac{1}{2^q \supnorm{f_q}} \abs{\pairing{f_q}{\mdpfactor^{-1} \centered{\xi_{n}}} - \pairing{f_q}{\theta} }
   + \frac{1}{2^{M-1}} \distancemeas(\mdpfactor^{-1} \centered{\xi_{n}}, \theta; \Bab{f_{q+M}}_{q \geq 1})\\
  &\leq \sum_{q=1}^M \frac{1}{2^q \supnorm{f_q}} \abs{\pairing{f_q}{\mdpfactor^{-1} \centered{\xi_{n}}} - \pairing{f_q}{\theta} }
   + \frac{1}{2^{M-1}} \distancemeas(\mdpfactor^{-1} \centered{\xi_{n}}, \accptmeas; \Bab{f_{q+M}}_{q \geq 1}) \\
   &\quad + \frac{1}{2^{M-1}} \sup_{\nu \in \accptmeas} \distancemeas(\nu, \theta; \Bab{f_{q+M}}_{q \geq 1}).
\end{align}
By \cref{lem:lil.subseqence.cptness},
it follows that 
$\distancemeas(\mdpfactor[\lilsubseq{s}{l}]^{-1} \centered{\xi_{\lilsubseq{s}{l}}}, \accptmeas; \Bab{f_{q+M}}_{q \geq 1}) \to 0$ as $l \to \infty$ for all $s > 1$.
Furthermore, since
\begin{align}
  \distancemeas(\nu, \theta; \Bab{f_{q+M}}_{q \geq 1})
  \leq \sum_{q = 1}^{\infty} \frac{1}{2^q} \ab\big(\tvnorm{\nu} + \tvnorm{\theta})
  = \tvnorm{\nu} + \tvnorm{\theta},
\end{align}
where $\tvnorm{\cdot} $ denotes the total variation norm,
if \eqref{eq:lil.lower.subseq} holds,
then applying \eqref{eq:lil.lower.subseq} to $g_1 = f_1/\supnorm{f_1}, \dots, g_M = f_M/\supnorm{f_M}$ yields
\begin{align}
  \liminf_{l \to \infty} \distancemeas(\mdpfactor[\lilsubseq{s}{l}]^{-1} \centered{\xi_{\lilsubseq{s}{l}}}, \theta; \Bab{f_q}_{q \geq 1})
  \leq \epsilon + \frac{1}{2^{M-1}} \ab\Big(\sup_{\nu \in \accptmeas} \tvnorm{\nu} + \tvnorm{\theta}).
\end{align}
Since this inequality holds for all $M \geq 1$, we obtain \eqref{eq:lil.lower.meas}.

For $s > 1 $, let
\begin{align}
  W_{\lilsubseq{s}{l}}' \coloneqq \cube{0}{\lilsubseq{s}{l-1}^{1/d} + 2 v_{\lilsubseq{s}{l}} t}
\end{align}
and
\begin{align}
  \xi_{\lilsubseq{s}{l}}' \coloneqq \sum_{\calZ \subset \pppentire_{\lilsubseq{s}{l}} \setminus W_{\lilsubseq{s}{l}}'} \isolated{\calZ}{\pppentire_{\lilsubseq{s}{l}}}{v_{\lilsubseq{s}{l}} t} \indicator[none]{H(v_{\lilsubseq{s}{l}}^{-1} \calZ) \neq \bfzero_{m}} \diracdelta{H(v_{\lilsubseq{s}{l}}^{-1} \calZ)}. \label{eq:lil.lower.xiprime.def}
\end{align}
To prove \eqref{eq:lil.lower.subseq},
it is enough to show that for all $\epsilon > 0$,
there exists $s > 1$ such that
\begin{align}
  \liminf_{l \to \infty} \max_{1\leq i \leq M} \abs{\pairing[big]{g_i}{\mdpfactor[\lilsubseq{s}{l}]^{-1} \centered{\xi_{\lilsubseq{s}{l}}'}} - \pairing{g_i}{\theta} } \leq \frac{\epsilon}{2} \quad \text{a.s.} \label{eq:lil.lower.xidash.liminf}
\end{align}
and
\begin{align}
  \sum_{l = 1}^{\infty} \probab[Big]{\max_{1\leq i \leq M} \abs{\pairing[big]{g_i}{\mdpfactor[\lilsubseq{s}{l}]^{-1} \centered{\xi_{\lilsubseq{s}{l}}}} - \pairing[big]{g_i}{\mdpfactor[\lilsubseq{s}{l}]^{-1} \centered{\xi_{\lilsubseq{s}{l}}'}}} > \frac{\epsilon}{4}} < \infty. \label{eq:lil.lower.xierror.limsup}
\end{align}

\textit{Step} 2.
We first prove \eqref{eq:lil.lower.xidash.liminf}.
Define the events
\begin{align}
  \calE_l \coloneqq \Bab{\max_{1\leq i \leq M} \abs{\pairing[big]{g_i}{\mdpfactor[\lilsubseq{s}{l}]^{-1}\centered{\xi_{\lilsubseq{s}{l}}}} - \pairing{g_i}{\theta} } \geq \frac{\epsilon}{4}}, \quad
  \calE_l' \coloneqq \Bab{\max_{1\leq i \leq M} \abs{\pairing[big]{g_i}{\mdpfactor[\lilsubseq{s}{l}]^{-1}\centered{\xi_{\lilsubseq{s}{l}}'}} - \pairing{g_i}{\theta} } \geq \frac{\epsilon}{2}}.
\end{align}
By the definition of $W_{\lilsubseq{s}{l}}' $,
the indicator $\isolated{\calZ}{\pppentire_{\lilsubseq{s}{l}}}{v_{\lilsubseq{s}{l}} t}$ in \eqref{eq:lil.lower.xiprime.def}
is equal to $\isolated{\calZ}{\pppentire_{\lilsubseq{s}{l}} \setminus \pppentire_{\lilsubseq{s}{l-1}}}{v_{\lilsubseq{s}{l}} t}$,
which implies that $\{\xi_{\lilsubseq{s}{l}}'\}_{l \geq 1} $ are independent.
Therefore, by the second Borel--Cantelli lemma, it suffices to show that
\begin{align}
  \sum_{l=1}^{\infty} \probab[none]{\calE_l'^{c}} = \infty. \label{eq:lil.lower.event}
\end{align}
Note that
\begin{align}
  \probab[none]{\calE_l'^{c}}
  &\geq \probab[none]{\calE_l^{c}} - \probab[bigg]{\max_{1\leq i \leq M} \abs{\pairing[big]{g_i}{\mdpfactor[\lilsubseq{s}{l}]^{-1} \centered{\xi_{\lilsubseq{s}{l}}}} - \pairing[big]{g_i}{\mdpfactor[\lilsubseq{s}{l}]^{-1} \centered{\xi_{\lilsubseq{s}{l}}'}} } >  \frac{\epsilon}{4}}.
\end{align}
Since $\theta \in \accptmeas$,
we have
\begin{align}
  \inf \Bab{\Lambda(\nu) : \nu \in \signedmeas(E), \, \max_{1\leq i \leq M} \abs{\pairing{g_i}{\nu} - \pairing{g_i}{\theta} } < \frac{\epsilon}{4} } < 1.
\end{align}
Thus, by \cref{thm:mdp.measure}, there exists a constant $\beta < 1$ such that
\begin{align}
  \sum_{l=1}^{\infty} \probab[none]{\calE_l^{c}}
  \geq \sum_{l=1}^{\infty} \exp(- \beta \log \log \ldpspeed[\lilsubseq{s}{l}])
  = \sum_{l=1}^{\infty} (l \log s)^{- \beta}
  = \infty.
\end{align}
Combining this with \eqref{eq:lil.lower.xierror.limsup},
we obtain \eqref{eq:lil.lower.event}.

\textit{Step} 3.
Next, we prove \eqref{eq:lil.lower.xierror.limsup}.
To show \eqref{eq:lil.lower.xierror.limsup},
it suffices to show that for all $g \in \bddmble(E)$ and all $\epsilon > 0$,
\begin{align}
  \sum_{l = 1}^{\infty} \probab[Big]{\abs[big]{\pairing[big]{g}{ \centered{\xi_{\lilsubseq{s}{l}}}} - \pairing[big]{g}{ \centered{\xi_{\lilsubseq{s}{l}}'}}} > \epsilon \mdpfactor[\lilsubseq{s}{l}]} < \infty. \label{eq:lil.lower.xierror.limsup.f}
\end{align}
We can decompose the difference as $\pairing{g}{\xi_{\lilsubseq{s}{l}}} - \pairing{g}{\xi_{\lilsubseq{s}{l}}'} = Y_{l}^{(1)} + Y_{l}^{(2)} + Y_{l}^{(3)} $,
where
\begin{align}
  Y_{l}^{(1)} &\coloneqq \sum_{\calZ \subset \pppentire_{
  \lilsubseq{s}{l-1}}} \isolated{\calZ}{\pppentire_{\lilsubseq{s}{l-1}}}{v_{\lilsubseq{s}{l}}t} \zeroed{g}(H(v_{\lilsubseq{s}{l}}^{-1} \calZ)), \\
  Y_{l}^{(2)} &\coloneqq \sum_{\calZ \subset \pppentire_{
  \lilsubseq{s}{l-1}}} \ab\big(\isolated{\calZ}{\pppentire_{\lilsubseq{s}{l}}}{v_{\lilsubseq{s}{l}}t} - \isolated{\calZ}{\pppentire_{\lilsubseq{s}{l-1}}}{v_{\lilsubseq{s}{l}}t}) \zeroed{g}(H(v_{\lilsubseq{s}{l}}^{-1} \calZ)), \\
  Y_{l}^{(3)} &\coloneqq \sum_{\calZ \subset \pppentire_{
  \lilsubseq{s}{l}}} \isolated{\calZ}{\pppentire_{\lilsubseq{s}{l}}}{v_{\lilsubseq{s}{l}}t} \indicator[none]{\calZ \cap W'_{\lilsubseq{s}{l}} \neq \emptyset, \, \calZ \not\subset W_{\lilsubseq{s}{l-1}}} \zeroed{g}(H(v_{\lilsubseq{s}{l}}^{-1} \calZ)).
\end{align}
Then, to show \eqref{eq:lil.lower.xierror.limsup.f}, it is enough to show that for $j=1,2,3$,
\begin{align}
  \sum_{l=1}^\infty \probab[Big]{\abs[Big]{\centered{Y_{l}^{(j)}}} \geq \epsilon \mdpfactor[\lilsubseq{s}{l}]} < \infty. \label{eq:lil.lower.xierror.sum}
\end{align}

The bound for $Y_{l}^{(2)}$ follows by an argument similar to that used for $Z_n^{(3)} $, defined in \eqref{eq:lil.upper.z3.def}.
Similarly, the bound for $Y_{l}^{(3)}$ follows by an argument analogous to that used for $Z_n^{(1)} $, defined in \eqref{eq:lil.upper.z1.def}.
Therefore, we only need to show \eqref{eq:lil.lower.xierror.sum} for $Y_{l}^{(1)}$.
We choose $s > 1$ such that $s^{1/(k - \alpha(k-1))}$ is an integer,
which ensures that $\lilsubseq{s}{l} = s^{l/(k - \alpha(k-1))} $.
Since $v_{\lilsubseq{s}{l}} = s^{-(\alpha-1)/d(k - \alpha(k-1))} v_{\lilsubseq{s}{l-1}} $,
we let $t' \coloneqq s^{-(\alpha-1)/d(k - \alpha(k-1))} t $ and
$H'(\calY) \coloneqq H(s^{(\alpha-1)/d(k - \alpha(k-1))} \calY) $.
Let $\xi_n^{(t', H')}$ denote the point process obtained by replacing $t$ and $H$ in $\xi_n \eqqcolon \xi_n^{(t,H)}$ with $t'$ and $H'$, respectively.
Then, we can rewrite $Y_{l}^{(1)}$ as
\begin{align}
  Y_{l}^{(1)} = \sum_{\calZ \subset \pppentire_{
  \lilsubseq{s}{l-1}}} \isolated{\calZ}{\pppentire_{\lilsubseq{s}{l-1}}}{v_{\lilsubseq{s}{l-1}} t'} \zeroed{g}(H'(v_{\lilsubseq{s}{l-1}}^{-1} \calZ))
  = \pairing[big]{g}{\xi_{\lilsubseq{s}{l-1}}^{(t', H')}}.
\end{align}
Since $H'$ also satisfies conditions (H1)--(H5),
we can apply \cref{thm:mdp.measure} to the sequence $\{\mdpfactor^{-1} \centered{\xi_n^{(t', H')}} \}_{n \geq 1} $.
This yields
\begin{align}
  &\limsup_{l \to \infty} \frac{1}{\log\log \ldpspeed[\lilsubseq{s}{l}]} \log \probab[Big]{\abs[Big]{\centered{Y_{l}^{(1)}}} \geq \epsilon \mdpfactor[\lilsubseq{s}{l}]} \\
  &\quad = \limsup_{l \to \infty} \frac{1}{\log\log \ldpspeed[\lilsubseq{s}{l}]} \log \probab[Big]{\frac{1}{\mdpfactor[\lilsubseq{s}{l-1}]} \abs{\pairing[Big]{g}{\centered{\xi_{\lilsubseq{s}{l-1}}^{(t', H')}}}} \geq \epsilon \frac{\mdpfactor[\lilsubseq{s}{l}]}{\mdpfactor[\lilsubseq{s}{l-1}]} } \\
  &\quad \leq \limsup_{l \to \infty} \frac{1}{\log\log \rho_{\lilsubseq{s}{l-1}}} \log \probab[Big]{\frac{1}{\mdpfactor[\lilsubseq{s}{l-1}]} \abs{\pairing[Big]{g}{\centered{\xi_{\lilsubseq{s}{l-1}}^{(t', H')}}}} \geq \epsilon \frac{\sqrt{s}}{2} } \\
  &\quad \leq - \inf \Bab{\Lambda^{(H')}(\nu) : \nu \in \signedmeas(E), \, \abs{\pairing{g}{\nu}} \geq \epsilon \frac{\sqrt{s}}{2}},
\end{align}
where $\Lambda^{(H')} $ is the rate function for $\{\mdpfactor^{-1} \centered{\xi_n^{(t', H')}} \}_{n \geq 1} $ obtained by replacing $H $ in $\Lambda \eqqcolon \Lambda^{(H)}$ with $H'$.
Similarly, let $\tau^{(H')}$ be the measure obtained by replacing $H$ in $\tau \eqqcolon \tau^{(H)}$ with $H'$.
Then, we have $\tau^{(H')} = s^{-\alpha'} \tau^{(H)} $,
where $\alpha' \coloneqq (k-1)(\alpha - 1) / (k - \alpha(k-1))$,
and for any $\nu \in \signedmeas(E)$, the absolute continuity $\nu \abscont \tau^{(H)}$ is equivalent to $\nu \abscont \tau^{(H')}$.
Consequently, if $\nu \abscont \tau^{(H)}$,
\begin{align}
  \Lambda^{(H')}(\nu)
  = \frac{1}{2} \int_{E} \ab\Big(\frac{d \nu}{d \tau^{(H')}})^2 \,d\tau^{(H')}
  = \frac{1}{2} \int_{E} \ab\bigg(\frac{d \nu}{s^{-\alpha'} d \tau^{(H)}})^2 s^{-\alpha'} \,d\tau^{(H)}
  = s^{\alpha'} \Lambda^{(H)}(\nu).
\end{align}
Thus, we obtain
\begin{align}
  \inf \Bab{\Lambda^{(H')}(\nu) : \nu \in \signedmeas(E), \, \abs{\pairing{g}{\nu}} \geq \epsilon \frac{\sqrt{s}}{2}}
  = s^{\alpha'} \inf \Bab{\Lambda^{(H)}(\nu) : \nu \in \signedmeas(E), \, \abs{\pairing{g}{\nu}} \geq \epsilon \frac{\sqrt{s}}{2}}. \label{eq:lil.lower.rate.func.convert}
\end{align}
We now choose $s$ large enough so that
$\epsilon \sqrt{s}/2 > \sup \{\abs{\pairing{g}{\nu}} : \nu \in \signedmeas(E), \, \Lambda^{(H)}(\nu) \leq 1\}$.
Then, we can find a constant $\beta$ such that the infimum on the left-hand side of \eqref{eq:lil.lower.rate.func.convert} is strictly bounded from below by $s^{\alpha'} > \beta > 1$.
It follows that for all sufficiently large $l$,
\begin{align}
  \probab[Big]{\abs[Big]{\centered{Y_{l}^{(1)}}} \geq \epsilon \mdpfactor[\lilsubseq{s}{l}]} \leq \exp(- \beta \log \log s^{l}) = (l \log s)^{-\beta},
\end{align}
which yields \eqref{eq:lil.lower.xierror.sum}.
\end{proof}

\subsection{Proof of \eqref{eq:lil.equivalent.xi.kappa}, \eqref{eq:lil.equivalent.kappappp.kappabpp}, \eqref{eq:lil.equivalent.phixi.t}, and \eqref{eq:lil.equivalent.tppp.tbpp}}

In this subsection, we prove \eqref{eq:lil.equivalent.xi.kappa}, \eqref{eq:lil.equivalent.kappappp.kappabpp}, \eqref{eq:lil.equivalent.phixi.t}, and \eqref{eq:lil.equivalent.tppp.tbpp} by combining the estimates established above.

\begin{proof}[Proof of \eqref{eq:lil.equivalent.xi.kappa} and \eqref{eq:lil.equivalent.phixi.t}]
To show \eqref{eq:lil.equivalent.xi.kappa},
it is enough to show that for all $\delta > 0$,
\begin{align}
  \sum_{n=1}^{\infty} \probab[big]{\distancemeas\ab\big(\centered{\xi_n(\ppp_n)}, \centered{\kappa_n(\ppp_n)}) > \delta \mdpfactor} < \infty. \label{eq:xi.kappa.dist.summable}
\end{align}
By \cref{prop:exp.eq.xi.kappa},
we have for all $q \geq 1$ and $R > 0$,
\begin{align}
  \probab[big]{\abs[big]{\pairing[big]{f_q}{\centered{\xi_n(\ppp_n)}} - \pairing[big]{f_q}{\centered{\kappa_n(\ppp_n)}}} > R}
  \leq 2\exp\ab\Big(- e^2 \tau(E) C \nrnd \ldpspeed \tailfunc\ab\Big(\frac{R}{e^2 \tau(E) C^2 \nrnd \ldpspeed})).
\end{align}
Note that the constant $C_{f_q}$ given in \eqref{eq:xi.kappa.diff.tail} is uniformly bounded by $e^2\tau(E)$ for all $q \geq 1$,
since $\supnorm{f_q} = 1$.
Setting the variables in \cref{prop:lil.upper.tail.sum.finite} as
$Y_{n, q} = \abs[big]{\pairing[big]{f_q}{\centered{\xi_n(\ppp_n)}} - \pairing[big]{f_q}{\centered{\kappa_n(\ppp_n)}}}$
and
\begin{align}
  (A_n, B_n, C_n, \beta, a') = (\delta \mdpfactor, e^2 \tau(E) C \nrnd \ldpspeed, 2, 0, C),
\end{align}
we see that the assumption \eqref{eq:lil.upper.tail.sum.ass} is satisfied.
Thus, by \cref{prop:lil.upper.tail.sum.finite},
we obtain \eqref{eq:xi.kappa.dist.summable}.

We next prove \eqref{eq:lil.equivalent.phixi.t}.
Setting $f_i \coloneqq \iprod{\bfe_i}{\cdot}$ for $i=1,\dots,m$,
we have
\begin{align}
  \norm{\mdpfactor^{-1} \Phi(\centered{\xi_n(\ppp_n)}) - \mdpfactor^{-1} \centered{T_n(\ppp_n)}}_{\reals^m}
  \leq \sum_{i=1}^m \mdpfactor^{-1}\abs[Big]{\pairing[big]{f_i}{\centered{\xi_n(\ppp_n)}} - \pairing[big]{f_i}{\centered{\kappa_n(\ppp_n)}}}.
\end{align}
As in \eqref{eq:xi.kappa.dist.summable},
we have for all $\delta > 0$ and $i=1,\dots, m$,
\begin{align}
  \sum_{n=1}^{\infty} \probab[big]{\abs[big]{\pairing[big]{f_i}{\centered{\xi_n(\ppp_n)}} - \pairing[big]{f_i}{\centered{\kappa_n(\ppp_n)}}} > \delta\mdpfactor} < \infty,
\end{align}
which yields \eqref{eq:lil.equivalent.phixi.t}.
\end{proof}

\begin{proof}[Proof of \eqref{eq:lil.equivalent.kappappp.kappabpp} and \eqref{eq:lil.equivalent.tppp.tbpp}]
As the proof of \eqref{eq:lil.equivalent.tppp.tbpp} is completely analogous to that of \eqref{eq:lil.equivalent.kappappp.kappabpp},
we present only the proof of \eqref{eq:lil.equivalent.kappappp.kappabpp}.
The constant $C_{f_q}$ in \eqref{eq:kappa.f.projection.sum} is uniformly bounded by $m \tau(E)$ for $q \geq 1$,
because $\supnorm{f_q} = 1$.
Thus, by \eqref{eq:kappa.ppp.bpp.expect.bound},
we have
\begin{align}
  \lim_{n \to \infty} \sup_{q \geq 1} \abs[big]{\expect[none]{\pairing{f_q}{\kappa_n(\ppp_n)}} - \expect[none]{\pairing{f_q}{\kappa_n(\bpp_n)}}} = 0.
\end{align}
Therefore, to show \eqref{eq:lil.equivalent.kappappp.kappabpp},
it is enough to show that for all $\delta > 0$,
\begin{align}
  \sum_{n=1}^{\infty} \probab[bigg]{\sum_{q=1}^{\infty} \frac{1}{2^q} \abs[big]{\pairing{f_q}{\kappa_n(\ppp_n)} - \pairing{f_q}{\kappa_n(\bpp_n)}} > \delta \mdpfactor} < \infty.
\end{align}

Set $\epsilon_n = (\nrnd)^{(k-1)/4} \mdpfactor / \ldpspeed$ as in \eqref{eq:pp.bpp.epsilon},
$l_1 \coloneqq l \land n$, and $l_2 \coloneqq l \lor n$ for $\abs{n - l} \leq n \epsilon_n$.
Arguing as in \eqref{eq:ppp.bpp.prob.decomposition} and using \eqref{eq:ppp.num.prob.bound},
we obtain, for all sufficiently large $n$,
\begin{align}
  &\probab[bigg]{\sum_{q=1}^{\infty} \frac{1}{2^q} \abs{\pairing{f_q}{\kappa_n(\ppp_n)} - \pairing{f_q}{\kappa_n(\bpp_n)}} > \delta \mdpfactor} \\
  &\quad \leq 4 \exp(- n \epsilon_n^2/3) \\
  &\qquad
  + \max_{l: \abs{n-l} \leq n\epsilon_n} \probab[bigg]{\frac{1}{2^q}\sum_{q=1}^{\infty} \abs[big]{\pairing{f_q}{\kappa_n(\bpp_{l_1})} - \pairing{f_q}{\kappa_n(\bpp_{l_2})}} > \delta \mdpfactor, \, \bpp_{l_1} \subset \ppp_{2n, l_1}', \, \bpp_{l_2} \setminus \bpp_{l_1} \subset \ppp_{2n\epsilon_n, l_1, l_2}'}.
\end{align}
As in \eqref{eq:tail.summable.1} for $\beta = 0$,
we obtain
\begin{align}
  &\probab[bigg]{\sum_{q=1}^{\infty} \frac{1}{2^q} \abs[big]{\pairing{f_q}{\kappa_n(\bpp_{l_1})} - \pairing{f_q}{\kappa_n(\bpp_{l_2})}} > \delta \mdpfactor, \, \bpp_{l_1} \subset \ppp_{2n, l_1}', \, \bpp_{l_2} \setminus \bpp_{l_1} \subset \ppp_{2n\epsilon_n, l_1, l_2}'} \\
  &\quad \leq \sum_{q=1}^{\infty} \probab[Big]{\abs[big]{\pairing{f_q}{\kappa_n(\bpp_{l_1})} - \pairing{f_q}{\kappa_n(\bpp_{l_2})}} > \frac{\delta \mdpfactor q}{16}, \, \bpp_{l_1} \subset \ppp_{2n, l_1}', \, \bpp_{l_2} \setminus \bpp_{l_1} \subset \ppp_{2n\epsilon_n, l_1, l_2}'}. \label{eq:lil.bppl1.bppl2.bound}
\end{align}
Furthermore, since $\mdpfactor / (\epsilon_{n}^{k-j} \ldpspeed) \to \infty$ for $j=0, \dots, k-1$
and $\mdpfactor / (\epsilon_{n} \nrnd\ldpspeed) \to \infty$,
applying \eqref{eq:tail.bppl1.bppl2} shows that,
for all sufficiently large $n$,
\eqref{eq:lil.bppl1.bppl2.bound} is bounded by
\begin{align}
  \sum_{q=1}^{\infty} k \exp\ab\Big(- C \frac{\delta \mdpfactor q}{16} )
  \leq k \frac{\exp(- C \delta \mdpfactor/16)}{1 - \exp(- C \delta \mdpfactor/16)}.
\end{align}
Note that $C_{f_q}'$ given by \eqref{eq:bpp.ppp.const.2} is uniformly bounded by $e^{m} \tau(E)$ for $q \geq 1$.
Consequently, we obtain
\begin{align}
  &\sum_{n=1}^{\infty} \probab[bigg]{\sum_{q=1}^{\infty} \frac{1}{2^q} \abs[big]{\pairing{f_q}{\kappa_n(\ppp_n)} - \pairing{f_q}{\kappa_n(\bpp_n)}} > \delta \mdpfactor} \\
  &\quad \leq 4 \sum_{n=1}^{\infty} \exp(- n \epsilon_n^2/3) + k\sum_{n=1}^{\infty} \frac{\exp(- C \delta \mdpfactor/16)}{1 - \exp(- C \delta \mdpfactor/16)}
  < \infty,
\end{align}
which completes the proof.
\end{proof}

\section*{Acknowledgments}
The authors are grateful to Professor Takashi Owada for his helpful explanations of \cite{ho2023ldp} and related topics.
YS was supported by the WISE program (MEXT) at Kyushu University.
KT was supported by JSPS KAKENHI Grant Number JP26K06832.

\section*{Declaration on the use of AI}
The authors used generative AI tools to assist with language editing and with checking notation and selected technical estimates.
All mathematical content was independently verified by the authors, who take full responsibility for the manuscript.

\appendix
\crefalias{section}{appendix} % "Appendix A"のように参照

\section{Technical Tools and Concentration Inequalities}
\label{sec:technical.tools}

In this appendix, we introduce concentration inequalities
and related estimates that were used in the proofs of our main results.

First, we recall Bousquet's concentration inequality for the supremum of sums of independent random vectors.

\begin{proposition}[Bousquet's inequality]
\label{prop:bousquet.ineq}
Let $\calT$ be a finite or countable index set,
let $X_1 = (X_{1, s})_{s\in \calT}, \dots, X_n = (X_{n, s})_{s\in \calT}$ be independent and identically distributed random vectors,
and let $Z = \sup_{s \in \mathcal{T}} \sum_{i=1}^n X_{i,s}$.
Assume that $\expect[none]{X_{i,s}} = 0$ and
that there exists $a > 0$ such that $X_{i,s} \le a$ for all $s \in \mathcal{T}$.
Let $v = 2a^{-1}\expect[none]{Z} + a^{-2}\sup_{s \in \mathcal{T}} \sum_{i=1}^n \expect[none]{X_{i,s}^2}$.
Then, for all $\lambda \ge 0$,
\begin{align}
  \log \expect[none]{e^{\lambda(Z - \expect[none]{Z})}} \le v (e^{a\lambda} - a\lambda - 1).
\end{align}
\end{proposition}

For the proof of \cref{prop:bousquet.ineq},
we refer the reader to Theorem 12.5 in \cite{blm2013concentrationineq}.

Since Bousquet's inequality is not applicable when the random variables $X_{i,s}$ are unbounded,
we use the following alternative inequality.
For a real-valued random variable $X$,
let $\orlicznorm{X}$ denote the $\psi_1$-Orlicz norm defined by
\begin{align}
  \orlicznorm{X} \coloneqq \inf
  \Bab{\lambda > 0 : \expect[Big]{\psi_1\ab\Big(\frac{\abs{X}}{\lambda})} \leq 1}, \label{eq:orlicz.norm.def}
\end{align}
where $\psi_1(x) \coloneqq e^x - 1$ for $x \geq 0$.

The following proposition provides a maximal inequality for empirical sums.
This result extends the concentration inequality by Adamczak \cite{adamczak2008tail}
to the maximum over an increasing sequence of index sets.
Furthermore,
while the original theorem is stated for empirical sums of the form $\sum f(X_i)$,
we reformulate it here using general random vectors $X_{i,s}$,
both to align with the framework of Bousquet's inequality and to facilitate its use in \cref{lemma:lil.upper.z4.ei}.

\begin{proposition}
\label{prop:adamczak.ineq}
Let $\calT$ be a finite or countable index set,
and let $X_1 = (X_{1, s})_{s\in \calT}, \dots, X_n = (X_{n, s})_{s\in \calT}$ be independent random vectors.
Let $\{I_j\}_{1 \leq j \leq J}$ be a sequence of finite index sets,
satisfying $I_1 \subset I_2 \subset \dots \subset I_J = \{1, \dots, n\}$.
Define $Z_j \coloneqq \sup_{s \in \calT} \sum_{i \in I_j} X_{i,s}$ for each $1 \leq j \leq J$,
and $\sigma^2 \coloneqq \sup_{s \in \calT} \sum_{i=1}^n \expect[none]{X_{i, s}^2}$.
Assume that $\expect[none]{X_{i, s}} = 0 $ and $\orlicznorm{\sup_{s\in \calT} \abs{X_{i, s}}} < \infty$ for all $i=1, \dots, n$.
Then,
for all $\eta \in (0, 1)$ and $\delta > 0$,
there exists a constant $C = C(\eta, \delta)$ depending only on $\eta$ and $\delta$,
such that for all $t \geq 0$,
\begin{align}
  \probab[Big]{\max_{1 \leq j \leq J} Z_j \geq (1 + \eta)\expect[none]{Z_J} + t}
  \leq \exp\ab\Big(- \frac{t^2}{2(1 + \delta)\sigma^2})
  + 3 \exp\ab\bigg(- \frac{t}{C(\eta, \delta) \norm[big]{\max\limits_{1 \leq i \leq n} \sup\limits_{s \in \calT} \abs{X_{i, s}}}_{\psi_1}}).
\end{align}
\end{proposition}

\begin{proof}
Let $\{\calG_j\}_{1 \leq j \leq J}$ be the filtration generated by the random vectors,
defined by $\calG_j \coloneqq \sigma(X_i : i \in I_j)$ with $\calG_0$ being the trivial $\sigma$-algebra.
Since the random vectors $X_1, \dots, X_n$ are independent and centered,
by an argument similar to that in Step 1 of the proof of \cref{lemma:lil.upper.z4},
we can easily verify that $\{Z_j\}_{1 \leq j \leq J}$ forms a submartingale,
with respect to the filtration $\{\calG_j\}_{1 \leq j \leq J}$.

The remainder of the proof essentially follows the argument of Theorem 4 in \cite{adamczak2008tail}.
Following the original proof,
we decompose the process into a bounded part (corresponding to $f_1$ in the original proof)
and an unbounded part (corresponding to $f_2$) via truncation.
For the bounded part,
instead of applying Markov's inequality to the exponential moment at the terminal time $J$,
we apply Doob's maximal inequality to the positive submartingale formed by its exponential,
and then bound the resulting terminal expectation using Lemma 1 based on Klein--Rio's inequality.
For the unbounded part,
we again apply Doob's maximal inequality to its exponential submartingale,
and evaluate the tail probability through the $\psi_1$-norm of its supremum,
corresponding to equation (11) in \cite{adamczak2008tail}.
Combining these upper bounds with the appropriate choice of parameters,
we obtain the desired maximal inequality.
\end{proof}

To effectively evaluate the $\psi_1$-norm appearing in the tail bound,
we summarize three fundamental properties.

\begin{proposition}
\label{prop:orliczc.norm}
\textup{(i)}
Let $X$ be a real-valued random variable with $\expect[none]{X} = 0$.
Assume that there exist constants $c > 0$ and $a \in (0, 1/4]$,
such that for all $\lambda \in \reals$,
\begin{align}
  \expect[none]{e^{\lambda X}} \leq 1 + a (e^{c\abs{\lambda}} - c\abs{\lambda} - 1). \label{eq:orlicz.moment.ass}
\end{align}
Then,
$\orlicznorm{X} \leq c$.

\textup{(ii)}
Let $X$ be a real-valued random variable with $\expect{X} = 0$.
Assume that there exist constants $c > 0$ and $a > 0$,
such that for all $\lambda \in \reals$ with $\abs{\lambda} < 1/c$,
\begin{align}
  \expect[none]{e^{\lambda X}} \leq 1 + a \frac{(c\lambda)^2}{1 - c\abs{\lambda}}.
\end{align}
Then, $\orlicznorm{X} \leq 2c(\sqrt{2a} + 1)$.

\textup{(iii)}
Let $X_1, \dots, X_N$ be real-valued random variables.
Assume that there exists $A > 0$ such that $\orlicznorm{X_i} \leq A$ for all $i=1, \dots, N$.
Then,
\begin{align}
  \orlicznorm{\max_{1 \leq i \leq N} \abs{X_i}} \leq (1 + \log_2 N)A. \label{eq:orlicz.norm.max}
\end{align}
\end{proposition}

\begin{proof}
\textup{(i)}
By comparing the second-order terms in the Taylor expansions of both sides as $\lambda \to 0$,
we obtain $\expect[none]{X^2} \leq a c^2$
and $\expect[none]{\abs{X}} \leq \sqrt{\expect[none]{X^2}} \leq c\sqrt{a}$.
For any $x \in \reals$,
the elementary inequality $e^{\abs{x}} - \abs{x} - 1 \leq (e^x - x - 1) + (e^{-x} + x - 1)$ holds.
Substituting $x = X/c$,
taking the expectation,
and using the assumption $\expect[none]{X} = 0$,
we obtain
\begin{align}
  \expect[Big]{\exp\ab\Big(\frac{\abs{X}}{c})} - \frac{1}{c}\expect[none]{\abs{X}} - 1
  \leq \ab(\expect[Big]{\exp\ab\Big(\frac{X}{c})} - 1) + \ab(\expect[Big]{\exp\ab\Big(-\frac{X}{c})} - 1).
\end{align}
Applying \eqref{eq:orlicz.moment.ass} with $\lambda = 1/c$ and $\lambda = -1/c$,
we can bound each term on the right-hand side by $a(e^1 - 1 - 1) = a(e - 2)$.
Using $\expect[none]{\abs{X}} \leq c\sqrt{a}$,
we get
\begin{align}
  \expect[Big]{\exp\ab\Big(\frac{\abs{X}}{c})}
  \leq 1 + \frac{1}{c}\expect[none]{\abs{X}} + 2a(e - 2)
  \leq 1 + \sqrt{a} + 2a(e - 2).
\end{align}
Since $1 + \sqrt{a} + 2a(e - 2) \leq 1 + 1/2 + (e - 2)/2 < 2$ for $a \in (0, 1/4] $,
by the definition of the Orlicz norm,
we conclude $\orlicznorm{X} \leq c$.

\textup{(ii)}
By the same argument as in the proof of (i),
we obtain $\expect{X^2} \leq 2ac^2$
and $\expect{\abs{X}} \leq \sqrt{\expect{X^2}} \leq c\sqrt{2a}$.
Let $u \in (0, 1)$ and $K_u \coloneqq c/u > c$.
Using $e^{\abs{x}} - \abs{x} - 1 \leq (e^x - x - 1) + (e^{-x} + x - 1)$ with $x = X/K_u$,
we obtain
\begin{align}
  \expect[Big]{\exp\ab\Big(\frac{\abs{X}}{K_u})}
  &\leq 1 + \frac{1}{K_u}\expect{\abs{X}} + \ab(\expect[Big]{\exp\ab\Big(\frac{X}{K_u})} - 1) + \ab(\expect[Big]{\exp\ab\Big(-\frac{X}{K_u})} - 1) \nonumber \\
  &\leq 1 + u\sqrt{2a} + 2a \frac{u^2}{1 - u}.
\end{align}
Choosing $u =(2\sqrt{2a} + 2)^{-1} \in (0, 1)$,
we obtain by a direct calculation
\begin{align}
  u\sqrt{2a} + 2a \frac{u^2}{1 - u}
  = \frac{\sqrt{2a}}{2\sqrt{2a} + 2} + \frac{2a}{(2\sqrt{2a} + 2)(2\sqrt{2a} + 1)}
  = \frac{6a + \sqrt{2a}}{8a + 6\sqrt{2a} + 2} < 1,
\end{align}
which implies $\expect{\exp(\abs{X}/K_u)} \leq 2$.
Thus,
% by the definition of the Orlicz norm,
we conclude $\orlicznorm{X} \leq K_u = c/u = 2c(\sqrt{2a} + 1)$.

\textup{(iii)}
Let $Y \coloneqq \max_{1 \leq i \leq N} \abs{X_i}$.
By the assumption $\orlicznorm{X_i} \leq A$,
we have $\expect[none]{\exp(\abs{X_i}/A)} \leq 2$ for all $i = 1, \dots, N$.
We have the bound
\begin{align}
  \expect[Big]{\exp\ab\Big(\frac{Y}{A})}
  = \expect[Big]{\max_{1 \leq i \leq N} \exp\ab\Big(\frac{\abs{X_i}}{A})}
  \leq \expect[bigg]{\sum_{i=1}^N \exp\ab\Big(\frac{\abs{X_i}}{A})}
  \leq 2N,
\end{align}
and for any $k \geq 1$, applying Jensen's inequality,
we obtain
\begin{align}
  \expect[Big]{\exp\ab\Big(\frac{Y}{kA})}
  \leq \expect[Big]{\exp\ab\Big(\frac{Y}{A})}^{1/k}
  \leq (2N)^{1/k}.
\end{align}
By choosing $k = 1 + \log_2 N \geq 1$,
we obtain $\expect[none]{\exp(Y/(kA))} \leq (2N)^{1/k} = 2$,
which yields \eqref{eq:orlicz.norm.max}.
\end{proof}

Next, we use a chaining argument to bound the expected maximum of a collection of random variables. Unlike standard sub-gamma chaining bounds, our setting requires a scale parameter uniformly bounded independently of the distance, as in (A.4) below.

To state the corresponding chaining inequality used in our proof,
we first recall the packing number of a metric space $(\calT, d)$.
For any $\delta > 0$, a subset $\calT_\delta \subset \calT$ is called a \textit{$\delta$-net} if it has maximal cardinality among subsets satisfying $d(s, t) > \delta$ for all distinct elements $s, t \in \calT_\delta$.
Its cardinality, denoted by $N(\delta, \calT)$, is defined as the \textit{$\delta$-packing number} of $\calT$. 

The following proposition provides a modified version of the chaining inequality
tailored to random variables satisfying this bound on pairwise differences.

\begin{proposition}
\label{prop:chaining.modify}
Let $(\calT, d)$ be a finite metric space and let $(Z_t)_{t \in \calT}$ be a collection of random variables such that for some constants $v, c > 0$,
\begin{align}
  \log \expect[none]{e^{\lambda(Z_t - Z_{t'})}}
  \leq \frac{v d(t, t')^2 \lambda^2}{2(1 - c\lambda)} \label{eq:chaining.exp.bound}
\end{align}
for all $t, t' \in \calT$ and all $0 < \lambda < c^{-1}$.
Then, there exists a positive integer $J$ such that
for any $t_0 \in \calT$,
\begin{align}
  \expect[bigg]{\max_{t \in \calT} Z_t - Z_{t_0}}
  \leq 12 \sqrt{v}\int_{0}^{\delta/2}\sqrt{\log N(u, \calT)} \,du
  + 2c \sum_{j=1}^{J} \log N(2^{-j}\delta, \calT),
\end{align}
where $\delta \coloneqq  \max_{t \in \calT} d(t, t_0)$.
Furthermore, $J$ satisfies
\begin{align}
  J \leq \Roundup{\log_2 \ab\big(\delta {\delta'}^{-1})} \lor 1, \label{eq:chaining.j.bound}
\end{align}
where $\delta' \coloneqq \min\{d(t, t'): t, t' \in \calT, t\neq t'\}  $.
\end{proposition}

\begin{proof}
Our proof is a slight modification of the argument presented in \cite{blm2013concentrationineq}.
For each integer $j$, let $\delta_j \coloneqq 2^{-j} \delta$ and let $\calT_j$ be a $\delta_j$-net of $\calT$.
By the properties of a $\delta_j$-net, we can construct a projection $\map{\Pi_j}{\calT}{\calT_j}$ for each $j$ that satisfies
\begin{align}
  d(t, \Pi_j(t)) \leq \delta_j \quad \text{ for all } t \in \calT.
\end{align}
When $j$ satisfies $\delta_j \leq \delta'$, the projection reduces to the identity map, i.e., $\Pi_j(t)= t$ for all $t \in \calT$.
Thus, there exists a positive integer $J$ satisfying \eqref{eq:chaining.j.bound} such that the following equation holds for any $t \in \calT$:
\begin{align}
  Z_t = Z_{\Pi_0(t)} + \sum_{j=0}^{J-1} (Z_{\Pi_{j+1}(t)} -  Z_{\Pi_j(t)}).
\end{align}
Without loss of generality, we can choose $\calT_0 = \{t_0\}$ based on the definition of $\delta$, implying $\Pi_{0}(t) = t_0$. It follows from the linearity of expectation that
\begin{align}
  \expect[bigg]{\max_{t \in \calT} Z_t - Z_{t_0}}
  \leq \sum_{j=0}^{J-1} \expect[bigg]{\max_{t \in \calT} \ab\big(Z_{\Pi_{j+1}(t)} -  Z_{\Pi_j(t)})}.
\end{align}
Note that for each integer $j$, the number of possible pairs is bounded by
\begin{align}
  \card{\Bab[big]{(\Pi_j(t), \Pi_{j+1}(t)) : t \in \calT}}
  % \leq e^{2 \log N(\delta_{j+1}, \calT)}.
  \leq N(\delta_{j+1}, \calT)^2.
\end{align}
Furthermore, the triangle inequality yields
\begin{align}
  d(\Pi_j(t), \Pi_{j+1}(t)) \leq d(\Pi_j(t), t) + d(t, \Pi_{j+1}(t)) \leq \delta_j + \delta_{j+1} = 3\delta_{j+1}
\end{align}
for all $t \in \calT$, which implies
\begin{align}
  \log \expect[big]{\exp\ab(\lambda(Z_{\Pi_{j+1}(t)} -  Z_{\Pi_j(t)}))}
  \leq \frac{9v \delta_{j+1}^2 \lambda^2 }{2(1 - c\lambda)}.
\end{align}
Applying the maximal inequality from Corollary 2.6 of \cite{blm2013concentrationineq}, we obtain
\begin{align}
  \expect[bigg]{\max_{t \in \calT} \ab\big(Z_{\Pi_{j+1}(t)} -  Z_{\Pi_j(t)})}
  \leq 6 \delta_{j+1} \sqrt{v \log N(\delta_{j+1}, \calT)} + 2c \log N(\delta_{j+1}, \calT).
\end{align}
Finally, summing these bounds over $j$ yields
\begin{align}
  \expect[bigg]{\max_{t \in \calT} Z_t - Z_{t_0}} 
  &\leq \sum_{j=1}^{J} \ab( 6 \delta_{j}  \sqrt{v \log N(\delta_{j}, \calT)} + 2c\log N(\delta_{j}, \calT)) \\
  &\leq 12 \sqrt{v}\int_{0}^{\delta/2}\sqrt{\log N(u, \calT)} \,du
  + 2c \sum_{j=1}^{J} \log N(2^{-j}\delta, \calT),
\end{align}
where the last inequality follows from the fact that $\log N(u, \calT)$ is nonincreasing in $u$.
\end{proof}

\section{Proof of \cref{cor:mdp.morse.vector,cor:lil.morse.vector}}
\label{sec:morse.proof}

In this appendix, we prove \cref{cor:mdp.morse.vector,cor:lil.morse.vector}.

\begin{proof}
\cref{cor:mdp.morse.vector} follows from the same calculations
as those used in \cref{prop:ge.eta}, \cref{prop:mdp.eta} and \cref{prop:exp.eq.eta.xi}.
Note that the boundedness of $\expect[none]{\exp(\abs{W_{j, z}})}$ in \cref{prop:exp.eq.eta.xi} as $n \to \infty$
is guaranteed by (5.53) in \cite{ho2023ldp}.

We now proceed to prove \cref{cor:lil.morse.vector}.
Here, we only show the counterpart of \eqref{eq:lil.upper.z4.ei.tail.max},
which is given by
\begin{align}
  \lim_{n \to \infty} \distancevec(\mdpfactor^{-1}\centered{N_n}, \accptvec) = 0.
\end{align}
In particular, we focus solely on establishing the inequality,
which corresponds to \cref{lemma:lil.upper.z4.ei}.
That is, we aim to show that for all $s \in (1, 2^{2(3 - \alpha)/(\alpha - 1)})$,
there exists a constant $C > 0$, independent of $s$, $l$, and $i$,
such that for all sufficiently large $l$ and all $R > C (n_{\lilsubseq{s}{l}}^3 r_{\lilsubseq{s}{l}}^4)^{1/2}$,
\begin{align}
  \probab[Big]{\max_{\lilsubseq{s}{l} \leq n \leq \lilsubseq{s}{l+1}} Z_{n}^{(4)} > R}
  \leq C\exp\ab\bigg(- \frac{R^2}{C (1 - v_{\lilsubseq{s}{l+1}}/v_{\lilsubseq{s}{l}}) n_{\lilsubseq{s}{l}}^3 r_{\lilsubseq{s}{l}}^4 })
  + C\exp\ab\bigg(- \frac{R}{Cl}), \label{eq:morse.lil.upper.z4}
\end{align}

The variables involved in this expression are defined analogously to those in \cref{sec:lil.upper}:
\begin{gather}
  S_{z}^{(v)}(\calX) \coloneqq \sum_{\calZ \subset \calX \cap Q_{z}^{(2Lv_{\lilsubseq{s}{l}})}} c(\calZ, \calX \cap Q_{z}^{(3Lv_{\lilsubseq{s}{l}})}) h^{(i)}(v^{-1} \calZ), \\
  \Delta_{\infty,z}^{(v)} \coloneqq S_{z}^{(v)}(\pppentire) - S_{z}^{(v)}(\pppentire_{z}'),
  \quad
  Z_{n}^{(4)} \coloneqq \sum_{z \in I_{n}} \condexpect[Big]{\Delta_{\infty,z}^{(v_{n})} - \Delta_{\infty,z}^{(v_{\lilsubseq{s}{l}})}}{\calF_z}.
\end{gather}
where $\pppentire_z' $ is given by \eqref{eq:ppp.indep},
$h^{(i)} $ is given by \eqref{eq:h.morse.critical} for $i=1, \dots, m$,
and $c(\calY, \calX) $ is defined in \eqref{eq:indicator.morse}.
Without loss of generality, we may assume that $i=1 $.
Define $Y_{n, z} $ analogously to \eqref{eq:lil.upper.z4.y} by $ Y_{n, z} \coloneqq \condexpect[Big]{\Delta_{\infty,z}^{(v_{n})} - \Delta_{\infty,z}^{(v_{\lilsubseq{s}{l}})}}{\calF_z}$.

Then, for any $\lilsubseq{s}{l} \leq n_1 < n_2 \leq \lilsubseq{s}{l+1} $ and $\lambda \in \reals$,
we show that
\begin{align}
  \expect{\exp\ab\Big(\lambda (Y_{n_1, z} - Y_{n_2, z}))}
  \leq 1 + v_{\lilsubseq{s}{l}}^6 \ab\big(1 - v_{n_2}/v_{n_1})
   \sum_{p=2}^{\infty} (C \abs{\lambda})^p. \label{eq:morse.y.exp.moment}
\end{align} 
Following an argument similar to that in \cref{prop:exp.moment.common},
we obtain
\begin{align}
  \expect{\exp\ab\Big(\lambda (Y_{n_1, z} - Y_{n_2, z}))}
  \leq 1 + \sum_{p=2}^{\infty} \frac{(C \abs{\lambda})^p}{p!}
   \expect{\ab(S_{z}^{(v_{n_1})}(\pppentire) - S_{z}^{(v_{n_2})}(\pppentire) )^p}.
\end{align}
Let $\calZ_1, \dots, \calZ_M $ denote the subsets $\calZ \subset \Pi \cap Q_{z}^{(3Lv_{\lilsubseq{s}{l}})} $
which satisfy $\card{\calZ} = 3 $ and $c(\calZ, \pppentire \cap Q_{z}^{(3Lv_{\lilsubseq{s}{l}})}) = 1 $.
Using this notation, we obtain the bound
\begin{align}
  \expect{\ab(S_{z}^{(v_{n_2})}(\pppentire) - S_{z}^{(v_{n_1})}(\pppentire) )^p}
  &\leq \expect[bigg]{\ab\bigg(\sum_{j=1}^M \ab\big(h^{(1)}(v_{n_1}^{-1}\calZ_j) - h^{(1)}(v_{n_2}^{-1} \calZ_j)) )^p} \\
  &\leq \expect[bigg]{M^{p-1} \sum_{j=1}^M \ab\big(h^{(1)}(v_{n_1}^{-1}\calZ_j) - h^{(1)}(v_{n_2}^{-1} \calZ_j))^p}. \label{eq:morse.s.m}
\end{align}
Since $Z_1, \dots, Z_M $ constitute Voronoi points in the plane,
and the number of Voronoi points associated with a set of $N$ vertices in $\reals^2$ is at most $2N $
(see Lemma 1 in \cite{ls1980delaunay}, as well as p.4041 of \cite{ho2023ldp}),
it follows that $M \leq 2 \card{\pppentire \cap Q_{z}^{(3Lv_{\lilsubseq{s}{l}})}}$.
Therefore, applying the Mecke formula, we can bound \eqref{eq:morse.s.m} by
\begin{align}
  &\expect[bigg]{2^{p-1} \card{\pppentire \cap Q_{z}^{(3Lv_{\lilsubseq{s}{l}})}}^{p-1} \sum_{\calZ \subset \pppentire \cap Q_{z}^{(3Lv_{\lilsubseq{s}{l}})}} \ab\big(h^{(1)}(v_{n_1}^{-1}\calZ) - h^{(1)}(v_{n_2}^{-1} \calZ))^p} \\
  &\quad = 2^{p-1} \frac{1}{6} \int_{(Q_{z}^{(3Lv_{\lilsubseq{s}{l}})})^3}  \ab\big(h^{(1)}(v_{n_1}^{-1}\bfx) - h^{(1)}(v_{n_2}^{-1} \bfx))^p \expect[bigg]{\card{(\bfx \cup \pppentire) \cap Q_{z}^{(3Lv_{\lilsubseq{s}{l}})}}^{p-1}} \,d\bfx  \\
  &\quad \leq C^{p} (v_{\lilsubseq{s}{l}}^2)^3 \ab\big(1 - v_{n_2}/v_{n_1}) \expect[bigg]{\ab\Big(\card{\pppentire \cap Q_{z}^{(3Lv_{\lilsubseq{s}{l}})}} + 3)^{p}}.
\end{align}
Consequently,
\begin{align}
  \expect{\exp\ab\Big(\lambda (Y_{n_1, z} - Y_{n_2, z}))}
  &\leq 1 + v_{\lilsubseq{s}{l}}^6 \ab\big(1 - v_{n_2}/v_{n_1}) \sum_{p=2}^{\infty} \frac{(C \abs{\lambda})^p}{p!}
   \expect[bigg]{\ab\Big(\card{\pppentire \cap Q_{z}^{(3Lv_{\lilsubseq{s}{l}})}} + 3)^{p}}\\
  &\leq 1 + v_{\lilsubseq{s}{l}}^6 \ab\big(1 - v_{n_2}/v_{n_1}) \sum_{p=2}^{\infty} (C \abs{\lambda})^p
   \expect[bigg]{\exp\ab\Big(\card{\pppentire \cap Q_{z}^{(3Lv_{\lilsubseq{s}{l}})}} + 3)},
\end{align}
and with
\begin{align}
  \expect[bigg]{\exp\ab\Big(\card{\pppentire \cap Q_{z}^{(3Lv_{\lilsubseq{s}{l}})}} + 3)}
  = e^3 \exp\ab\Big((e-1)\lebmeas\ab\big(Q_{z}^{(3Lv_{\lilsubseq{s}{l}})}))
  = \lorder(1)
  \quad \text{as } l \to \infty,
\end{align}
we obtain \eqref{eq:morse.y.exp.moment}.

Thus, arguing as in the proof of \cref{lemma:lil.upper.z4.ei}
and applying \cref{prop:orliczc.norm} (ii) and \cref{prop:adamczak.ineq},
we obtain the desired inequality \eqref{eq:morse.lil.upper.z4}.
This completes the proof.
\end{proof}

\bibliographystyle{abbrv}
\bibliography{ref}

\end{document}